\documentclass[a4paper,UKenglish,cleveref, autoref, thm-restate]{lipics-v2021}
\usepackage{algpseudocode}
\usepackage{amsmath}
\usepackage{amsthm}
\usepackage{amssymb}
\usepackage{caption}
\usepackage{float}
\usepackage{thmtools}
\usepackage{enumerate}
\usepackage{graphicx}
\usepackage{lineno}
\usepackage{microtype}
\usepackage[numbers]{natbib}
\usepackage[parfill]{parskip}
\newcommand{\classt}[2]{\mathcal{C}_{}(#2)}
\newcommand{\classtt}[2]{\mathcal{C}^*(#2)}

\newcommand{\edget}[3]{\mathcal{E}(#2, #3)}
\newcommand{\edgett}[3]{\mathcal{E}^*(#2, #3)}

\newcommand{\bdryt}[3]{\mathcal{B}_{#3}(#2)}
\newcommand{\bdrytt}[3]{\mathcal{B}^*_{#1,#3}(#2)}

\renewenvironment{definition}{}{}

\pdfoutput=1 
\hideLIPIcs  

\title{Improved mixing for the convex polygon triangulation flip walk\footnote{This paper subsumes a previous version of the same preprint, as well as parts of another, namely~\cite{eppfrisharx}.}} 

\author{David Eppstein}{Department of Computer Science, University of California, Irvine, United States}{eppstein@uci.edu}{}{}
\author{Daniel Frishberg}{Department of Computer Science, University of California, Irvine, United States}{dfrishbe@uci.edu}{https://orcid.org/0000-0002-1861-5439}{}
\authorrunning{D. Eppstein and D. Frishberg} 

\Copyright{Anonymous} 



\category{} 

\relatedversion{} 

\acknowledgements{The authors wish to acknowledge a number of helpful conversations on this topic with Hadi Khodabandeh, Milena Mihail, Ioannis Panageas, Eric Vigoda, Charlie Carlson, Prasad Tetali, Vedat Alev, Michail Sarantis, Zongchen Chen, Alexandre Stauffer, Karthik Gajulapalli, and Pedro Matias.}

\nolinenumbers 

\EventEditors{John Q. Open and Joan R. Access}
\EventNoEds{2}
\EventLongTitle{42nd Conference on Very Important Topics (CVIT 2016)}
\EventShortTitle{CVIT 2016}
\EventAcronym{CVIT}
\EventYear{2016}
\EventDate{December 24--27, 2016}
\EventLocation{Little Whinging, United Kingdom}
\EventLogo{}
\SeriesVolume{42}
\ArticleNo{23}

\begin{document}

\maketitle

\begin{abstract}
We prove that the well-studied \emph{triangulation flip walk} on a convex point set mixes in time $O(n^{3}\log^3 n)$, the first progress since McShine and Tetali's $O(n^5 \log n)$ bound in 1997. In the process we give lower and upper bounds of respectively~$\Omega(1/(\sqrt n\log n))$ and~$O(1/\sqrt n)$\textemdash asymptotically tight up to an~$O(\log n)$ factor\textemdash for the \emph{expansion} of the \emph{associahedron} graph $K_n$. The upper bound recovers Molloy, Reed, and Steiger's~$\Omega(n^{3/2})$ bound on the mixing time of the walk. To obtain these results, we introduce a framework consisting of a set of sufficient conditions under which a given Markov chain mixes rapidly. This framework is a purely combinatorial analogue that in some circumstances gives better results than the \emph{projection-restriction} technique of Jerrum, Son, Tetali, and Vigoda. In particular, in addition to the result for triangulations, we show quasipolynomial mixing for the \emph{$k$-angulation} flip walk on a convex point set, for fixed $k \geq 4$.
\end{abstract}

\section{Introduction and background}
\label{sec:associntro}
The study of~\emph{mixing times}\textemdash the art and science of proving upper and lower bounds on the efficiency of Markov chain Monte Carlo sampling methods\textemdash is a well-established area of research, of interest for combinatorial sampling problems, spin systems in statistical physics, probability, and the study of subset systems. Work in this area brings together techniques from spectral graph theory, combinatorics, and probability, and dates back decades; for a comprehensive survey of classic methods, results, and open questions see the canonical text by Levin, Wilmer, and Peres~\cite{levin2017markov}. Recent breakthroughs~\cite{anarimixis, anari2, anari1, chen2020rapid, chen2020rapiduniq, chen2021optimal, kaufman2020high, levi2020improved}\textemdash incorporating techniques from the theory of abstract simplicial complexes\textemdash have led to a recent slew of results for the mixing times of combinatorial chains for sampling independent sets, matchings, Ising model configurations, and a number of other structures in graphs, injecting renewed energy into an already active area.

We focus on a class of \emph{geometric} sampling problems that has received considerable attention from the counting and sampling~\cite{anclin, kzlimit} and mixing time~\cite{mct, molloylb, Stauffer2015ALF, sinclairlambda} research communities over the last few decades, but for which tight bounds have been elusive: sampling \emph{triangulations}. A triangulation is a maximal set of non-crossing edges connecting pairs of points (see Figure~\ref{fig:octtri}) in a given~$n$-point set. Every pair of triangles sharing an edge forms a quadrilateral. A triangulation~\emph{flip} consists of removing such an edge, and replacing it with the only other possible diagonal within the same quadrilateral. Flips give a natural Markov chain (the~\emph{flip walk}): one selects a uniformly random diagonal from a given triangulation and (if possible) flips the diagonal.

McShine and Tetali gave a classic result in a 1997 paper~\cite{mct}, showing that in the special case of a convex two-dimensional point set (a convex $n$-gon), the flip walk \emph{mixes} (converges to approximately uniform) in time~$O(n^5 \log n)$, improving on the best-known prior (and first polynomial) upper bound,~$O(n^{25})$, by Molloy, Reed, and Steiger~\cite{molloylb}. McShine and Tetali applied a Markov chain comparison technique due to Diaconis and Saloff-Coste~\cite{diaconis1993comparison} and to Randall and Tetali~\cite{randall1998analyzing} to obtain their bound, using a bijection between triangulations and a structure known as \emph{Dyck paths}. They noted that they could not improve on this bound using this bijection. Furthermore, they believed that an earlier \emph{lower} bound of~$\Omega(n^{3/2})$, also by Molloy, Reed, and Steiger~\cite{molloylb}, should be tight. We show the following result (see Section~\ref{sec:prelim} for the precise definition of mixing time):
\begin{theorem}
\label{thm:triangmixub}
The triangulation flip walk on the convex $n + 2$-point set mixes in time $O(n^{3}\log^3 n).$
\end{theorem}

Prior to the present paper, no progress had been made either on upper or lower bounds for this chain in 25 years\textemdash even as new polynomial upper bounds and exponential lower bounds were given for other geometric chains, from lattice point set triangulations~\cite{Stauffer2015ALF, sinclairlambda} to quadrangulations of planar maps~\cite{caraceni2020}, and despite many breakthroughs using the newer techniques for other problems.

In addition to this specific result, we give a general decomposition theorem\textemdash which we will state as Theorem~\ref{thm:flowprojres} once we have built up enough preliminaries, for bounding mixing times by recursively decomposing the state space of a Markov chain. This theorem is a purely combinatorial alternative to the spectral result of Jerrum, Son, Tetali, and Vigoda~\cite{jerrumprojres}.

\subsection{Decomposition framework}
To prove our result, we develop a general decomposition framework that applies to a broad class of Markov chains, as an alternative to prior work by Jerrum, Son, Tetali, and Vigoda~\cite{jerrumprojres} that used spectral methods. We obtain our new mixing result for triangulations, then generalize our technique to obtain the first nontrivial mixing result for~$k$-angulations. In a companion paper~\cite{eppfrisharx} we further generalize this work to obtain the first rapid mixing bounds for Markov chains for sampling independent sets, dominating sets, and \emph{$b$-edge covers} (generalizing edge covers) in graphs of bounded treewidth, and for maximal independent sets, $b$-matchings, and \emph{maximal $b$-matchings} in graphs of bounded treewidth and degree. In that work we also strengthen existing results~\cite{heinrich2020glauber, dyergj} for proper $q$-colorings in graphs of bounded treewidth and degree.

The key observation that unifies these chains is that, when viewing their state spaces as graphs (exponentially large graphs relative to the input), they all admit a recursive decomposition satisfying key properties. First, each such graph, called a ``flip graph,'' can be partitioned into a small number of induced subgraphs, where each subgraph is a \emph{Cartesian product} of smaller graphs that are structurally similar to the original graph\textemdash and thus can be partitioned again into even smaller product graphs. Second, at each level of recursion, pairs of subgraphs are connected by large matchings. Intuitively, we can ``slice'' a flip graph into subgraphs that are well connected to each other, then ``peel'' apart the subgraphs using their Cartesian product structure, and repeat the process recursively. Each recursive level of slicing cuts through many edges (the large matchings), and indeed the peeling also disconnects many mutually well-connected subgraphs from one another. Prior work exists applying this ``slicing'' and ``peeling'' paradigm\textemdash albeit with spectral methods instead of purely combinatorial methods\textemdash using Jerrum, Son, Tetali, and Vigoda's decomposition theorem (Theorem~\ref{thm:specprojres}) for combinatorial chains~\cite{jerrumprojres, heinrich2020glauber, dyergj}. One of our contributions is to unify these applications, along with the geometric chains, into a sufficient set of conditions under which one can apply the existing decomposition theorem: Lemma~\ref{lem:fw}.

A more substantial technical contribution is our Theorem~\ref{thm:flowprojres}, a combinatorial analogue to Jerrum, Son, Tetali, and Vigoda's Theorem~\ref{thm:specprojres}. One can use our theorem in place of theirs and, in some cases, obtain better mixing bounds. In particular, in the case of triangulations, we obtain polynomial mixing via an adaptation of our (combinatorial) technique (Lemma~\ref{lem:fwstrong})\textemdash and it is not clear how to adapt the existing spectral methods to get even a polynomial bound. In the case of~$k$-angulations, our theorem gives a bound that has better dependence on the parameter~$k$.

\subsection{Paper organization}
In the remainder of this section we will define the Markov chains we are analyzing and summarize our main results. Then, in Section~\ref{sec:sliceintuition}, we will give intuition for the decomposition by describing its application to triangulations. In Section~\ref{sec:framework} we will present our general decomposition meta-theorems, and compare our contribution to prior work by Jerrum, Son, Tetali, and Vigoda~\cite{jerrumprojres}. In particular, we will discuss why our purely combinatorial machinery is needed for obtaining new bounds in the case of triangulations. In Appendix~\ref{sec:flowdetails} we will prove a general result that gives a coarse bound on triangulation mixing; we will then improve this bound to near tightness in Appendix~\ref{sec:combinedec}, and give a matching upper bound (up to logarithmic factors) in Appendix~\ref{sec:assocexpub}. In Appendix~\ref{sec:quadmix}, we show that general $k$-angulations admit a decomposition satisfying a relaxation (Lemma~\ref{lem:fwquasi}) of our general theorem that implies quasipolynomial-time mixing. We analyze the particular quasipolynomial bound we obtain, and show that our combinatorial technique (Theorem~\ref{thm:flowprojres}) gives a better dependence on~$k$ than one would obtain with the prior decomposition theorem. In Appendix~\ref{sec:fwpf} we prove our general combinatorial decomposition theorem, Theorem~\ref{thm:flowprojres}. In Appendix~\ref{sec:latticetri} we prove a theorem about lattice triangulations; in Appendix~\ref{sec:missingproofs} we fill in a few remaining proof details.

\subsection{Triangulations of convex point sets and lattice point sets}
Let $P_n$ be the regular polygon with $n$ vertices. Every triangulation $t$ of $P_{n+2}$ has $n - 1$ diagonals, and every diagonal can be \emph{flipped}: every diagonal $D$ belongs to two triangles forming a convex quadrilateral, so $D$ can be removed and replaced with the diagonal $D'$ lying in the same quadrilateral and crossing $D$. The set of all triangulations of $P_{n+2}$, for $n \geq 1$, is the vertex set of a graph that we denote~$K_n$ (this notation is standard), whose edges are the flips between adjacent triangulations. The graph~$K_n$ is known to be realizable as the 1-skeleton of an $n-1$-dimensional polytope \cite{loday} called the \emph{associahedron} (we also use this name for the graph itself). It is also known to be isomorphic to the rotation graph on the set of all binary plane trees with $n+1$ leaves~\cite{sleator1988rotation}, and equivalently the set of all parenthesizations of an algebraic expression with $n+1$ terms, with ``flips'' defined as applications of the associative property of multiplication.

The structure of this graph depends only on the convexity and the number of vertices of the polygon, and not on its precise geometry. That is, $P_{n+2}$ need not be regular for $K_{n}$ to be well defined.


McShine and Tetali \cite{mct} showed that the \emph{mixing time} (see Section~\ref{sec:prelim}) of the uniform random walk on $K_{3,n+2}$ is $O(n^5 \log n)$, following Molloy, Reed, and Steiger's~\cite{molloylb} lower bound of $\Omega(n^{3/2})$. These bounds together can be shown, using standard inequalities \cite{sinclair_1992}, to imply that the \emph{expansion} of $K_{3,n+2}$ is $\Omega(1/(n^4 \log n))$ and $O(n^{1/4}).$  It is easy to generalize triangulations to \emph{$k$-angulations} of a convex polygon $P_{(k-2)n+2}$, and to generalize the definition of a flip between triangulations to a flip between $k$-angulations: a $k$-angulation is a maximal division of the polygon into $k$-gons, and a flip consists of taking a pair of $k$-gons that share a diagonal, removing that diagonal, and replacing it with one of the other diagonals in the resulting $2k-2$-gon. One can then define the \emph{$k$-angulation flip walk} on the $k$-angulations of $P_{(k-2)n+2}$. An analogous graph to the associahedron is defined over the triangulations of the integer lattice (grid) point set with $n$ rows of points and $n$ columns. Substantial prior work has been done on bounds for the number of triangulations in this graph (\cite{anclin}, \cite{kzlimit}), as well as characterizing the mixing time of random walks on the graph, when the walks are weighted by a function of the lengths of the edges in a triangulation (\cite{sinclairlambda} \cite{phaseshift}).

\subsection{Convex triangulation flip walk and mixing time}
\label{sec:mixingexp}
Consider the following random walk on the triangulations of the convex $n+2$-gon:

\begin{algorithmic}
\For{$t = 1, 2, \dots$}
\State{Begin with an arbitrary triangulation $t$.}
\State{Flip a fair coin.}
\State{If the result is tails, do nothing.}
\State{Else, select a diagonal in $t$ uniformly at random, and flip the diagonal.}
\EndFor
\end{algorithmic}

(The ``do nothing'' step is a standard MCMC step that enforces a technical condition known as \emph{laziness}, required for the arguments that bound mixing time.) At any given time step, this walk induces a probability distribution $\pi$ over the triangulations of the $n+2$-gon. Standard spectral graph theory shows that $\pi$ converges to the uniform distribution in the limit. Formally, what McShine and Tetali showed~\cite{mct} is that the number of steps before $\pi$ is within \emph{total variation distance}~$1/4$ of the uniform distribution is bounded by $O(n^5\log n)$\textemdash in other words, that the \emph{mixing time} is~$O(n^5 \log n)$. Any polynomial bound means the walk \emph{mixes rapidly}.
We formally define total variation distance:

\begin{definition}
\label{def:tvd}
The \emph{total variation distance} between two probability distributions $\mu$ and $\nu$ over the same set $\Omega$ is defined as
$$d(\mu, \nu) = \frac{1}{2}\sum_{S \in \Omega}|\pi(S) - \pi^*(S)|.$$
\end{definition}

Consider a Markov chain with state space $\Omega$. Given a starting state $S \in \Omega$, the chain induces a probability distribution $\pi_t$ at each time step $t$. Under certain mild conditions, all of which are satisfied by the $k$-angulation flip walk, this distribution is known to converge in the limit to a \emph{stationary} distribution $\pi^*,$ which for the $k$-angulation flip walk is the uniform distribution on the $k$-angulations of the convex polygon. The \emph{mixing time} is defined as follows:
\begin{definition}
\label{def:mixingtime}
Given an arbitrary $\varepsilon > 0$, the \emph{mixing time}, $\tau(\varepsilon)$, of a Markov chain with state space $\Omega$ and stationary distribution $\pi^*$ is the minimum time $t$ such that, regardless of starting state, we always have
$$d(\pi_t, \pi^*) < \varepsilon.$$
Suppose that the chain belongs to a family of chains, whose size is parameterized by a value $n$. (It may be that $\Omega$ is exponential in $n$.) If $\tau(\varepsilon)$ is upper bounded by a function that is polynomial in $\log (1/\varepsilon)$ and in $n$, say that the chain is \emph{rapidly mixing}.
\end{definition} It is common to omit the parameter $\varepsilon$, assuming its value to be the arbitrary constant 1/4.

\subsection{Main results}
We show the following result for the expansion of the associahedron:
\begin{theorem}
\label{thm:assocexplb}
\label{thm:assocexpub}
The expansion of the associahedron $K_{3,n+2}$ is $\Omega(1/(\sqrt{n}\log n))$ and~$O(1/\sqrt n)$.
\end{theorem}
We will prove the lower bound in Appendix~\ref{sec:flowdetails} and Appendix~\ref{sec:combinedec} using the \emph{multicommodity flow}-based machinery we introduce in Section~\ref{sec:framework}, after giving intuition in Section~\ref{sec:sliceintuition}. Combining this result with the connection between flows and mixing~\cite{sinclair_1992}\textemdash with some additional effort in Appendix~\ref{sec:combinedec}\textemdash gives our new~$O(n^{3} \log^3 n)$ bound (Theorem~\ref{thm:triangmixub}) for triangulation mixing.

Although the expansion lower bound is more interesting for the sake of rapid mixing, the upper bound in Theorem~\ref{thm:assocexpub}\textemdash which we prove in Appendix~\ref{sec:assocexpub}\textemdash recovers Molloy, Reed, and Steiger's~$\Omega(n^{3/2})$ mixing lower bound~\cite{molloylb}. It is also the first result showing that the associahedron has combinatorial expansion $o(1)$. By contrast, Anari, Liu, Oveis Gharan, and Vinzant recently proved~\cite{anari1, anari2}, settling a conjecture of Mihail and Vazirani~\cite{mihail1989expansion}, that matroids have expansion one. (Mihail and Vazirani in fact conjectured that all graphs realizable as the 1-skeleton of a \emph{0-1 polytope} have expansion one.) Although the set of convex $n$-gon triangulations is not a matroid, it is an important subset system\textemdash and this work shows that it does not have expansion one. More generally, we give the following quasipolynomial bound for $k$-angulations:
\begin{theorem}
\label{thm:kangmix}
For every fixed~$k \geq 3$, the $k$-angulation flip walk on the convex $(k-2)n+2$-point set mixes in time $n^{O(k\log n)}.$
\end{theorem}

In Appendix~\ref{sec:latticetri}, we give a lower bound on the \emph{treewidth} of the $n \times n$ integer lattice point set triangulation flip graph:
\begin{restatable}{theorem}{thmlatticetw}
\label{thm:latticetw}
The treewidth of the triangulation flip graph $F_n$ on the $n \times n$ integer lattice point set is $\Omega(N^{1-o(1)})$, where $N = |V(F_n)|$.
\end{restatable}

\section{Decomposing the convex point set triangulation flip graph}
\label{sec:sliceintuition}

\subsection{Bounding mixing via expansion}
We have a Markov chain that is in fact a random walk on the associahedron~$K_n$. We wish to bound the mixing time of this walk. It turns out that one way to do this is by lower-bounding the \emph{expansion} of the same graph~$K_n$. Intuitively, expansion concerns the extent to which ``bottlenecks'' exist in a graph. More precisely, it measures the ``sparsest'' cut\textemdash the minimum ratio of the number of edges in a cut divided by the number of vertices on the smaller side of the cut:

\begin{definition}
\label{def:exp}
The \emph{edge expansion} (or simply \emph{expansion}), $h(G)$, of a graph $G = (V, E)$ is the quantity
$$\min_{S \subseteq V: |S| \leq |V|/2} |\partial S|/|S|,$$
where~$\partial S = \{(s, t) | s \in S, t \notin S\}$ is the set of edges across the~$(S, V \setminus S)$ cut.
\end{definition}
It is known~\cite{jsconductance, sinclair_1992} that a lower bound on edge expansion leads to an upper bound on mixing:
\begin{lemma}
\label{lem:expmixing}
The mixing time of the Markov chain whose transition matrix is the normalized adjacency matrix of a $\Delta$-regular graph $G$ is 
$$O\left(\frac{\Delta^2\log(|V(G)|)}{(h(G))^2}\right).$$
\end{lemma}
One can do better~\cite{diacstroock, sinclair_1992} if the paths in a \emph{multicommodity flow} are not too long (Section~\ref{sec:cartexpmcflow}).

\subsection{``Slicing and peeling''}
We would like to show that there are many edges in every cut, relative to the number of vertices on one side of the cut. We partition the triangulations~$V(K_n)$ into~$n$ equivalence classes, each inducing a subgraph of~$K_n$. We show that many edges exist between each pair of the subgraphs. Thus the partitioning ``slices'' through many edges. After the partitioning, we show that each of the induced subgraphs has large expansion. To do so, we show that each such subgraph decomposes into many copies of a smaller flip graph~$K_i$, $i < n$. This inductive structure lets us assume that~$K_i$ has large expansion\textemdash then show that the copies of the smaller flip graph are all well connected to one another. We call this ``peeling,'' because one must peel the many~$K_i$ copies from one another\textemdash removing many edges\textemdash to isolate each copy. Molloy, Reed, and Steiger~\cite{molloylb} obtained their~$O(n^{25})$ mixing \emph{upper} bound via a different decomposition, namely using the \emph{central} triangle, via a non-flow-based method. That decomposition is the one we use for our quasipolynomial bound for general $k$-angulations in Appendix~\ref{sec:quadmix}. However, we use a different decomposition here, one with a structure that lets us obtain a nearly tight bound, via a multicommodity flow construction. We formalize the slicing step now:

\begin{definition}
\label{def:orientedpart}
Fix a ``special'' edge~$e^*$ of the convex $n+2$-gon $P_{n+2}$. For each triangle~$T$ having~$e^*$ as one of its edges, define the \emph{oriented class} $\classtt{n}{T}$ to be the set of triangulations of $P_{n+2}$ that include~$T$ as one of their triangles. Let $\mathcal{T}_n$ be the set of all such triangles; let~$\mathcal{S}_n$ be the set of all classes~$\{\classtt{n}{T} | T \in \mathcal{T}_n\}$.

Orient $P_{n+2}$ so that $e^*$ is on the bottom. Then say that $T$ (respectively $\classtt{n}{T}$) is to the \emph{left} of $T'$ (respectively $\classtt{n}{T'}$) if the topmost vertex of $T$ lies counterclockwise around $P_{n+2}$ from the topmost vertex of $T'$. Say that $T'$ lies to the \emph{right} of $T$. Write $T < T'$ and $T' > T$.
\end{definition}

See Figure~\ref{fig:eclass}.

\begin{figure}[h]
\includegraphics[height=8em]{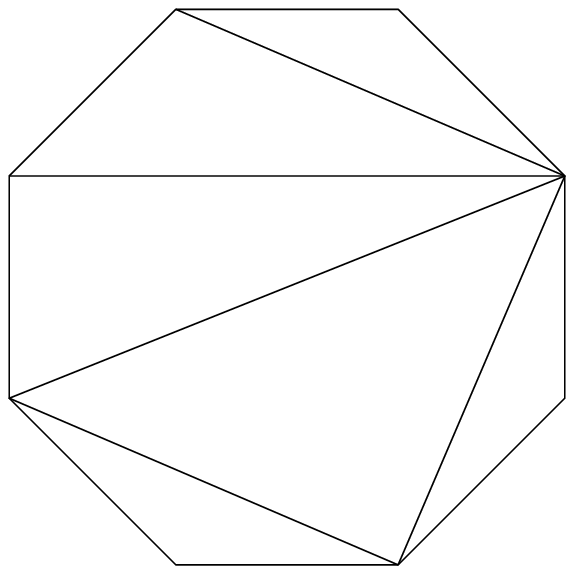}
\hspace{1.5em}
\includegraphics[height=8em]{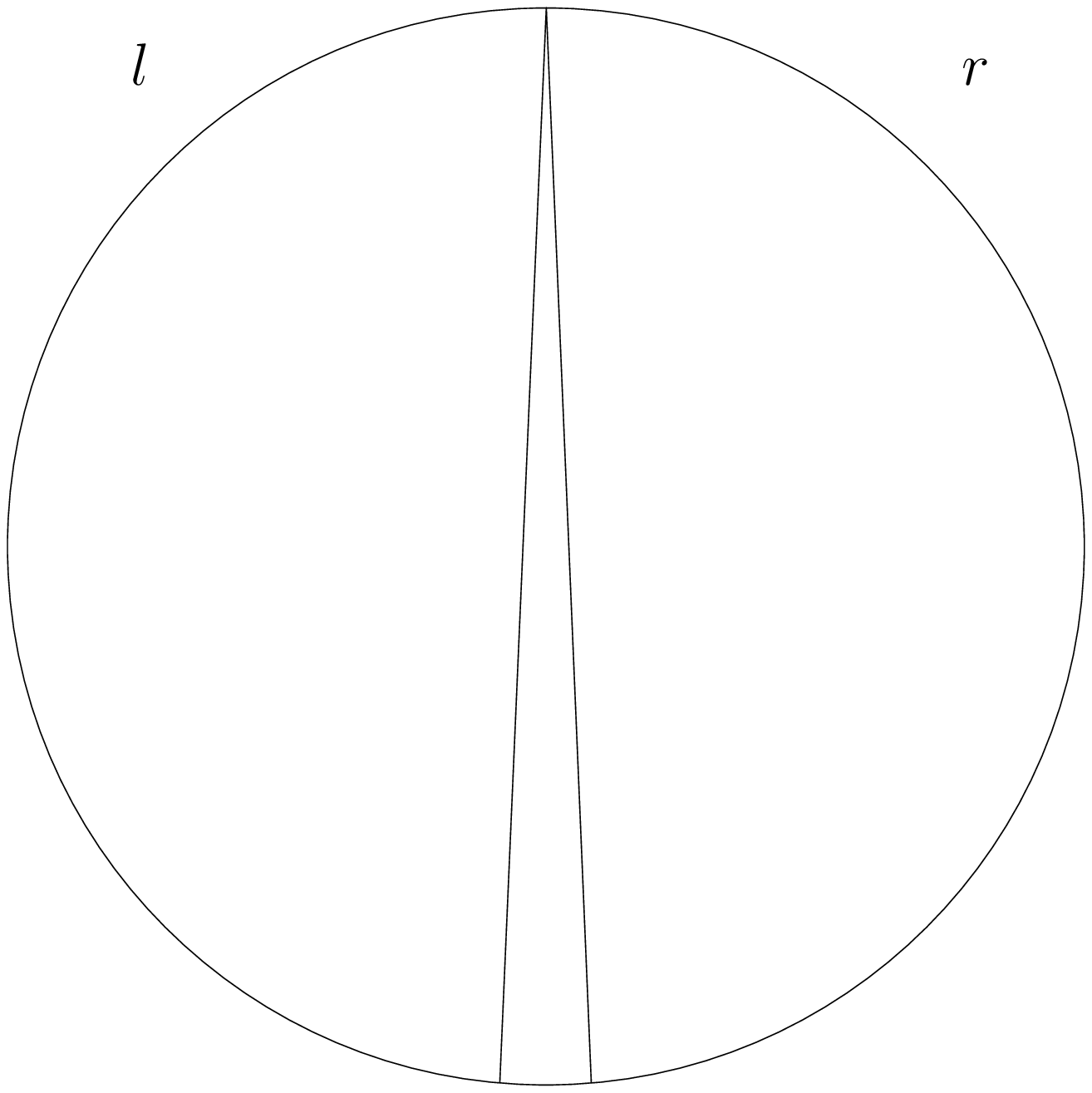}
\hspace{1.5em}
\includegraphics[height=8em]{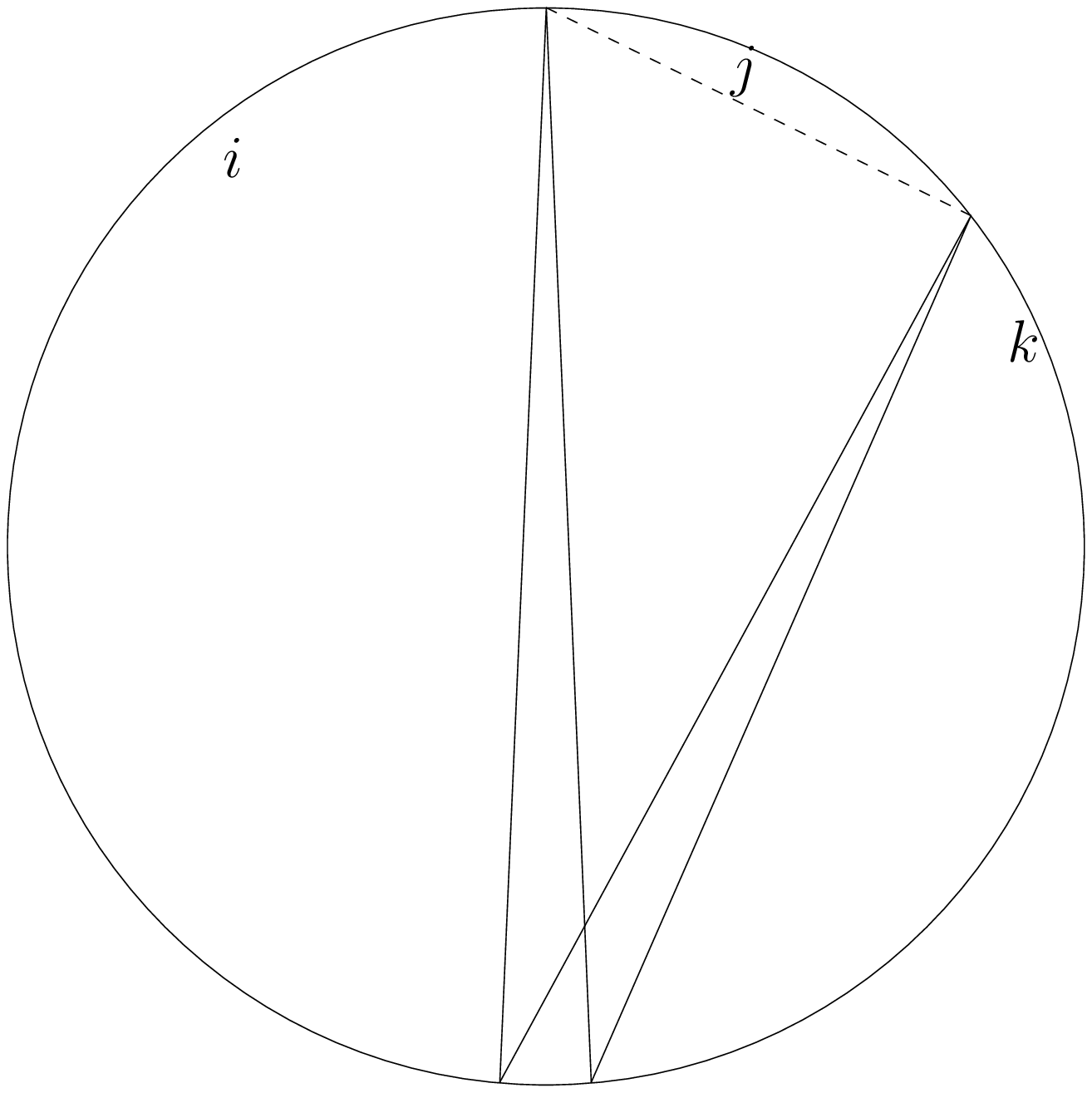}
\caption{Left: A triangulation of the regular octagon. Center: a class $\classtt{n}{T} \in \mathcal{S}_n$, represented schematically by the triangle~$T$ that induces it. We depict the regular $n+2$-gon as a circle (which it approximates as $n\rightarrow \infty$), for ease of illustration. Each triangulation~$t\in \classtt{n}{T}$ consists of~$T$ (the triangle shown), and an arbitrary triangulation of the two polygons on either side of~$T$. Notice that~$\classtt{}{T} \cong K_l \Box K_r$, where~$T$ partitions the~$n+2$-gon into an~$l$-gon and an~$r$-gon. Right: the matching $\edgett{n}{T}{T'}$ between classes $\classtt{n}{T} \cong K_i \Box K_{j+k}$ and $\classtt{n}{T'} \cong K_{i+j} \Box K_k$, is in bijection with the triangulations in $K_i \Box K_j \Box K_k$ (induced by the quadrilateral containing $T$ and $T'$). Therefore, $|\edgett{n}{T}{T'}| = C_{i}C_{j}C_{k}.$}
\label{fig:octtri}
\label{fig:eclass}
\end{figure}

We make observations about the structure of each class as an induced subgraph of~$K_n$

\begin{definition}
\label{def:cartprod}
The \emph{Cartesian product} graph $G \Box H$ of graphs $G$ and $H$ has vertices $V(G) \times V(H)$ and edges $$\{((u, v), (u', v)) | (u, u') \in E(G), v \in V(H)\}$$
$$\cup \{((u, v), (u, v')) | (v, v') \in E(H), u \in V(G)\}.$$

Given a vertex~$w = (u, v) \in V(G) \times V(H)$, call~$u$ the \emph{projection} of~$w$ onto~$G$, and similarly call~$v$ the projection of~$w$ onto~$H$.
\end{definition}
(Applying the obvious associativity of the Cartesian product operator, one can naturally define the product~$G_1 \Box G_2 \Box \cdots \Box G_k = \Box_{i=1}^k G_i$.)

We can now characterize the structure of each class as an induced subgraph of~$K_n$:

\begin{lemma}
\label{lem:projcart}
Each class $\classtt{n}{T}$ is isomorphic to a Cartesian product of two associahedron graphs $K_l$ and $K_r$, with $l + r = n - 1$.
\end{lemma}
\begin{proof}
Each triangle $T$ partitions the $n+2$-gon into two smaller convex polygons with side lengths $l+1$ and~$r+1$, such that~$l + r = n - 1$. Thus each triangulation in $\classtt{n}{T}$ can be identified with a tuple of triangulations of these smaller polygons. The Cartesian product structure then follows from the fact that every flip between two triangulations in $\classtt{n}{T}$ can be identified with a flip in one of the smaller polygons.
\end{proof}

Lemma~\ref{lem:projcart} will be central to the peeling step. For the slicing step, building on the idea in Lemma~\ref{lem:projcart} will help us characterize the edge sets between classes:

\begin{definition}
\label{def:projggedgebdry}
Given classes~$\classtt{n}{T}, \classtt{n}{T'} \in \mathcal{S}_n$, denote by $\edgett{n}{T}{T'}$ the set of edges (flips) between~$\classtt{n}{T}$ and~$\classtt{n}{T'}$.
Let $\bdrytt{n}{T}{T'}$ and $\bdrytt{n}{T'}{T}$ be the \emph{boundary sets}\textemdash the sets of endpoints of edges in $\edgett{n}{T}{T'}$\textemdash that lie respectively in $\classtt{n}{T}$ and $\classtt{n}{T'}$.
\end{definition}

\begin{lemma}
\label{lem:projggbdry}
For each pair of classes $\classtt{n}{T}$ and $\classtt{n}{T'}$, the boundary set $\bdrytt{n}{T}{T'}$ induces a subgraph of $\classtt{n}{T}$ isomorphic to a Cartesian product of the form~$K_i \Box K_j \Box K_k$, for some~$i + j + k = n - 2$.
\end{lemma}
\begin{proof}
Each flip between triangulations in adjacent classes $\classtt{n}{T}$ involves flipping a diagonal of~$T$ to transform the triangulation $t \in \classtt{n}{T}$ into triangulation~$t' \in \classtt{n}{T'}$. Whenever this is possible, there must exist a quadrilateral~$Q$, sharing two sides with~$T$ (the sides that are not flipped), such that both~$t$ and~$t'$ contain $Q$. Furthermore, every~$t \in \classtt{n}{T}$ containing $Q$ has a flip to a distinct~$t' \in \classtt{n}{T'}$. The set of all such boundary vertices $t \in \classtt{n}{T}$ can be identified with the Cartesian product described because~$Q$ partitions~$P_{n+2}$ into three smaller polygons, so that each triangulation in $\bdrytt{n}{T}{T'}$ consists of a tuple of triangulations in each of these smaller polygons, and such that every flip between triangulations in~$\bdrytt{n}{T}{T'}$ consists of a flip in one of these smaller polygons.
\end{proof}

\begin{lemma}
\label{lem:projggmatching}
The set~$\edgett{n}{T}{T'}$ of edges between each pair of classes~$\classtt{n}{T}$ and~$\classtt{n}{T'}$ is a nonempty matching. Furthermore, this edge set is in bijection with the vertices of a Cartesian product~$K_i \Box K_j \Box K_k, i + j + k = n - 2$.
\end{lemma}
\begin{proof}
The claim follows from the reasoning in Lemma~\ref{lem:projggbdry} and from the observation that each triangulation in~$\bdrytt{n}{T}{T'}$ has exactly one flip (namely, flipping a side of the triangle~$T$) to a neighbor in~$\bdrytt{n}{T'}{T}$.
\end{proof}

Lemma~\ref{lem:projggmatching} characterizes the structure of the edge sets (namely matchings) between classes; we would also like to know the sizes of the matchings. We will use the following formula:

\begin{definition}
\label{def:catalan}
Let $C_{n}$ be the $n$th \emph{Catalan number}, defined as $C_{n} = \frac{1}{n+1}{2n \choose n}$.
\end{definition}
\begin{lemma}~\cite{klarner,hiltonpedersen}:
\label{lem:catalan}
The number of vertices in the associahedron $K_{n}$ is~$C_{n}$, and this number grows as $\frac{1}{\sqrt \pi \cdot n^{3/2}}\cdot 2^{2n}.$
\end{lemma}

We will prove the following in Appendix~\ref{sec:missingproofs}:
\begin{restatable}{lemma}{lemmatchingscard}
\label{lem:matchingscard}
For every $T, T' \in \mathcal{T}_n,$ 
$$|\edgett{n}{T}{T'}| \geq \frac{|\classtt{n}{T}||\classtt{n}{T'}|}{C_n}.$$
\end{restatable}

Lemma~\ref{lem:matchingscard}\textemdash which states that the number of edges between a pair of classes is at least equal to the product of the cardinalities of the classes, divided by the total number of vertices in the graph~$|V(K_n)| = C_n$\textemdash is crucial to this paper. To explain why this is, we will need to present our multicommodity flow construction (Appendix~\ref{sec:flowdetails}). We will give intuition in Section~\ref{sec:framework}. For now, it suffices to say that Lemma~\ref{lem:matchingscard} implies that there are many edges between a given pair of classes, justifying (intuitively) the slicing step. For the peeling step, we need the fact that Cartesian graph products preserve the well-connectedness of the graphs in the product~\cite{Graham1998IsoperimetricIF}:

\begin{lemma}
\label{lem:cartexp}
Given graphs $G_1, G_2, \dots, G_k$, Cartesian product $G_1 \Box G_2 \Box \cdots \Box G_k$ satisfies
$$h(G_1\Box G_2\Box \cdots \Box G_k) \geq \frac{1}{2}\min_i h(G_i).$$
\end{lemma}

Lemma~\ref{lem:projcart} says that each of the classes~$\classtt{n}{T} \in \mathcal{S}_n$ is a Cartesian graph product of associahedron graphs~$K_l, K_r$, $l < n, r < n$, allowing us to ``peel'' (decompose)~$\classtt{n}{T}$ into graphs that can then be recursively sliced into classes and peeled. Lemma~\ref{lem:cartexp} implies that the peeling must disconnect many edges, as it involves splitting a Cartesian product graph into many subgraphs (copies of~$K_l$).

We will make all of this intuition rigorous in Appendix~\ref{sec:flowdetails} by constructing our flow. The choice of paths through which to route flow will closely trace the edges in this recursive ``slicing and peeling'' decomposition. We will then show that, with this choice of paths, the resulting~\emph{congestion}\textemdash the maximum amount of flow carried along an edge\textemdash is bounded by a suitable polynomial factor. This will provide a lower bound on the expansion.

\begin{figure}[h]
\includegraphics[width=9em]{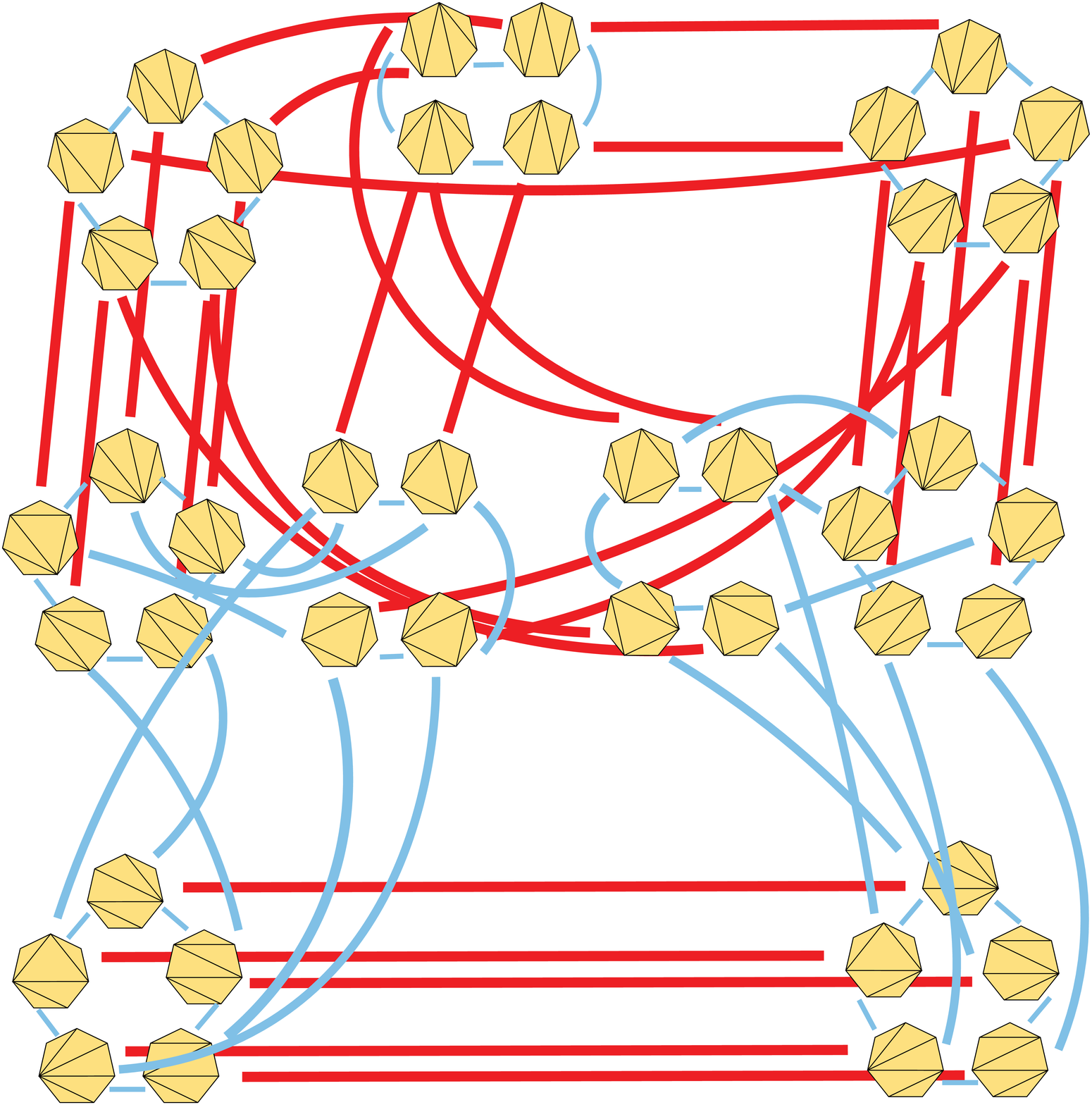}
\hspace{8em}
\includegraphics[width=13em]{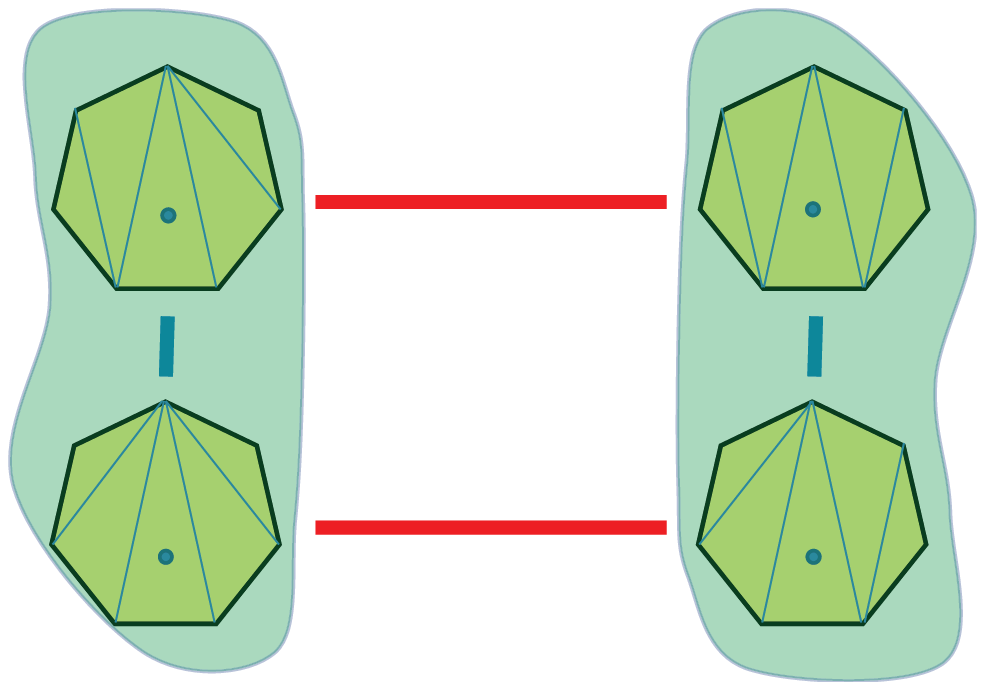}
\caption{\textbf{Left}: The associahedron graph $K_5$, with each vertex representing a triangulation of the regular heptagon. Flips are shown with edges (in blue and red). The vertex set~$V(K_n)$ is partitioned into a set~$\mathcal{S}_n$ of five equivalence classes (of varying sizes). Within each class, all triangulations share the same triangle containing the bottom edge~$e^*$. Flips (edges) between triangulations in the same class are shown in blue. Flips between triangulations in different classes are shown in red. To ``slice''~$K_5$ into its subgraphs, one must cut through these red matchings. \textbf{Right}: A class~$\classtt{5}{T}$ from the graph~$K_5$ on the left-hand side, viewed as an induced subgraph of~$K_5$. The identifying triangle~$T$ is marked with a blue dot. This subgraph is isomorphic to a Cartesian product of two~$K_2$ graphs; each copy of~$K_2$ induced by fixing the rightmost diagonal is outlined in green. ``Peeling'' apart this product requires disconnecting the two red edges connecting the~$K_2$ copies.}
\label{fig:heptclassesor}
\end{figure}

\section{Bounding expansion via multicommodity flows}
\label{sec:prelim}
\label{sec:cartexpmcflow}
The way we will lower-bound expansion is by using \emph{multicommodity flows}~\cite{sinclair_1992, kaibelexp}.
\begin{definition}
\label{def:mcflow}
A \emph{multicommodity flow}~$\phi$ in a graph $G = (V, E)$ is a collection of functions $\{f_{st}:A \rightarrow \mathbb{R}\ \mid s, t \in V\}$, where~$A = \bigcup_{\{u,v\}\in E} \{(u, v), (v, u)\},$ combined with a \emph{demand function} $D: V \times V \rightarrow \mathbb{R}$.

Each $f_{st}$ is a flow sending $D(s, t)$ units of a commodity from vertex $s$ to vertex $t$ through the edges of $G$. We consider the capacities of all edges to be infinite. Let $f_{st}(u,v)$ be the amount of flow sent by $f_{st}$ across the arc $(u, v)$. (It may be that $f_{st}(u, v) \neq f_{st}(v, u)$.) Let 
$$f(u, v) = \frac{1}{|V|}\sum_{s,t\in V\times V} f_{st}(u, v),$$ and let~$\rho = \max_{(u,v) \in A} f(u,v).$ Call~$\rho$ the \emph{congestion}.
\end{definition}
Unless we specify otherwise, we will mean by ``multicommodity flow'' a \emph{uniform multicommodity flow}, i.e. one in which~$D(s, t) = 1$ for all $s,t$. The following is well established and enables the use of multicommodity flows as a powerful lower-bounding technique for expansion:
\begin{lemma}
\label{lem:flowexp}
Given a uniform multicommodity flow $f$ in a graph $G = (V, E)$ with congestion~$\rho$, the expansion $h(G)$ is at least $1/(2\rho)$.
\end{lemma}

Lemma~\ref{lem:flowexp}, combined with Lemma~\ref{lem:expmixing}, gives an automatic upper bound on mixing time given a multicommodity flow with an upper bound on congestion\textemdash but with a quadratic loss. As we will discuss in Appendix~\ref{sec:fwpf}, one can do better if the paths used in the flow are short~\cite{diacstroock, sinclair_1992}.

\section{Our framework}
\label{sec:framework}
In addition to the new mixing bounds for triangulations and for general $k$-angulations, we make general technical contributions, in the form of three meta-theorems, which we present in this section. Our first general technical contribution, Theorem~\ref{thm:flowprojres}, provides a recursive mechanism for analyzing the expansion of a flip graph in terms of the expansion of its subgraphs. Equivalently, viewing the random walk on such a flip graph as a Markov chain, this theorem provides a mechanism for analyzing the mixing time of a chain, in terms of the mixing times of smaller \emph{restriction} chains into which one decomposes the original chain\textemdash and analyzing a \emph{projection} chain over these smaller chains. We obtain, in certain circumstances such as the $k$-angulation walk, better mixing time bounds than one obtains applying similar prior decomposition theorems\textemdash which used a different underlying machinery.

The second theorem, Lemma~\ref{lem:fw}, observes and formalizes a set of conditions satisfied by a number of chains (equivalently, flip graphs) under which one can apply either our Theorem~\ref{thm:flowprojres}, or prior decomposition techniques, to obtain rapid mixing reuslts. Depending on the chain, one may then obtain better results either by applying Theorem~\ref{thm:flowprojres}, or by applying the prior techniques. Lemma~\ref{lem:fw} does not require using our Theorem~\ref{thm:flowprojres}; instead, one can use the spectral gap or log-Sobolev constant as the underlying techincal machinery using Jerrum, Son, Tetali, and Vigoda's Theorem~\ref{thm:specprojres}. Prior work exists applying these techniques (using Theorem~\ref{thm:specprojres}) to sampling $q$-colorings~\cite{heinrich2020glauber} in bounded-treewidth graphs and independent sets in regular trees~\cite{jerrumprojres}, as well as probabilistic graphical models in machine learning~\cite{hierwidth} satisfying certain conditions. Lemma~\ref{lem:fw} amounts to an observation unifying these applications. We apply this observation to general $k$-angulations, noting that they satisfy a relaxation of this theorem (Lemma~\ref{lem:fwquasi}), giving a quasipolynomial bound. This bound will come from incurring a polynomial loss over logarithmic recursion depth.

The third theorem, Lemma~\ref{lem:fwstrong}, adapts the machinery in Theorem~\ref{thm:flowprojres} to eliminate this multiplicative loss altogether, assuming that a chain satisfies certain properties. One such key property is the existence large matchings in Lemma~\ref{lem:matchingscard} in Section~\ref{sec:sliceintuition}. Another property, which we will discuss further after presenting Lemma~\ref{lem:fwstrong}, is that the \emph{boundary sets}\textemdash the vertices in one class (equivalently, states in a restriction chain) having neighbors in another class\textemdash are well connected to the rest of the first class. When these properties are satisfied, one can apply our flow machinery to overcome the multiplicative loss and obtain a polynomial bound. However, the improvement relies on observations about congestion that do not obviously translate to the spectral setting.

\subsection{Markov chain decomposition via multicommodity flow}

In this section we state our first general theorem. To place our contribution in context with prior work, we cast our flip graphs in the language of Markov chains. As we discussed in Section~\ref{sec:mixingexp}, any Markov chain satisfying certain mild conditions has a \emph{stationary} distribution~$\pi^*$ (which in the case of our triangulation walks is uniform). We can view such a chain as a random walk on a graph~$\mathcal{M}$ (an unweighted graph in the case of the chains we consider, which have uniform distributions and regular transition probabilities). In the case of convex polygon triangulations, we have~$\mathcal{M} = K_n$.

The flip graph~$\mathcal{M}$ has vertex set~$\Omega$ and (up to normalization by degree) adjacency matrix~$P$\textemdash and we abuse notation, identifying the Markov chain~$\mathcal{M}$ with this graph. When~$\pi^*$ is not uniform, it is easy to generalize the flip graph to a \emph{weighted} graph, with each vertex (state)~$t$ assigned weight~$\pi(t)$, and each transition (edge)~$(t, t')$ assigned weight~$\pi(t)P(t, t') = \pi(t')P(t', t)$. We assume here that this latter equality holds, a condition on the chain~$\mathcal{M}$ known as \emph{reversibility}. We then replace a uniform multicommodity flow with one where~$D(t, t') = \pi(t)\pi(t')$ (up to normalization factors). 

\begin{definition}
\label{def:projres}
Consider a Markov chain~$\mathcal{M}$ with finite state space~$\Omega$ and probability transition matrix~$P$, and stationary distribution~$\pi$. Consider a partition of the states of~$\Omega$ into classes~$\Omega_1, \Omega_2, \dots, \Omega_k$. Let the \emph{restriction} chain, for $i=1,\dots,k$, be the chain with state space~$\Omega_i$, probability distribution~$\pi_i$, with~$\pi_i(x) = \pi(x)/(\sum_{y\in\Omega_i}\pi(y))$, for~$x\in\Omega_i$, and transition probabilities~$P_i(x, y) = P(x, y)/(\sum_{z\in\Omega_i} P(x,z))$. Let the \emph{projection} chain be the chain with state space~$\bar\Omega = \{1, 2, \dots, k\}$, stationary distribution~$\bar\pi$, with~$\bar\pi(i) = \sum_{x\in\Omega_i} \pi(i)$, and transition probabilities~$\bar P(i,j) = \sum_{x\in\Omega_i,y\in\Omega_j}P(x,y)$.
\end{definition}

\begin{restatable}{theorem}{thmflowprojres}
\label{thm:flowprojres}
Let~$\mathcal{M}$ be a reversible Markov chain with finite state space~$\Omega$ probability transition matrix~$P$, and stationary distribution~$\pi^*$. Suppose~$\mathcal{M}$ is connected (irreducible). Suppose~$\mathcal{M}$ can be decomposed into a collection of restriction chains~$(\Omega_1, P_1), (\Omega_2, P_2), \dots, (\Omega_k, P_k)$, and a projection chain~$(\bar\Omega, \bar P)$. Suppose each restriction chain admits a multicommodity flow (or canonical paths) construction with congestion at most~$\rho_{\max}$. Suppose also that there exists a multicommodity flow construction in the projection chain with congestion at most~$\bar\rho$. Then there exists a multicommodity flow construction in~$\mathcal{M}$ (viewed as a weighted graph in the natural way) with congestion
$$(1+2\bar\rho\gamma\Delta)\rho_{\max},$$ where~$\gamma = \max_{i\in[k]}\max_{x\in\Omega_i}\sum_{y\notin\Omega_i}P(x,y),$ and~$\Delta$ is the degree of~$\mathcal{M}$.
\end{restatable}
The proof of Theorem~\ref{thm:flowprojres} is in Appendix~\ref{sec:fwpf}.

We give a full proof in Appendix~\ref{sec:fwpf}. Jerrum, Son, Tetali, and Vigoda~\cite{jerrumprojres} presented an analogous (and classic) decomposition theorem, which we restate below as Theorem~\ref{thm:specprojres}, and which has become a standard tool in mixing time analysis. The key difference between our theorem and theirs is that our theorem uses multicommodity flows, while their theorem uses the so-called \emph{spectral gap}\textemdash another parameter that can use to bound the mixing time of a chain. Often, the spectral gap gives tighter mixing bounds than combinatorial methods. Their Theorem~\ref{thm:specprojres} gave bounds analogous to our Theorem~\ref{thm:flowprojres}, but with the multicommodity flow congestion replaced with the \emph{spectral gap} of a chain\textemdash and with a~$3\gamma$ term in place of our~$2\gamma$. (They also gave an analogous version for the \emph{log-Sobolev} constant\textemdash yet another parameter for bounding mixing times.) The spectral gap of a chain~$\mathcal{M} = (\Omega, P)$, which we denote~$\lambda$, is the difference between the two largest eigenvalues of the transition matrix~$P$ (which we can view as the normalized adjacency matrix of the corresponding weighted graph). The key point is that while on the one hand the mixing time~$\tau$ satisfies
$\tau \leq \lambda^{-1}\log |\Omega|,$ the bound on mixing using expansion in Lemma~\ref{lem:expmixing} comes from passing through the spectral gap:
$\lambda \geq \frac{(h(\mathcal{M}))^2}{2\Delta^2},$ where~$\Delta$ is the degree of the flip graph and~$h(\mathcal{M})$ is the expansion of~$\mathcal{M}$.
The quadratic loss in passing from expansion to mixing is not incurred when bounding the spectral gap directly, so one can obtain better bounds via the spectral gap. Jerrum, Son, Tetali, and Vigoda gave a mechanism for doing precisely this:

\begin{theorem}
\label{thm:specprojres}\cite{jerrumprojres}
Let~$\mathcal{M}$ be a reversible Markov chain with finite state space~$\Omega$ probability transition matrix~$P$, and stationary distribution~$\pi^*$. Suppose~$\mathcal{M}$ is connected (irreducible). Suppose~$\mathcal{M}$ can be decomposed into a collection of restriction chains~$(\Omega_1, P_1), (\Omega_2, P_2), \dots, (\Omega_k, P_k)$, and a projection chain~$(\bar\Omega, \bar P)$. Suppose each restriction chain has spectral gap at least~$\lambda_{\min}$. Suppose also that the projection chain has spectral gap at least~$\bar\lambda$. Then ~$\mathcal{M}$ has gap at least $$\min\left\{\frac{\lambda_{\min}}{3}, \frac{\bar\lambda\lambda_{\min}}{3\gamma + \bar\lambda}\right\},$$ where~$\gamma$ is as in Theorem~\ref{thm:flowprojres}.
\end{theorem}

Our Theorem~\ref{thm:flowprojres} has a simple, purely combinatorial proof (Appendix~\ref{sec:fwpf}), and fills a gap in the literature by showing that such a construction can be used in place of the spectral machinery from the earlier technique. We also obtain a tighter bound on expansion than would result from a black-box application of Theorem~\ref{thm:specprojres}. The cost to our improvement is in passing from expansion to mixing via the spectral gap.
Nonetheless, we will show that in the case of triangulations, our Theorem~\ref{thm:flowprojres} can be adapted to give a new mixing bound whereas, by contrast, it is not clear how to obtain even a polynomial bound adapting Jerrum, Son, Tetali, and Vigoda's spectral machinery. We will also show that for general $k$-angulations, one can, with our technique, use a combinatorial insight to eliminate the~$\gamma$ factor in our decomposition in favor of a~$\Delta^{-1}$ factor (for $k$-angulations we have~$\gamma = k/\Delta$)\textemdash whereas it is not clear how to do so with the spectral decomposition.

\subsection{General pattern for bounding projection chain congestion}
Our second decomposition theorem, which we will apply to general $k$-angulations, states that if one can recursively decompose a chain into restriction chains in a particular fashion, and if the projection chain is well connected, then Theorem~\ref{thm:flowprojres} gives an expansion bound:

\begin{lemma}
\label{lem:fw}
Let~$\mathcal{F} = \{\mathcal{M}_1, \mathcal{M}_2, \dots \}$ be a family of connected graphs, parameterized by a value $n$. Suppose that every graph~$\mathcal{M}_n = (\mathcal{V}_n, \mathcal{E}_n) \in \mathcal{F}$, for $n \geq 2$, can be partitioned into a set~$\mathcal{S}_n$ of classes satisfying the following conditions:
\begin{enumerate}
\item\label{multcondcart} Each class in $\mathcal{S}_n$ is isomorphic to a Cartesian product of one or more graphs~$\classt{}{T} \cong \mathcal{M}_{i_1} \Box \cdots \mathcal{M}_{i_k}$, where for each such graph~$\mathcal{M}_{i_j} \in \mathcal{F}$, $i_j \leq n/2$.
\item\label{multcondnum} The number of classes is~$O(1)$.
\item\label{multcondmatch} For every pair of classes~$\classt{}{T}, \classt{}{T'}\in \mathcal{S}_n$ that share an edge, the number of edges between the two classes is~$\Omega(1)$ times the size of each of the two classes.
\item\label{multcondsize} The ratio of the sizes of any two classes is~$\Theta(1)$.
\end{enumerate}
Suppose further that~$|\mathcal{V}_1| = 1$. Then the expansion of~$\mathcal{M}_n$ is~$\Omega(n^{-O(1)})$.
\end{lemma}

Lemma~\ref{lem:fw} is easy to prove given Theorem~\ref{thm:flowprojres}. An analogue in terms of spectral gap is easy to prove given Theorem~\ref{thm:specprojres}. Furthermore, as we will prove in Appendix~\ref{sec:fwpf}, a precise statement of the bounds given by Lemma~\ref{lem:fw} is as follows:
\begin{restatable}{lemma}{lemnonhierexact}
\label{lem:nonhierexact}
Suppose a flip graph $\mathcal{M}_n = (\mathcal{V}_n, \mathcal{E}_n)$ belongs to a family~$\mathcal{F}$ of graphs satisfying the conditions of Lemma~\ref{lem:fw}. Suppose further that every graph~$\mathcal{M}_k = (\mathcal{V}_k, \mathcal{E}_k) \in \mathcal{F}$, $k < n$, satisfies
$$|\mathcal{V}_k|/|\mathcal{E}_{k,\min}| \leq f(k),$$
for some function $f(k)$, where $\mathcal{E}_{k,\min}$ is the smallest edge set between adjacent classes $\classt{}{T},\classt{}{T'} \in \mathcal{S}_k$, where $\mathcal{S}_k$ is as in Lemma~\ref{lem:fw}.
Then the expansion of $\mathcal{M}_n$ is
$$\Omega(1/(2 f(n))^{\log n})),$$
where~$\gamma$ is as in Theorem~\ref{thm:flowprojres}, and~$\Delta$ is the degree of~$\mathcal{M}_n$.
\end{restatable}
\begin{proof}
Constructing an arbitrary multicommodity flow (or set of canonical paths) in the projection graph at each inductive step gives the result claimed. The term~$|\mathcal{V}_k|/|\mathcal{E}_{k,\min}|$ bounds the (normalized) congestion in any such flow because the total amount of flow exchanged by all pairs of vertices (states) combined is~$|\mathcal{V}_k|^2$, and the minimum weight of an edge in the projection graph is~$|\mathcal{E}_{k,\min}|$.

Notice that we do not incur a~$\gamma\Delta$ term here, because even if a state (vertex) in~$\Omega_i \subseteq \mathcal{V}_k$ has neighbors~$x\in \Omega_j, y \in \Omega_l$, $z$ still only receives no more than~$|\mathcal{V}_k|^2/\mathcal{E}_{k,\min}\}$ flow across the edges~$(z, x)$ and~$(z, y)$ combined.
\end{proof}

\begin{remark}
\label{rmk:deletegamma}
The~$\gamma\Delta$ factor in Theorem~\ref{thm:flowprojres}, which does not appear in Lemma~\ref{lem:nonhierexact}, does appear in a straightforward appliation of Jerrum, Son, Tetali, and Vigoda's Theorem~\ref{thm:specprojres}. 
\end{remark}

We will show that $k$-angulations (with fixed~$k \geq 4$) satisfy a relaxation of Lemma~\ref{lem:fw}:
\begin{lemma}
\label{lem:fwquasi}
Suppose a family~$\mathcal{F}$ of graphs satisfies the conditions of Lemma~\ref{lem:fw}, with the~$\Omega(1)$,~$O(1)$, and~$\Theta(1)$ factors in Conditions~\ref{multcondmatch},~\ref{multcondnum}, and~\ref{multcondsize} respectively replaced by~$\Omega(n^{-O(1)})$, $O(n^{O(1)})$, and~$\Theta(n^{O(1)})$. Then for every~$\mathcal{M}_n \in \mathcal{F}$, the expansion of~$\mathcal{M}_n$ is~$\Omega(n^{-O(\log n)})$.
\end{lemma}

Lemma~\ref{lem:fw} enables us to relate a number of chains admitting a certain decomposition process in a black-box fashion, unifying prior work applying Theorem~\ref{thm:specprojres} separately to individual chains. Marc Heinrich~\cite{heinrich2020glauber} presented a similar but less general construction for the Glauber dynamics on~$q$-colorings in bounded-treewidth graphs; other precursors exist, including for the hardcore model on certain trees~\cite{jerrumprojres} and a general argument for a class of graphical models~\cite{hierwidth}. 
In the companion paper we mentioned in Section~\ref{sec:associntro}, we apply Lemma~\ref{lem:fw} to chains for sampling independent sets and dominating sets in bounded-treewidth graphs, as well as chains on $q$-colorings, maximal independent sets, and several other structures, in graphs whose treewidth and degree are bounded.

\subsection{Eliminating inductive loss: nearly tight conductance for triangulations}
We now give the meta-theorem that we will apply to triangulations. Lemma~\ref{lem:fw}\textemdash using either Theorem~\ref{thm:flowprojres} or Theorem~\ref{thm:specprojres}\textemdash gives a merely quasipolynomial bound when applied straightforwardly to $k$-angulations, including the case of triangulations\textemdash simply because the $f(n)$ term in Lemma~\ref{lem:nonhierexact} is~$\omega(1)$ and thus the overall congestion is~$\omega(1)^{\log n}$ (not polynomial). However, it turns out that the large matchings given by Lemma~\ref{lem:matchingscard} between pairs of classes in the case of triangulations (but not general $k$-angulations), combined with some additional structure in the triangulation flip walk, satisfy an alternative set of conditions that suffice for rapid mixing. The conditions are:
\begin{restatable}{lemma}{lemfwstrong}
\label{lem:fwstrong}
Let~$\mathcal{F} = \{\mathcal{M}_1, \mathcal{M}_2, \dots\}$ be an infinite family of connected graphs, parameterized by a value $n$. Suppose that for every graph~$\mathcal{M}_n = (\mathcal{V}_n, \mathcal{E}_n) \in \mathcal{F}$, for $n \geq 2$, the vertex set~$\mathcal{V}_n$ can be partitioned into a set~$\mathcal{S}_n$ of classes inducing subgraphs of~$\mathcal{M}_n$ that satisfy the following conditions:
\begin{enumerate}
\item\label{addcondcart} Each subgraph is isomorphic to a Cartesian product of one or more graphs~$\classt{}{T} \cong \mathcal{M}_{i_1} \Box \cdots \mathcal{M}_{i_k}$, where for each such graph~$\mathcal{M}_{i_j} \in \mathcal{F}$, $i_j < n$.
\item\label{addcondnum} The number of classes is~$n^{O(1)}$.
\item\label{addcondmatch} For every pair of classes~$\classt{}{T}, \classt{}{T'}\in \mathcal{S}_n$, the set of edges between the subgraphs induced by the two classes is a matching of size at least~$\frac{|\classt{}{T}||\classt{}{T'}|}{|\mathcal{V}_n|}.$
\item\label{addcondbdry} Given a pair of classes~$\classt{}{T}, \classt{}{T'} \in \mathcal{S}_n$, there exists a graph~$\mathcal{M}_i$ in the Cartesian product~$\classt{}{T}$, and a class~$\classt{}{U} \in \mathcal{S}_i$ within the graph~$\mathcal{M}_i$, such that the set of vertices in~$\classt{}{T}$ having a neighbor in~$\classt{}{T'}$ is precisely the set of vertices in~$\classt{}{T}$ whose projection onto~$\mathcal{M}_i$ lies in~$\classt{}{U}$. Furthermore, no class~$\classt{}{U}$ within~$\mathcal{M}_i$ is the projection of more than one such boundary.
\end{enumerate}

Suppose further that~$|\mathcal{V}_1| = 1$. Then the expansion of~$\mathcal{M}_n$ is~$\Omega(1/(\kappa(n)n))$, where~$\kappa(n) = \max_{1\leq i\leq n}|\classt{}{S_i}|$ is the maximum number of classes in any~$\mathcal{M}_i, i\leq n$.
\end{restatable}

Unlike Lemma~\ref{lem:fw}, this lemma requires a purely combinatorial construction; it is not clear how to apply spectral methods to obtain even a polynomial bound.
Condition~\ref{addcondbdry} is crucial. To give more intuition for this condition, we state and prove the following fact about the triangulation flip graph (visualized in Figure~\ref{fig:tk}):

\begin{figure}[h]
\includegraphics[height=8em]{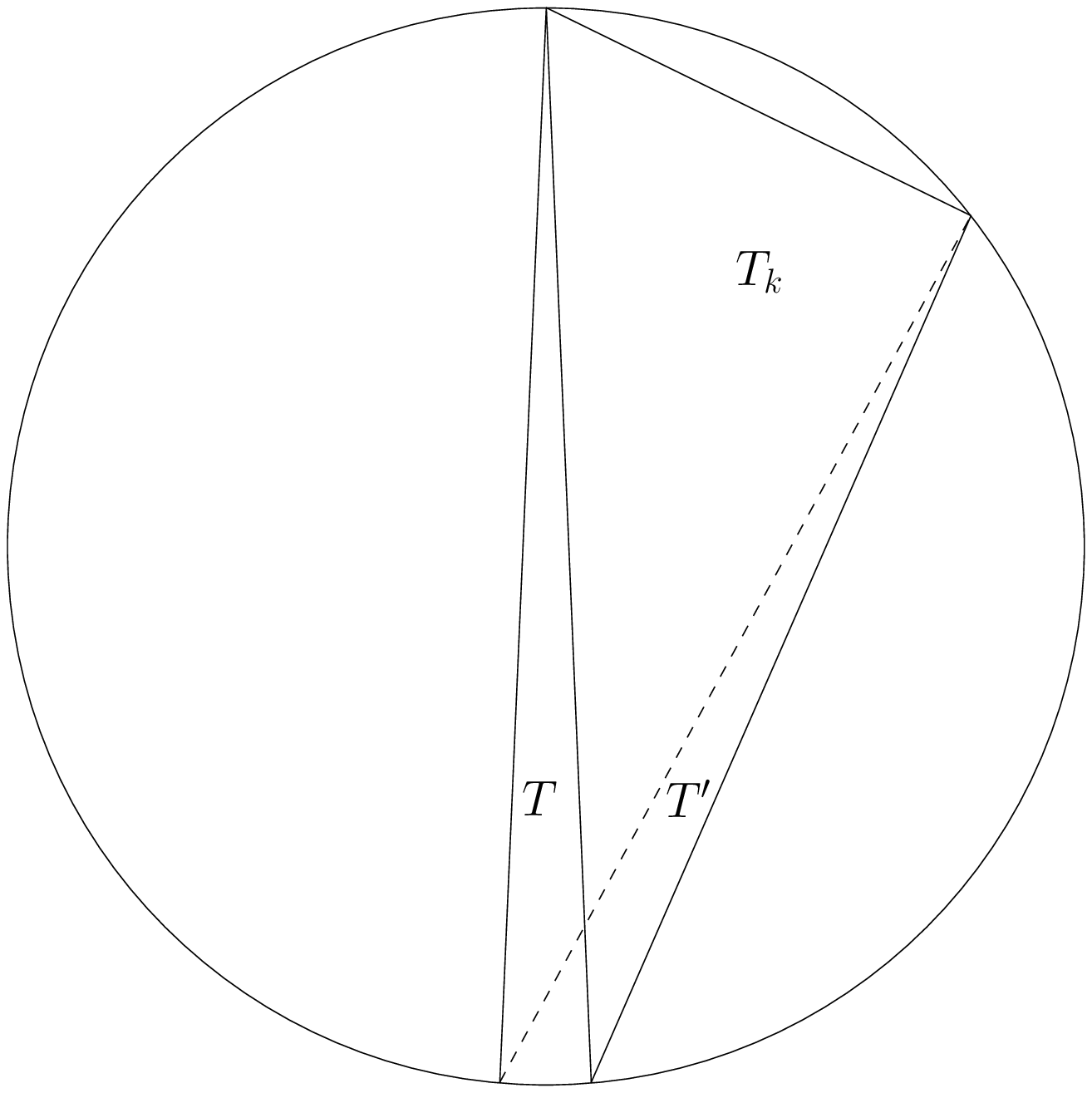}
\hspace{0.5em}
\includegraphics[height=8em]{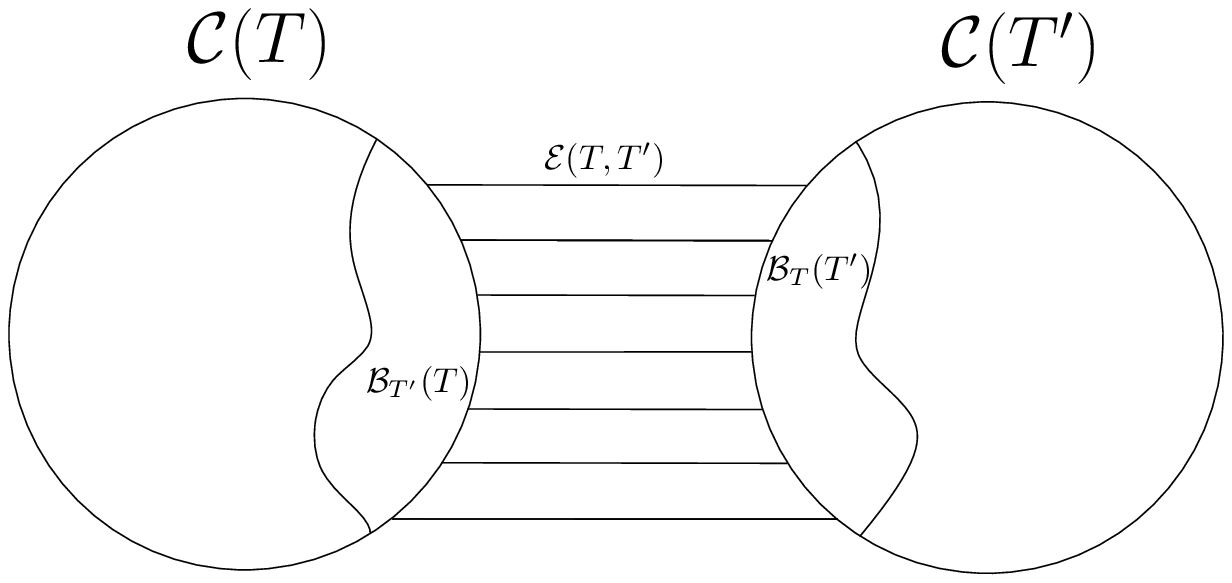}
\hspace{0.5em}
\includegraphics[width=10em]{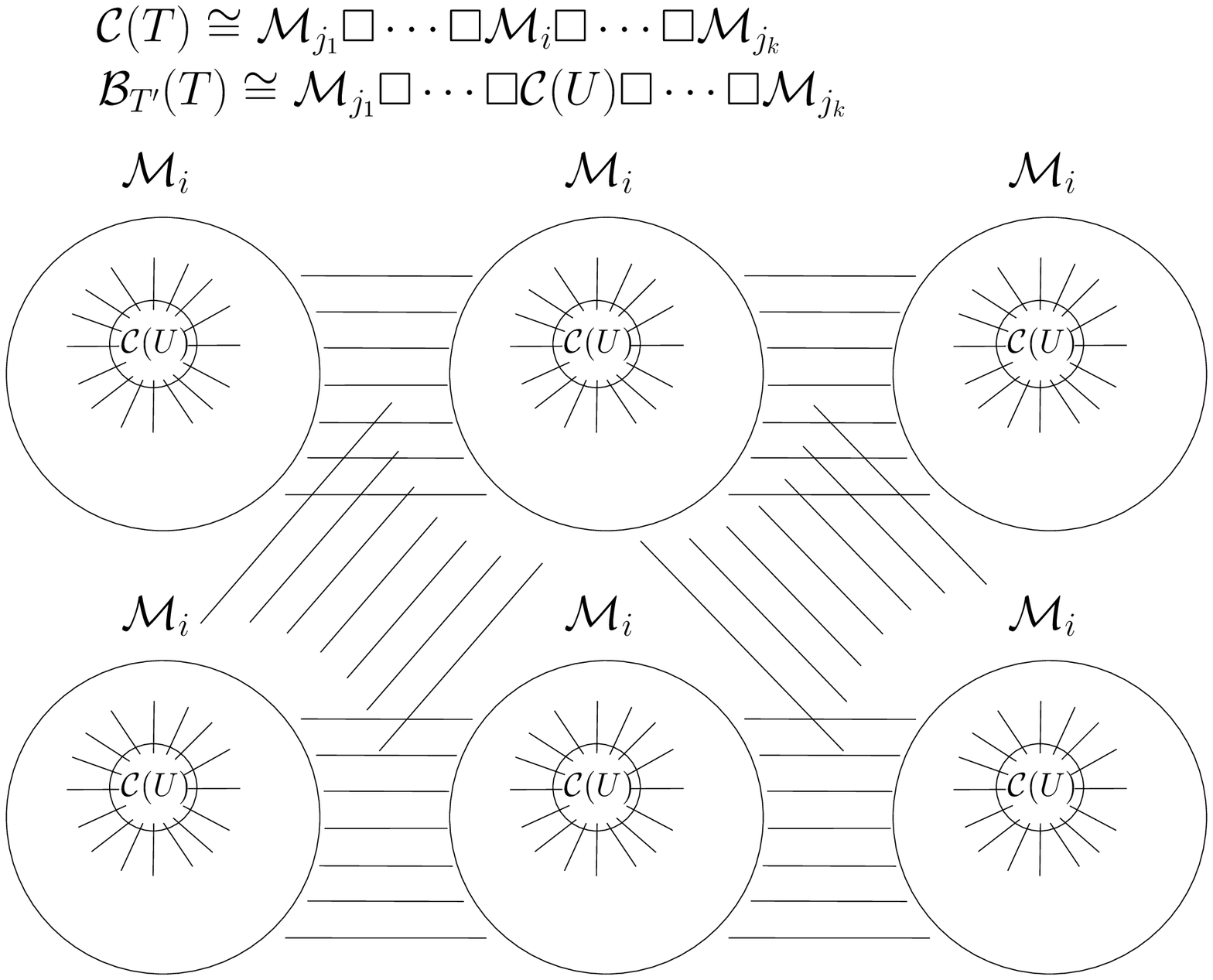}
\caption{\textbf{Left}: (Lemma~\ref{lem:bdryquad}) The set of edges $\edgett{}{T}{T'}$ has $K_i \Box \classtt{}{T_k}$ as its set of boundary vertices in $\classtt{}{T}$.
\textbf{Center}: An illustration of Condition~\ref{addcondmatch} in Lemma~\ref{lem:fwstrong}, showing a large matching~$\edget{}{T}{T'}$ between two classes (subgraphs)~$\classt{}{T}$ and~$\classt{}{T'}$.
\textbf{Right}: An illustration of Conditions~\ref{addcondcart} and~\ref{addcondbdry} in Lemma~\ref{lem:fwstrong}: $\classt{}{T}$ as a Cartesian product of smaller graphs~$\mathcal{M}_{j_1}, \dots, \mathcal{M}_i, \dots, \mathcal{M}_{j_k}$ in the family~$\mathcal{F}$. The schematic view shows this Cartesian product as a collection of copies of~$\mathcal{M}_i$, connected via perfect matchings between pairs of the copies\textemdash with the pairs to connect determined by the structure of the Cartesian product. The boundary~$\bdryt{}{T}{T'}$ (center) is isomorphic to a class~$\classt{}{U}$ (right) within~$\mathcal{M}_i$, a graph in the product. Within each copy of~$\mathcal{M}_i$, many edges connect~$\classt{}{U}$ to the rest of~$\mathcal{M}_i$.}
\label{fig:fwaddbdry-noflow}
\label{fig:tk}
\end{figure}
\begin{lemma}
\label{lem:bdryquad}
Given $T, T' \in \mathcal{T}_n,$ suppose $T'$ lies to the right of $T$. Then the subgraph of $\classtt{n}{T}$ induced by $\bdrytt{n}{T}{T'}$ is isomorphic to a Cartesian product $K_l \Box \classtt{r}{T_k},$ where $l + r = n - 1$, and where $T_k$ has as an edge the right diagonal of $T$, and as the vertex opposite this edge the topmost vertex of $T'.$ A symmetric fact holds for $\bdrytt{n}{T'}{T}.$
\end{lemma}
\begin{proof}
Every triangulation in~$\bdrytt{n}{T}{T'}$ (i) includes the triangle~$T$ and (ii) is a single flip away from including the triangle~$T'$. As we observed in the proof of Lemma~\ref{lem:projggbdry}, this implies that~$\bdrytt{n}{T}{T'}$ consists of the set of triangulations in~$\classtt{n}{T}$ containing a quadrilateral~$Q$. Specifically,~$Q$ shares two sides with~$T$: one of these is $e^*$, and the other is the left side of~$T$. One of the other two sides of~$Q$ is the right side of~$\classtt{n}{T'}$. Combining this side with the ``top'' side of~$Q$ and with the right side of~$T$, one obtains the triangle~$T_k$, proving the claim.
\end{proof}

Lemma~\ref{lem:bdryquad} implies that there are many edges between the boundary set~$\bdrytt{n}{T}{T'}$ and the rest of~$\classtt{}{T}$: ~$\classtt{}{T}\cong K_l \Box K_r$, where~$K_l$ and~$K_r$ are smaller associahedron graphs, so~$\classtt{}{T}$ is a collection of copies of~$K_r$, with pairs of copies connected by perfect matchings. Each~$K_r$ copy can itself be decomposed into a set~$\mathcal{S}_r$ of classes, one of which, namely~$\classtt{r}{T_k}$, is the intersection of~$\bdrytt{n}{T}{T'}$ with the~$K_r$ copy. Applying Condition~\ref{addcondmatch} to the~$K_r$ copy implies that there are many edges between boundary vertices in~$\classtt{r}{T_k}$ to other subgraphs (classes) in the~$K_r$ copy. That is, the boundary set~$\bdrytt{n}{T}{T'}$ is well connected to the rest of~$\classtt{n}{T}$.

Figure~\ref{fig:fwaddbdry-noflow} visualizes this situation in general terms for the framework. We have now proven:
\begin{lemma}
\label{lem:triangnewcond}
The associahedron graph~$K_n$, along with the oriented partition in Definition~\ref{def:orientedpart}, satisfies the conditions of Lemma~\ref{lem:fwstrong}.
\end{lemma}
\begin{proof}
The connectedness of~$K_n$ is known~\cite{mct}. Conditions~\ref{addcondcart} and~\ref{addcondmatch} follow from Lemma~\ref{lem:projcart}, Lemma~\ref{lem:projggmatching}, and Lemma~\ref{lem:matchingscard}. Concerning the boundary sets, Condition~\ref{addcondbdry} follows from Lemma~\ref{lem:bdryquad} and from the discussion leading to this lemma.
\end{proof}

Together with Lemma~\ref{lem:expmixing} and the easy fact that $K_n$ is a $\Theta(n)$-regular graph, Lemma~\ref{lem:triangnewcond} implies rapid mixing, pending the proof of Lemma~\ref{lem:fwstrong}\textemdash which we prove in Appendix~\ref{sec:flowdetails}.

\subsection{Intuition for the flow construction for triangulations}
\label{sec:decompintuition}
We will prove Lemma~\ref{lem:fwstrong} in Appendix~\ref{sec:flowdetails}, from which a coarse expansion lower bound for triangulations\textemdash and a corresponding coarse (but polynomial) upper bound for mixing\textemdash will be immediate by Lemma~\ref{lem:triangnewcond}. We give some intuition now for the flow construction we will give in the proof of Lemma~\ref{lem:fwstrong}, and in particular for the centrality of Condition~\ref{addcondmatch} and Condition~\ref{addcondbdry} (corresponding respectively to Lemma~\ref{lem:matchingscard} and Lemma~\ref{lem:bdryquad} for triangulations). Consider the case of triangulations, for concreteness. Every~$t\in\classtt{n}{T}, t'\in\classtt{n}{T'}$ must exchange a unit of flow. This means that a total of~$|\classtt{n}{T}||\classtt{n}{T'}|$ flow must be sent across the matching~$\edgett{n}{T}{T'}$. To minimize congestion, it will be optimal to equally distribute this flow across all of the boundary matching edges. We can decompose the overall problem of routing flow from each~$t\in\classtt{n}{T}$ to each~$t'\in\classtt{n}{T'}$ into three subproblems: (i) \emph{concentrating} flow from every triangulation in~$\classtt{n}{T}$ within the boundary set~$\bdrytt{n}{T}{T'}$, (ii) routing flow across the matching edges~$\edgett{n}{T}{T'}$, i.e. from~$\bdrytt{n}{T}{T'} \subseteq \classtt{n}{T}$ to~$\bdrytt{n}{T'}{T} \subseteq \classtt{n}{T'}$, and (iii) \emph{distributing} flow from the boundary~$\bdrytt{n}{T'}{T}$ to each~$t' \in \classtt{n}{T'}$.
\begin{figure}[h]
\includegraphics[height=7em]{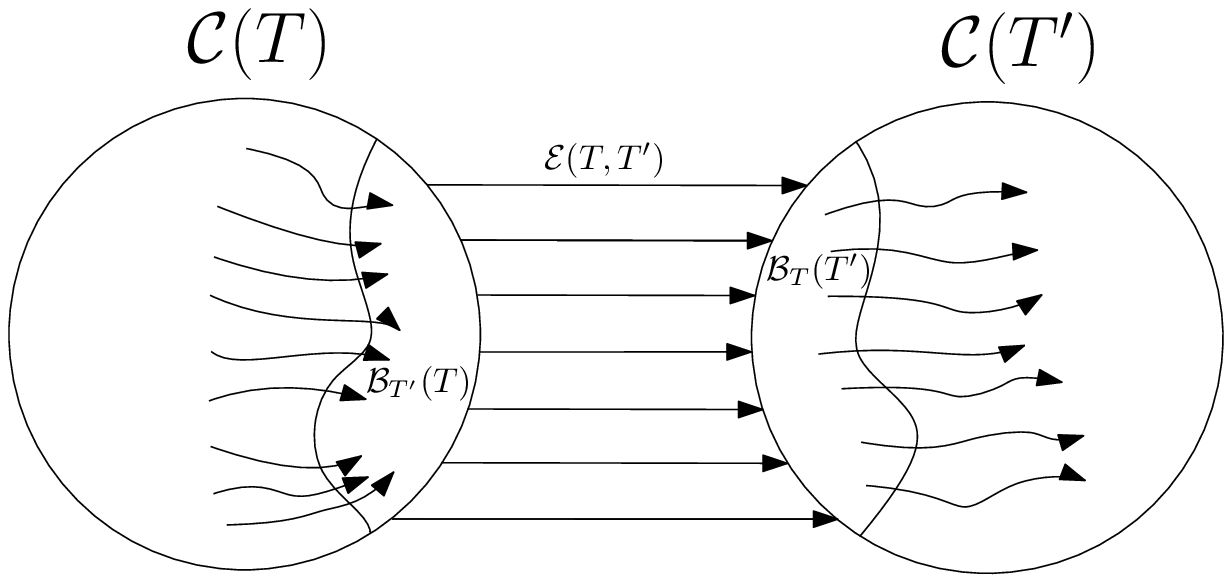}
\hspace*{0.5em}
\includegraphics[height=7em]{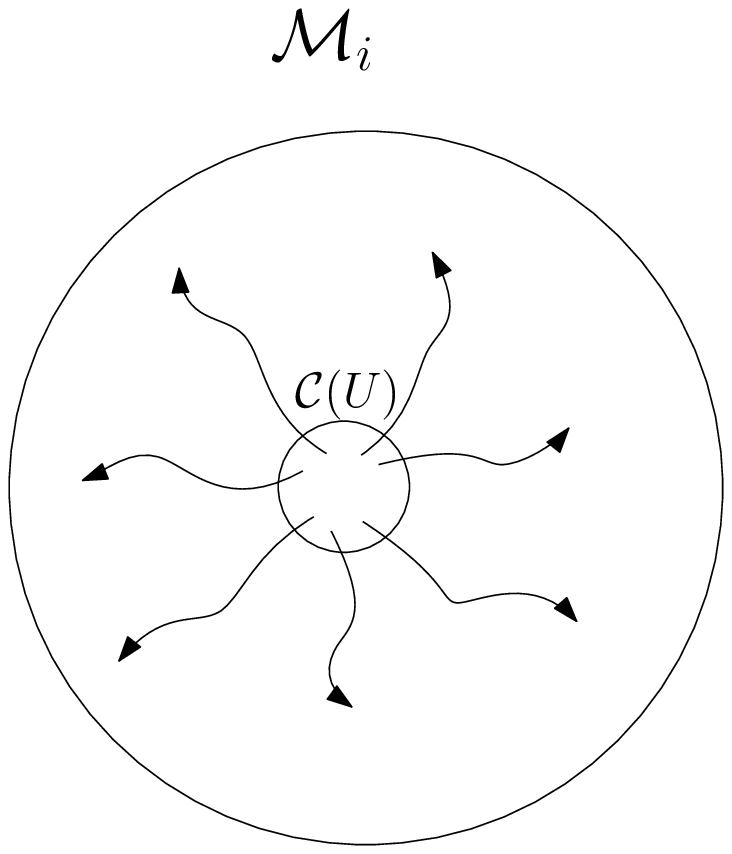}
\hspace*{0.5em}
\includegraphics[height=7em]{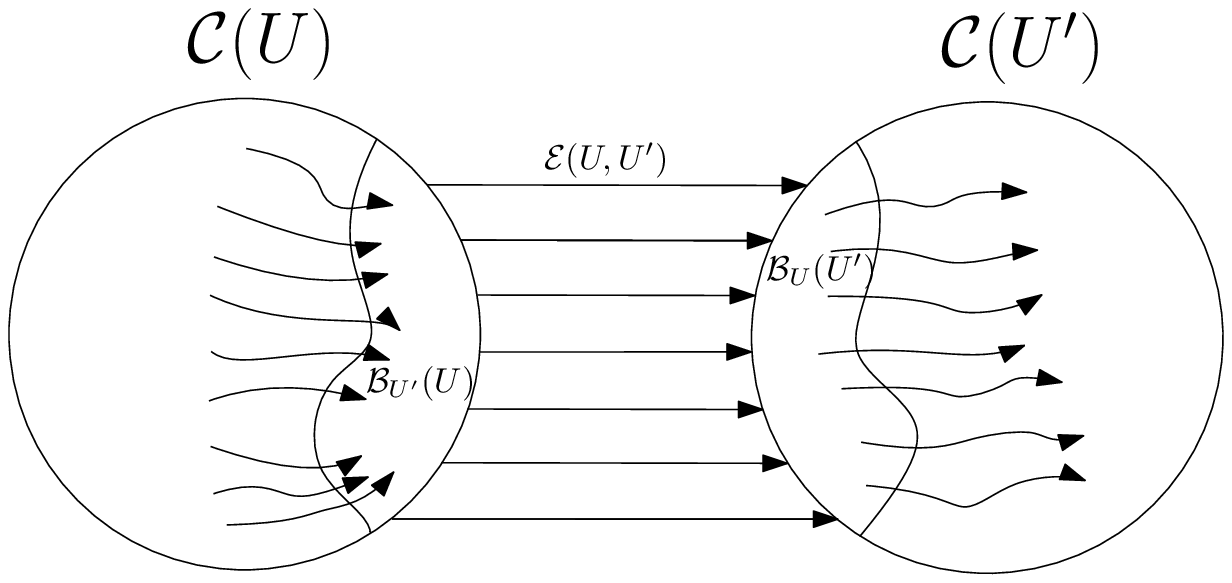}
\caption{\textbf{Left}: The problem of sending flow from each~$t \in \classtt{n}{T}$ to each~$t' \in \classtt{n}{T'}$, decomposed into subproblems: (i) \emph{concentrating} flow within~$\bdrytt{n}{T}{T'}$, (ii) \emph{transmitting} the flow across the boundary matching~$\edgett{n}{T}{T'}$, and (iii) \emph{distributing} the flow from~$\bdrytt{n}{T'}{T}$ throughout~$\classtt{n}{T'}$.
\textbf{Center}: Within each copy of~$\mathcal{M}_i$ in the product~$\classtt{n}{T'} \cong \mathcal{M}_{j_1} \Box  \cdots \Box \mathcal{M}_i \Box \cdots \Box \mathcal{M}_{j_k}$, the distribution problem in Figure~\ref{fig:decompintuition1} induces the problem of distributing flow from a class~$\classtt{n}{U}$\textemdash namely the projection of~$\bdrytt{n}{T'}{T}$ onto~$\mathcal{M}_i$\textemdash throughout the rest of~$\mathcal{M}_i$.
\textbf{Right}: The problem in the center figure induces subproblems in which~$\classtt{n}{U} \subseteq \mathcal{M}_i$ must send flow to each~$\classtt{n}{U'} \subseteq \mathcal{M}_i$. These subproblems are of the same form as the original~$\classtt{n}{T}, \classtt{n}{T'}$ problem (left), and can be solved recursively. The large matchings~$\edgett{n}{T}{T'}, \edgett{n}{U}{U'}$ guaranteed by Condition~\ref{addcondmatch} prevent any recursive congestion increase.}
\label{fig:decompintuition1}
\label{fig:decompintuition2}
\end{figure}
Now, the amount of flow that must be concentrated from~$\classtt{n}{T}$ at \emph{each} boundary triangulation~$u\in\bdrytt{n}{T}{T'}$ (and symmetrically distributed from \emph{each}~$v\in\bdrytt{n}{T'}{T}$ throughout~$\classtt{n}{T'}$) is equal to
$$\frac{|\classtt{n}{T}||\classtt{n}{T'}|}{|\bdrytt{n}{T}{T'}|} = \frac{|\classtt{n}{T}||\classtt{n}{T'}|}{|\bdrytt{n}{T'}{T}|} = \frac{|\classtt{n}{T}||\classtt{n}{T'}|}{|\edgett{n}{T}{T'}|} \leq C_n,$$
where we have used the equality~$|\bdrytt{n}{T}{T'}| = |\bdrytt{n}{T'}{T}| = |\edgett{n}{T}{T'}|$ by Lemma~\ref{lem:projggbdry} and Lemma~\ref{lem:projggmatching}, and where the inequality follows from Lemma~\ref{lem:matchingscard}. As a result, in the ``concentration'' and ``distribution'' subproblems (i) and (iii), at most~$C_n$ flow is concentrated at or distributed from any given triangulation (Figure~\ref{fig:decompintuition1}). This bound yields a recursive structure: the concentration (respectively distribution) subproblem decomposes into a flow problem within~$\classtt{n}{T}$ (respectively~$\classtt{n}{T'}$), in which, by the inequality, each triangulation has~$C_n$ total units of flow it must receive (or send).  We will then apply Condition~\ref{addcondbdry}, observing (see Figure~\ref{fig:decompintuition2}) that the concentration (symmetrically) distribution of this flow can be done entirely between pairs of classes~$\classtt{n}{U}, \classtt{n}{U'}$ within copies of a smaller flip graph~$\mathcal{M}_i$ in the Cartesian product~$\classtt{n}{T'} \cong \mathcal{M}_{j_1} \Box  \cdots \Box \mathcal{M}_i \Box \cdots \Box \mathcal{M}_{j_k}$. 

The~$\classtt{n}{U}, \classtt{n}{U'}$ subproblem is of the same form as the original~$\classtt{n}{T}, \classtt{n}{T'}$ problem (Figure~\ref{fig:decompintuition1}), and we will show that the~$C_n$ bound on the flow (normalizing to congestion one) across the~$\edgett{n}{T}{T'}$ edges will induce the same~$C_n$ bound across the~$\edgett{n}{U}{U'}$ edges in the induced subproblem. We further decompose the~$\classtt{n}{U},\classtt{n}{U'}$ problem into concentration, transmission, and distribution subproblems without any gain in overall congestion. To see this, view the initial flow problem in~$K_n$ as though every triangulation~$t \in V(K_n)$ is initially ``charged'' with~$|V(K_n)| = C_n$ total units of flow to distribute throughout~$K_n$. Similarly, in the induced distribution subproblem within each copy of~$\mathcal{M}_i = K_i$ in the product~$\classtt{n}{T'}$, each vertex on the boundary~$\bdrytt{n}{T}{T'}$ is initially ``charged'' with~$C_n$ total units to distribute throughout~$K_i$. Just as the original problem in~$K_n$ results in each~$\edgett{n}{T}{T'}$ carrying at most~$C_n$ flow across each edge, similarly (we will show in Appendix~\ref{sec:flowdetails}) the induced problem in~$K_i$ results in each~$\edgett{n}{U}{U'}$ carrying at most~$C_n$ flow across each edge. This preservation of the bound~$C_n$ under the recursion avoids any congestion increase.

One must be cautious, due to the linear recursion depth, not to accrue even a constant-factor loss in the recursive step (the coefficient~$2$ in Theorem~\ref{thm:flowprojres}). In Theorem~\ref{thm:flowprojres}, it turns out that this loss comes from routing \emph{outbound} flow within a class~$\classtt{n}{T}$\textemdash flow that must be sent to other classes\textemdash and then also routing \emph{inbound} flow. The combination of these steps involves two ``recursive invocations'' of a uniform multicommodity flow that is inductively assumed to exist within~$\classtt{n}{T}$. 
We will show in Appendix~\ref{sec:flowdetails} that one can avoid the second ``invocation'' with an initial ``shuffling'' step: a uniform flow within~$\classtt{n}{T}$ in which each triangulation~$t\in\classtt{n}{T}$ distributes all of its outbound flow evenly throughout~$\classtt{n}{T}$.

It is here that Jerrum, Son, Tetali, and Vigoda's spectral Theorem~\ref{thm:specprojres} breaks down, giving a~$3$-factor loss at each recursion level, due to applying the Cauchy-Schwarz inequality to a \emph{Dirichlet form} that is decomposed into expressions over the restriction chains. Although Jerrum, Son, Tetali, and Vigoda gave circumstances for mitigating or eliminating their multiplicative loss, this chain does not satisfy those conditions in an obvious way. 

\appendix
\section*{Appendix}

\section{Proof that the conditions of Lemma~\ref{lem:fwstrong} imply rapid mixing}
\label{sec:assocexplb}
\label{sec:flowdetails}
In this section we prove Lemma~\ref{lem:fwstrong}:
\lemfwstrong*

We will use the fact that one can prove an analogue of Lemma~\ref{lem:cartexp} for multicommodity flows\textemdash namely one that does not lose a factor of two. We prove this in Appendix~\ref{sec:missingproofs}:
\begin{restatable}{lemma}{lemcartflow}
\label{lem:cartflow}
Let $J = G \Box H$. Given multicommodity flows $g$ and $h$ in $G$ and $H$ respectively with congestion at most~$\rho$, there exists a multicommodity flow $f$ for $J$ with congestion at most~$\rho$.
\end{restatable}

We will construct a ``good flow''\textemdash that is, a uniform multicommodity flow with polynomially bounded congestion\textemdash in any~$\mathcal{M}_n \in \mathcal{F}$ satisfying the conditions of Lemma~\ref{lem:fwstrong}, via an inductive process. The base case,~$|\mathcal{V}| = 1$, is trivial. For the inductive hypothesis, we assume that for all~$i < n$, there exists a good flow in~$\mathcal{M}_i$. For the inductive step, we begin by combining Lemma~\ref{lem:cartflow} with Condition~\ref{addcondcart} to obtain a good flow in each~$\classt{}{T}$: since each class is a product of smaller graphs~$\{\mathcal{M}_i\}$ in the same family, the inductive assumption that those smaller graphs have good flows carries through to~$\classt{}{T}$ by Lemma~\ref{lem:cartflow}.

The more difficult part of the inductive step is then to route flow between pairs of vertices that lie in different classes. We now introduce machinery, in the form of \emph{multi-way single-commodity flows}, that we will apply to the boundary set structure in Condition~\ref{addcondbdry} to find the right paths for these pairs.

\begin{definition}
\label{def:mwscf}
Define a \emph{multi-way single-commodity flow} (MSF), given a graph $G = (V, E)$, with \emph{source} set $S \subseteq V$ and \emph{sink} set $T \subseteq V$, and a set of ``surplus'' and ``deficit'' amounts $\sigma: S \rightarrow \mathbb{R}$ and $\delta: T \rightarrow \mathbb{R}$, as a flow $f: A(E) \rightarrow \mathbb{R}$ in $G$, such that:
\begin{enumerate}
\item the net flow out of each vertex $s \in S \setminus T$ is $\sigma(s)$,
\item the net flow into each vertex $t \in T \setminus S$ is $\delta(t)$, 
\item the net flow out of each vertex $u \in S \cap T$ is $\sigma(u) - \delta(u)$, and
\item the net flow into (out of) each vertex $u \in V \setminus (S \cup T)$ is zero.
\end{enumerate}
Denote the MSF as the tuple $\rho = (f, S, T, \sigma, \delta)$.
(Here $A(E)$ is the directed arc set obtained by creating directed arcs $(u, v)$ and $(v, u)$ for each edge $\{u, v\} \in E$.)
When~$\sigma$ and~$\delta$ are constant functions, abuse notation and denote by~$\sigma$ and~$\delta$ their values.
\end{definition}
Intuitively, Definition~\ref{def:mwscf} describes sending flow from some set of vertices (the source set) in a graph to another set (the sink set). It differs from a multicommodity flow in that it is not important that every vertex in $S$ send flow to every vertex in $T$. For instance, in a bipartite graph, if the source set and sink set are the two sides of the bipartition, and all surpluses and demands are one, it suffices to direct the flow across a matching.

It will also be useful to talk about an \emph{MSF problem}, in which we are given surpluses and demands but need to find the actual flow function.
\begin{definition}
\label{def:msfprob}
Define a \emph{multi-way single-commodity flow problem} (MSF problem) as a tuple $\pi = (S, T, \sigma, \delta)$, where $S, T, \sigma, \delta$ are as in Definition~\ref{def:mwscf}, but no flow function $f$ is specified.
\end{definition}

(One could alternatively formulate an MSF problem as a more familiar $s-t$ flow problem by adding extra vertices and edges. However, Definition~\ref{def:mwscf} will make our flow construction more convenient.)

The main lemma of this section is as follows:
\begin{restatable}{lemma}{lemfwstrongindstep}
\label{lem:fwstrongindstep}
Let a graph~$\mathcal{M}_n \in \mathcal{F}$ be given, with~$n > 1$ and~$\mathcal{F}$ satisfying the conditions of Lemma~\ref{lem:fwstrong}. Suppose that for all~$1 \leq i < n$, the graph~$\mathcal{M}_i$ has a uniform multicommodity flow with congestion at most~$\rho$, for some~$\rho > 0$. Then there exists a uniform multicommodity flow in~$\mathcal{M}_n$ with congestion at most~$\rho + \kappa$, where~$\kappa = |\mathcal{S}_n|$ is the number of classes in the partition described in Lemma~\ref{lem:fwstrong}.
\end{restatable}

Lemma~\ref{lem:fwstrongindstep} forms the inductive step of an argument that easily proves Lemma~\ref{lem:fwstrong}.

To prove Lemma~\ref{lem:fwstrongindstep}, we will start by partitioning~$\mathcal{M}_n$ into the classes~$\mathcal{S}_n$ as described in~Lemma~\ref{lem:fwstrong}. Now consider any vertex~$s\in \classt{}{T}$, for a given class~$\classt{}{T} \in \mathcal{S}_n$, and consider any other class~$\classt{}{T'} \neq \classt{}{T}$. Consider a multi-way single-commodity flow problem
$$\pi_s = (\{s\}, \classt{}{T'}, \sigma_s = |\classt{}{T'}|, \delta_s = 1).$$

We will ``solve'' this problem\textemdash construct a flow function~$f_s$ that satisfies the surpluses and demands of the problem. Notice that to solve~$\pi_s$ is to send a unit of flow from~$s$ to every~$t\in \classt{}{T'}$. Thus if we construct such a function~$f_s$ for every~$s\in\classt{}{T}$, and construct similar flows for every pair of classes~$\classt{}{T}, \classt{}{T'}$, we will have constructed a uniform multicommodity flow in~$\mathcal{M}_n$. We will do precisely this, then analyze the congestion of the sum of these flow functions.

To construct~$f_s$, we will express the problem~$\pi_s$ as the composition of four MSF problems
$$\pi_{shuf} = (\{s\}, \classt{}{T}, \sigma_{shuf} = \sigma_s = |\classt{}{T'}|, \delta_{shuf} = \frac{|\classt{}{T'}|}{|\classt{}{T}|}),$$
$$\pi_{conc} = (\classt{}{T}, \bdryt{}{T}{T'}, \sigma_{conc} = \delta_{shuf}, \delta_{conc} = \frac{|\classt{}{T'}|}{|\bdryt{}{T}{T'}|}),$$
$$\pi_{tran} = (\bdryt{}{T}{T'}, \bdryt{}{T'}{T}, \sigma_{tran} = \delta_{tran} = \delta_{conc} = \frac{|\classt{}{T'}|}{|\bdryt{}{T}{T'}|} = \frac{|\classt{}{T'}|}{|\bdryt{}{T'}{T}|},$$
$$\pi_{dist} = (\bdryt{}{T'}{T}, \classt{}{T'}, \sigma_{dist} = \delta_{tran}, \delta_{dist} = \delta_s = 1).$$

(Here we have defined the matching~$\edget{}{T}{T'}$ and the boundary set~$\bdryt{}{T}{T'}$ for the general family~$\mathcal{F}$ in the same way we defined~$\edgett{n}{T}{T'}$ and~$\bdrytt{n}{T}{T'}$ for the associahedron in Definition~\ref{def:projggedgebdry}. We have implicitly used the equality $|\bdryt{}{T}{T'}| = |\edget{}{T}{T'}| = |\bdryt{}{T'}{T}|,$ which follows from the assumption in Condition~\ref{addcondmatch} that these boundary edges form a matching.)

\begin{remark}
\label{rmk:composemsf}
It is easy to see, by comparing~$\sigma$ and~$\delta$ values and by comparing source and sink sets, that if one specifies flow functions solving the four subproblems~$\pi_{shuf},\pi_{conc},\pi_{tran},\pi_{dist}$, one can take the arc-wise sum of these functions as a solution to the original MSF problem~$\pi_s$.
\end{remark}

\begin{figure}[h]
\includegraphics[height=10em]{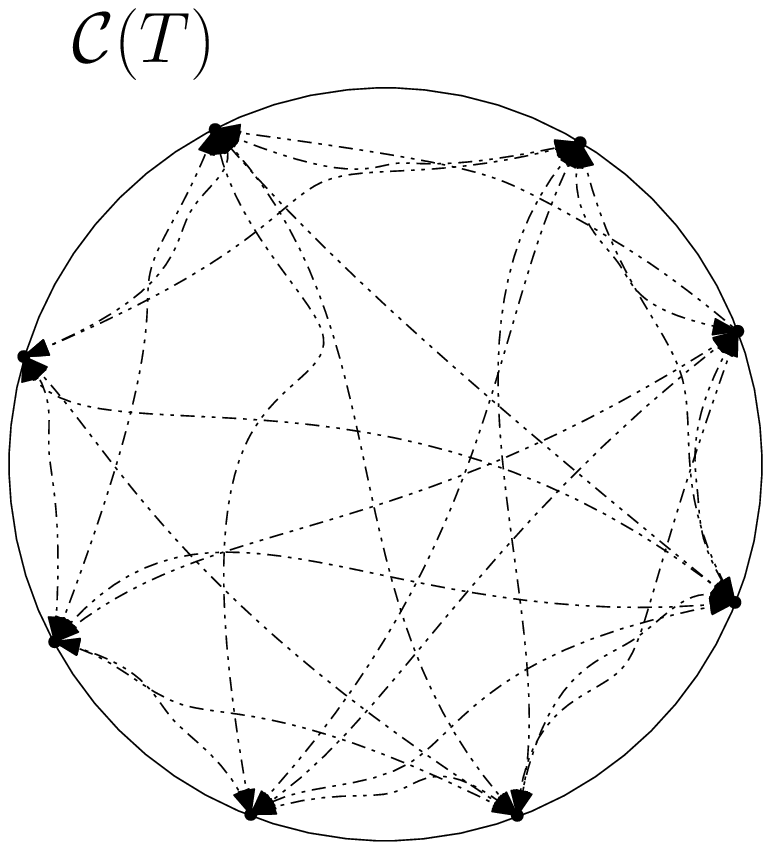}
\hspace*{1.5em}
\includegraphics[height=10em]{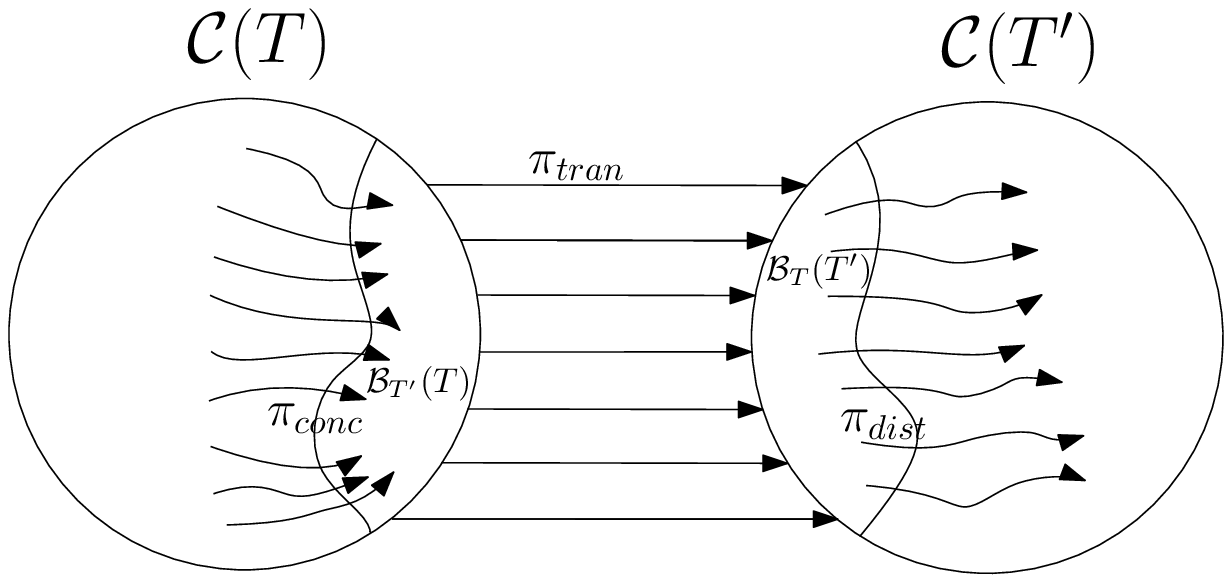}
\caption{The decomposition of the MSF problem~$\pi_s$. Left:~$\pi_{shuf}$, solved in aggregate for all~$s\in\classt{}{T}$ by a uniform multicommodity flow in~$\classt{}{T}$. Right: the problems~$\pi_{conc},\pi_{tran},$ and~$\pi_{dist}$, in which the (single) commodity from~$s\in\classt{}{T}$ begins uniformly spread throughout~$\classt{}{T}$. The flow must then be concentrated on the boundary~$\bdryt{}{T}{T'}$ (for~$\pi_{conc}$), sent to~$\classt{}{T'}$ (for~$\pi_{tran}$), and distributed uniformly throughout~$\classt{}{T'}$ (for~$\pi_{dist}$).}
\label{fig:fwaddtoplvl}
\end{figure}

\begin{figure}[h]
\includegraphics[width=8em]{fwaddrec.eps}
\hspace*{1.5em}
\includegraphics[width=20em]{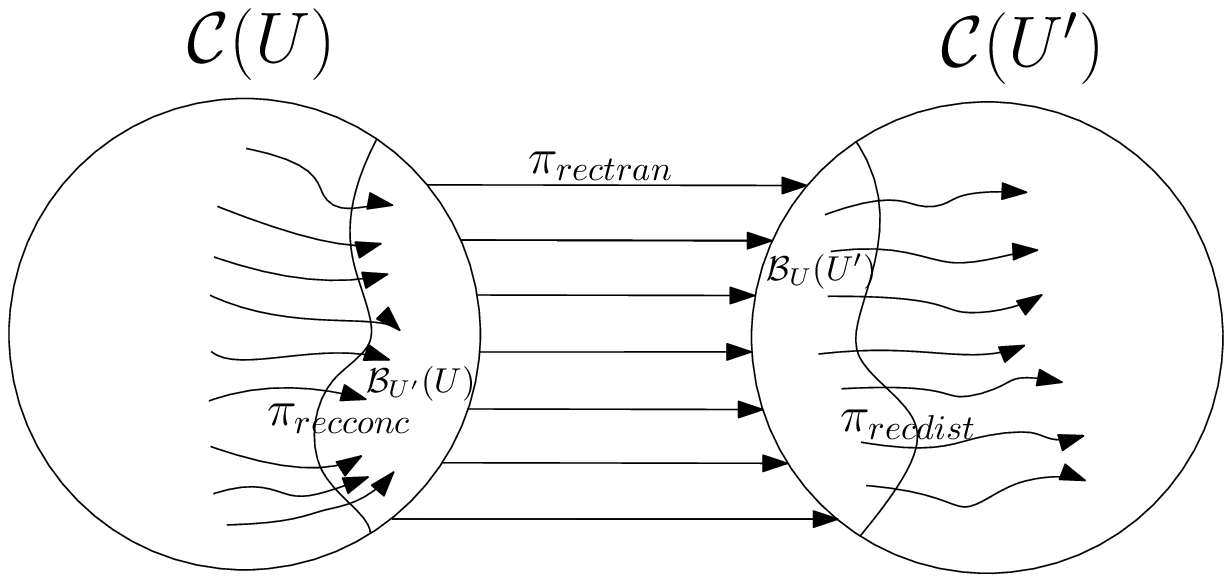}
\caption{Left: An illustration of~$\pi_{rec}$, to which we reduce~$\pi_{dist}$ in Lemma~\ref{lem:reducrec}, in which~$\classt{}{U}$ must distribute its flow throughout~$\mathcal{M}_i$,  inducing a corresponding distribution of flow from~$\protect\bdryt{}{T'}{T}$ throughout~$\protect\classt{}{T}$, by the isomorphism in Condition~\ref{addcondbdry}.
\\\\Right: a decomposition of the flow~$\pi_{rec,U,U'}$ from Lemma~\ref{lem:recuu}, which decomposes into~$\pi_{conc},\pi_{tran},\pi_{dist}$, which are similar to~$\protect\pi_{conc},\pi_{tran},\pi_{dist}$ and thus admit a recursive decomposition (Lemma~\ref{lem:recuucong}).}
\label{fig:fwaddrec}
\end{figure}

Intuitively, $\pi_{shuf}$ describes the problem of ``shuffling,'' or distributing evenly throughout~$\classt{}{T}$, the flow that~$s$ must send to vertices in~$\classt{}{T'}$. We solve this subproblem in aggregate for every~$s \in \classt{}{T}$ by applying the inductive hypothesis and Lemma~\ref{lem:cartflow}, obtaining a uniform multicommodity flow~$f_T$ in~$\classt{}{T}$ with combined congestion at most~$\rho$. We then let~$f_{shuf} = f_{s,shuf}$ be the part of~$f_T$ that sends flow just for~$s$\textemdash since~$f_T$ can be written as a sum~$\sum_{s\in\classt{}{T}} f_{s,shuf},$ where~$f_s = \sum_{s'\in\classt{}{T}} f_{s,s'}$, where~$f_{s,s'}$ is the single-commodity flow function as described in Definition~\ref{def:mcflow}.

Thus we prove the following:
\begin{lemma}
\label{lem:shufcong}
The MSF subproblem~$\pi_{shuf}$ as defined in this section for any two classes~$\classt{}{T}, \classt{}{T'} \in \mathcal{S}_n$, with~$\mathcal{S}_n$ partitioning~$\mathcal{M}_n \in \mathcal{F}$, $n > 1$, with~$\mathcal{F}$ satisfying the conditions of Lemma~\ref{lem:fwstrong}, can be solved in aggregate for all~$s\in\classt{}{T}$ and for all~$\classt{}{T'}\neq\classt{}{T}$, while generating at most congestion~$\rho$\textemdash where~$\rho$ is as in the statement of Lemma~\ref{lem:fwstrongindstep}.
\end{lemma}
\begin{proof}
As in the discussion leading to this lemma, the uniform multicommodity flow~$f_T$ in~$\classt{}{T}$ given by the application of the inductive hypothesis and Lemma~\ref{lem:cartflow} has congestion at most~$\rho$. More precisely, in this uniform multicommodity flow, the \emph{un-normalized} congestion, as in Definition~\ref{def:mcflow}, is at most~$\rho|\classt{}{T}|$. Under the definition of~$\sigma_{shuf} = |\classt{}{T'}|$, and summing over all~$s$ and over all~$\classt{}{T'}$, what we in fact need is a \emph{scaled} version of~$f_T$\textemdash in which the amount of flow sent between each pair of vertices~$s,s'\in\classt{}{T}$, and therefore the overall congestion across each edge within~$\classt{}{T}$, is scaled so that each $s$ sends to each $s'$ 
$$\frac{|\mathcal{V}_n|}{|\classt{}{T}|}$$
units of flow, instead of just one unit.

Thus we increase the un-normalized congestion from~$\rho|\classt{}{T}|$ to~$\rho|\mathcal{V}_n|$. However, since we are now considering congestion within the graph~$\mathcal{M}_n$ instead of the induced subgraph~$\classt{}{T}$, the \emph{normalized} congestion~$\rho$ does not change.
\end{proof}

We define~$f_{tran}$\textemdash solving the problem~$\pi_{tran}$ of transmitting the flow from the boundary edges $\bdryt{}{T}{T'} \subseteq \classt{}{T}$ to~$\bdryt{}{T'}{T} \subseteq \classt{}{T'}$ in the natural way: for each directed arc~$(u,v)\in\edget{}{T}{T'}$, let~$f(u,v) = \sigma_{tran} = \delta_{tran}$. Summing the resulting flow over every~$s\in\classt{}{T}$ gives (normalized) congestion
$$\frac{1}{|\mathcal{V}_n|}|\classt{}{T}|\sigma_{tran} = \frac{|\classt{}{T}||\classt{}{T'}|}{|\edget{}{T}{T'}||\mathcal{V}_n|} \leq 1,$$
where the inequality follows from Condition~\ref{addcondmatch} of Lemma~\ref{lem:fwstrong}.

Thus we have proven:
\begin{lemma}
\label{lem:trancong}
The MSF subproblem~$\pi_{tran}$ as defined in this section for a given pair of classes~$\classt{}{T},\classt{}{T'}$ can be solved by a function~$f_{tran}$ while generating at most congestion one\textemdash when summing over all~$s\in\classt{}{T}$.
\end{lemma}

It remains to solve~$\pi_{conc}$ and~$\pi_{dist}$. We observe that these two problems are of the same form up to reversal of flows: $\pi_{conc}$ describes beginning with flow from a single commodity distributed equally throughout~$\classt{}{T}$, and ending with that flow concentrated (uniformly) within the boundary~$\bdryt{}{T}{T'}$. On the other hand, ~$\pi_{dist}$ describes just the reverse process within~$\classt{}{T'}$. We will construct~$\pi_{dist}$ within~$\classt{}{T'}$, in aggregate, for all~$s\in\classt{}{T}$; the form of this construction will easily give a symmetric construction for~$\pi_{conc}$ within~$\classt{}{T}$.

Our construction is recursive, and it is here that we use the boundary set structure in Condition~\ref{addcondbdry}: we use this condition to reduce the problem~$\pi_{dist}$ to a problem
$$\pi_{rec} = (\classt{}{U} \in \mathcal{S}_i, \mathcal{M}_i, \sigma_{rec} = \sigma_{dist} = \frac{|\classt{}{T'}|}{|\bdryt{}{T'}{T}|}, \delta_{rec} = \delta_s = 1).$$

We obtain a reduction that allows us to pass from the problem~$\pi_{dist}$ to the problem~$\pi_{rec}$: by Condition~\ref{addcondbdry}, we have that the projection of~$\bdryt{}{T'}{T}$ onto some~$\mathcal{M}_i$ in the Cartesian product~$\classt{}{T} \cong \Box_i \mathcal{M}_i$ is precisely~$\classt{}{U}$, for some~$\classt{}{U} \in \mathcal{S}_i$. Therefore, if one views~$\pi_{dist}$ as a process of distributing flow throughout~$\classt{}{T'}$, the flow is initially uniform within every copy of~$\mathcal{M}_j$, for all graphs~$\mathcal{M}_j$ in the product other than~$\mathcal{M}_i$. It therefore suffices to distribute the flow within each copy of~$\mathcal{M}_i$, in which it is initially concentrated uniformly within~$\classt{}{U}$.

Thus we prove:
\begin{lemma}
\label{lem:reducrec}
The problem~$\pi_{dist}$ described in this section can be solved by any flow function~$f_{rec}$ that solves the MSF problem~$\pi_{rec}$ as described in this section. Furthermore, if~$f_{rec}$ generates congestion at most~$\rho$, then~$f_{dist}$ also generates congestion at most~$\rho$. The problem~$\pi_{conc}$ is of the same form as the reversal of~$\pi_{dist}$ and therefore is solved by a flow function similar to~$f_{rec}$, also with congestion at most~$\rho$.
\end{lemma}
\begin{proof}
The first part of the lemma statement\textemdash the reduction\textemdash is justified by the discussion leading to this lemma. That is, we can easily construct a flow function~$f_{dist}$ that solves~$\pi_{dist}$ as the arc-wise sum of many separate (but identical) functions~$f_{rec}$\textemdash  one such function within each copy of~$\mathcal{M}_i$ in the Cartesian product~$\classt{}{T'}$.

The preservation of the congestion bound~$\rho$ follows from the fact that these functions are defined over disjoint sets of arcs, since the copies of~$\mathcal{M}_i$ are all mutually disjoint.

Finally, the symmetry of~$\pi_{conc}$ and~$\pi_{dist}$ follows from the discussion leading to this lemma.
\end{proof}

Furthermore, notice that in~$\pi_{rec}$, we have the problem of flow that is initially concentrated uniformly within a class~$\classt{}{U} \in \mathcal{S}_i$, such that an equal amount must be distributed to each vertex~$t\in\classt{}{U'}$, for every class~$\classt{}{U'} \in \mathcal{S}_i$. Let~$\pi_{rec,U,U'}$ be this problem of sending the flow that is bound for vertices in~$\classt{}{U'}$. We now have:
\begin{lemma}
\label{lem:recuu}
The problem~$\pi_{rec}$, defined with respect to~$s \in \classt{}{T}$ and~$\classt{}{U}\in\mathcal{M}_i$, can be decomposed into a collection of problems~$\pi_{rec,U,U'}$, one for each~$\classt{}{U'} \in \mathcal{S}_i$.
\end{lemma}
\begin{proof}
Following the discussion leading to this lemma, it suffices to define
$$\pi_{rec,U,U'} = (\classt{}{U}, \classt{}{U'}, \sigma_{rec,U,U'} = \sigma_{rec}\frac{|\classt{}{U'}|}{|\mathcal{V}_i|}, \delta_{rec,U,U'} = \delta_{rec}).$$
The definitions of~$\sigma_{rec,U,U'}$ and~$\delta_{rec,U,U'}$ are indeed correct (achieve the decomposition of~$\pi_{rec}$ stated in the lemma):~$\delta_{rec,U,U'} = \delta_{rec}$ obviously agrees with~$\pi_{rec}$, and one can check that
$$\sum_{U'}\sigma_{rec,U,U'} = \sigma_{rec},$$
as needed.
\end{proof}

Furthermore, since~$\mathcal{M}_i$ is in the family~$\mathcal{F}$ and thus satisfies the conditions of Lemma~\ref{lem:fwstrong}, the problem~$\pi_{rec,U,U'}$ is of the same form as our original problem (~$\pi_{conc}, \pi_{tran}, \pi_{dist}$) of sending flow that was uniformly concentrated within~$\classt{}{T}$ to vertices in~$\classt{}{T'}$, where $\classt{}{T},\classt{}{T'}\in\mathcal{S}_n$ were classes in the original graph~$\mathcal{M}_n$.

That is, just as we decomposed the original problem~$\pi_s$ into the ``concentration'' problem~$\pi_{conc}$, the ``transmission'' problem $\pi_{tran},$ and the ``distribution'' problem~$\pi_{dist}$, we can recursively decompose~$\pi_{rec,U,U'}$ in the same fashion. In particular, we can solve the resulting transmission problem, in the same fashion as before. Furthermore, recall that the original problem~$\pi_s$ was defined with respect to a single~$s\in\classt{}{T}$. We claim that even after solving the transmission problem for \emph{all}~$s\in\classt{}{T}$, we obtain congestion at most one.

To see why this is, note first that:
\begin{remark}
\label{rmk:noconginc}
Summing~$\sigma_{dist}$ over all~$s\in\classt{}{T}$ produces
$$|\classt{n}{T}|\sigma_{dist} = \frac{|\classt{n}{T}||\classt{n}{T'}|}{|\bdryt{n}{T}{T'}}| \leq |\mathcal{V}_n|$$ flow ``concentrated'' within each boundary vertex.
\end{remark}

These facts, we claim, indicate that the congestion does not increase as we pass from one level of recursion to the next. 
Remark~\ref{rmk:noconginc} implies that in this reduction, we have within~$\mathcal{M}_i$ a problem similar to the original problem in~$\mathcal{M}_n$: that is, in the original problem, the overall flow construction, we have a collection of MSF problems~$\{\{\pi_s\} | s \in \mathcal{V}_n\}$, in which each~$s\in\mathcal{V}_n$ is ``charged'' with initial surplus values~$\sum_{T'} |\classt{}{T'}| = |\mathcal{V}_n|$. What we have now is a single MSF problem, in~$\mathcal{M}_i$, in which each~$u\in\bdryt{}{T'}{T} \cap \mathcal{M}_i = \classt{}{U}$ has a surplus (summing over all~$s\in\classt{}{T}$) of~$|\classt{n}{T}|\sigma_{dist} \leq |\mathcal{V}_n|$, by Remark~\ref{rmk:noconginc}. Furthermore, just as the original problem of distributing $|\classt{}{T}||\mathcal{V}_n|$ outbound flow from vertices~$s\in\classt{}{T}$ throughout~$\mathcal{M}_n$ induces the subproblem of sending~$|\classt{}{T}||\classt{}{T'}|$ flow from~$\classt{}{T}$ to~$\classt{}{T'}$ (and thus by Condition~\ref{addcondmatch} producing~$\leq|\mathcal{V}_n|$ flow across each~$\edget{}{T}{T'}$ edge), similarly the subproblem of distributing~$|\mathcal{V}_n|$ flow from each~$u\in\classt{}{T}$ throughout~$\mathcal{M}_i$ induces the subproblem of sending
$$|\classt{}{U}||\mathcal{V}_n|\frac{|\classt{}{U'}|}{|V(\mathcal{M}_i)|}$$
flow from~$\classt{}{U}$ to each~$\classt{}{U'}$ in~$\mathcal{M}_i$, since each~$\classt{}{U'}$ receives a portion of the $|\classt{}{U}||\mathcal{V}_n|$ outbound flow from~$\classt{}{U}$ that is proportional to the cardinality of~$\classt{}{U'}$ within~$V(\mathcal{M}_i)$. This generates at most
$$\frac{|\classt{}{U}||\mathcal{V}_n|\frac{|\classt{}{U'}|}{|V(\mathcal{M}_i)|}}{|\edget{}{U}{U'}} \leq \frac{|V(\mathcal{M}_i)||\mathcal{V}_n|}{|V(\mathcal{M}_i)|} = |\mathcal{V}_n|$$
flow across the matching edges~$\edget{}{U}{U'}$, producing (normalized) congestion one, and matching the flow across~$\edget{}{T}{T'}$. (Here, in the first inequality, we have applied Condition~\ref{addcondmatch} to the matching~$\edget{}{U}{U'}$.) Thus we have a recursive decomposition in which the congestion does not increase in the recursion.

\begin{lemma}
\label{lem:recuucong}
Let the problem~$\pi_{rec,U,U'}$ be defined as in this section, with respect to~$s\in\classt{}{T}$, class~$\classt{}{U},\classt{}{U'}$ being classes in~$\mathcal{M}_i$, with~$\mathcal{M}_i$ a graph in the Cartesian product~$\classt{}{T}$.

Then~$\pi_{rec,U,U'}$ can be recursively decomposed into~$\pi_{recconc}, \pi_{rectran},$ and~$\pi_{recdist}$, with each problem solved by a respective flow~$f_{recconc}, f_{rectran}, f_{recdist},$ such that:

\begin{enumerate}[(i)]
\item The sum total congestion incurred by all of the~$f_{rectran}$ subproblems induced by all~$s\in\classt{}{T}$, is at most one, and
\item ~$\pi_{reccconc}$ and~$\pi_{recdist}$ are similar to the problems~$\pi_{dist}$ and~$\pi_{conc}$ described in this section and thus admit a recursive decomposition as in Lemma~\ref{lem:reducrec}, and
\item the demand~$\delta_{recconc}$ is upper-bounded by~$\delta_{conc}$, the surplus value in the original concentration problem~$\pi_{conc}$; similarly, $\sigma_{recdist} \leq \sigma_{dist}$.
\end{enumerate}
\end{lemma}
\begin{proof}
We prove~(ii) first: define
$$\pi_{recconc} = (\classt{}{U}, \bdryt{}{U}{U'}, \sigma_{recconc} = \sigma_{rec}, \delta_{recconc} = \sigma_{recconc}\frac{|\classt{}{U}|}{|\bdryt{}{U}{U'}|}),$$
$$\pi_{rectran} = (\bdryt{}{U}{U'}, \bdryt{}{U'}{U}, \sigma_{rectran} = \delta_{rectran} = \delta_{recconc}),$$
$$\pi_{recdist} = (\bdryt{}{U'}{U}, \classt{}{U'}, \sigma_{recdist} = \delta_{rectran}, \delta_{recdist} = \sigma_{recdist}\frac{|\bdryt{}{U'}{U}|}{|\classt{}{U'}|}).$$

Comparing source and sink sets, and comparing~$\sigma$ and~$\delta$ functions shows that~$\pi_{rec}$ decomposes into~$\pi_{recconc}, \pi_{rectran}$, and~$\pi_{recdist}$. Each class~$\classt{}{U}$ and~$\classt{}{U'}$ decomposes as a Cartesian product satisfying Condition~\ref{addcondcart} in Lemma~\ref{lem:fwstrong}, and similarly the boundary sets~$\bdryt{}{U}{U'}, \bdryt{}{U'}{U}$ satisfy Condition~\ref{addcondbdry}. Thus exactly the same form of decomposition used to reduce the original~$\pi_{dist}$ and~$\pi_{conc}$ to~$\pi_{rec}$ also works for~$\pi_{recconc}$ and~$\pi_{recdist}$. We can thus recursively construct~$f_{recconc}$ and~$f_{recdist}$, proving~(ii).

For~(i), we need to define~$f_{rectran}$ and to bound the resulting congestion.

Define~$f_{rectran}$ in the same natural way we defined~$f_{tran}$: simply assign~$\sigma_{rectran} = \delta_{rectran}$ to each arc.

We observe that
$$\sigma_{recconc} = \sigma_{rec,U,U'} = \sigma_{rec}\frac{|\classt{}{U'}|}{|\mathcal{V}_i|} = \sigma_{dist}\frac{|\classt{}{U'}|}{|\mathcal{V}_i|} =  \frac{|\classt{}{T'}|}{|\bdryt{}{T'}{T}|}\cdot \frac{|\classt{}{U'}|}{|\mathcal{V}_i|},$$
by the definitions of the MSF problems we have given in this section. It is easy to see also that
$$\sigma_{rectran} = \delta_{rectran} = \delta_{recconc} = \sigma_{recconc}\frac{|\classt{}{U}|}{|\bdryt{}{U}{U'}|}.$$

Combining these facts gives
$$\sigma_{rectran} = \frac{|\classt{}{T'}|}{|\bdryt{}{T'}{T}|}\cdot \frac{|\classt{}{U'}|}{|\mathcal{V}_i|}\cdot \frac{|\classt{}{U}|}{|\bdryt{}{U}{U'}|} \leq \frac{|\classt{}{T'}|}{|\edget{}{T'}{T}|} = \sigma_{tran},$$

where the inequality follows from the fact that the matching~$\edget{}{U}{U'}$ satisfies Condition~\ref{addcondmatch} of Lemma~\ref{lem:fwstrong}.

Now, to obtain the un-normalized congestion~$\bar \rho_{rectran}$ that results from~$f_{rectran}$, we sum over all~$s\in\classt{}{T}$, scaling the above quantity by a factor of~$|\classt{}{T}|$, giving 
$$\bar\rho_{rectran} = |\classt{}{T}|\sigma_{rectran} = |\classt{}{T}|\sigma_{tran} = \frac{|\classt{}{T}||\classt{}{T'
}|}{|\edget{}{T}{T'}|} \leq |\mathcal{V}_n|,$$
where we have again applied Condition~\ref{addcondmatch} of Lemma~\ref{lem:fwstrong}.

Thus we obtain \emph{normalized} congestion at most $\frac{|\mathcal{V}_n|}{|\mathcal{V}_n|} \leq 1$, proving~(i).




For~(iii), claim~(i) also implies that the congestion does not increase in the recursive decomposition given by~(ii)\textemdash that is, passing from~$\pi_{dist}$, to~$\pi_{rec}$, to~$\pi_{rec,U,U'}$, to~$\pi_{recdist}$, preserves the bound
$$\sigma_{recdist} \leq \sigma_{dist}.$$
The analogous fact for~$\sigma_{recconc}$ is symmetric.
\end{proof}

We now have all the pieces we need to prove Lemma~\ref{lem:fwstrongindstep}:
\lemfwstrongindstep*
\begin{proof}
To construct the desired uniform multicommodity flow, it suffices to construct, for every~$\classt{}{T},\classt{}{T'} \in \mathcal{S}_n$ and for every~$s\in\classt{}{T}$, the flow~$f_s$ solving the MSF problem~$\pi_s$. As shown in this section, $\pi_s$ decomposes (Remark~\ref{rmk:composemsf}) as the subproblems~$\pi_{shuf},\pi_{conc},\pi_{tran},$ and~$\pi_{dist}$.

For~$\pi_{shuf}$, summing over all~$s\in\classt{}{T}$ and over all~$s\in\classt{}{T'}$, the sum of the~$f_{shuf}$ flows given by the inductive hypothesis and the Cartesian flow structure (Lemma~\ref{lem:cartflow}) of~$\classt{}{T}$ gives congestion at most~$\rho$, by Lemma~\ref{lem:shufcong}.

For a given~$\classt{}{T},\classt{}{T'}$ pair, again summing over all~$s\in\classt{}{T'}$, we obtain flows $f_{tran}$ for~$\pi_{tran}$ whose sum is congestion one, by Lemma~\ref{lem:trancong}.

Dividing~$\pi_{dist}$ (and symmetrically~$\pi_{conc}$) into copies of the~$\pi_{rec}$ problem as in Lemma~\ref{lem:reducrec}, and further dividing each~$\pi_{rec}$ into problems~$\pi_{rec,U,U'}$ (by Lemma~\ref{lem:recuu}), each of which we further divide into~$\pi_{recconc}, \pi_{rectran},$ and~$\pi_{recdist}$. Furthermore, by Lemma~\ref{lem:recuucong}, these subproblems are of the same form as~$\pi_{conc}, \pi_{tran},$ and~$\pi_{dist}$, with the natural solution~$f_{rectran}$ to the ``transmission'' problem~$\pi_{tran}$ being of the same form as~$f_{tran}$ and producing, like~$f_{tran}$, overall congestion one after summing over all~$s\in\classt{}{T}$.

We then recursively decompose~$\pi_{recconc}$ and~$\pi_{recdist}$ in the same fashion as we did~$\pi_{conc}$ and~$\pi_{dist}$, with, by Lemma~\ref{lem:recuucong}, congestion one in the transmission problems at each level of recursion. Since all flow produced by solving the subproblems in this decomposition is counted by the transmission flows, and since (it is easy to see) each arc occurs in only one such transmission flow, we obtain overall congestion one for~$\pi_{rec,U,U'}$.

Recall that~$\pi_{rec,U,U'}$ is defined with respect to a given~$\classt{}{T},\classt{}{T'}$ pair, where~$\classt{}{U}$ is determined by~$\classt{}{T}$, as a class within the graph~$\mathcal{M}_i$, within the Cartesian product~$\classt{}{T} \cong \Box_j \mathcal{M}_j$. Thus we must sum this bound of congestion one for~$f_{rec,U,U'}$ over all~$\classt{}{U'} \in\mathcal{S}_i$. By assumption~$|\mathcal{S}_i| \leq \kappa$, so we obtain~$\kappa$ flows each with congestion one, giving overall congestion at most~$\kappa$.

One may worry that the~$\kappa^2$ pairs of classes exchanging flow may produce~$\kappa^2$ congestion, since we do obtain~$\kappa^2$ subproblems. Fortunately, we can justify the~$\kappa$ bound as follows: consider~$\kappa$ MSF problems instead of~$\kappa^2$ problems. In each of the~$\kappa$ MSF problems, a given class~$\classt{}{T}$ must send flow to~\emph{all} other classes. This introduces some asymmetry, as the concentration flow within~$\classt{}{T}$ involves only a single commodity, while the distribution flow within~$\classt{}{T}$ involves~$\kappa-1$ commodities. Thus we can easily break this distribution flow into~$\kappa-1$ recursive distribution flows that each involve a single commodity distributed throughout~$\classt{}{T}$ from~$\bdryt{}{T}{T'}$ for some~$\classt{}{T'}$.

The concentration flow takes slightly more work: it involves a single commodity but induces a subproblem in which every pair of subclasses within~$\classt{}{T}$ must exchange a unit of flow. Consider the boundary sets~$\bdryt{}{T}{T'}$ and~$\bdryt{}{T}{T''}$ along which~$\classt{}{T}$ must send flow to any two of the other classes~$\classt{}{T'}$ and~$\classt{}{T''}$. By Condition~\ref{addcondbdry}, we know that all of this flow occurs between subclasses within copies of smaller flip graphs. Say these subclasses are~$\classt{}{U'}$ and~$\classt{}{U''}$. Notice that we do not need to send flow in both directions, \emph{because we have only a single commodity}. Only the \emph{amount} of flow sent matters. This observation gives us a convenient subproblem in which for each pair of subclasses~$\classt{}{U'}, \classt{}{U''}$, one class sends to the other an amount of flow that, by Condition~\ref{addcondmatch}, generates congestion at most one, producing appropriate recursive subproblems without an increase in congestion.
\end{proof}

Lemma~\ref{lem:fwstrongindstep} forms the inductive step of Lemma~\ref{lem:fwstrong} (with a trivial base case), and thus we have proven Lemma~\ref{lem:fwstrong}.

\section{Nearly tight conductance for triangulations: lower bound}
\label{sec:combinedec}
Lemma~\ref{lem:triangnewcond} and Lemma~\ref{lem:fwstrong}, as we showed in Appendix~\ref{sec:assocexplb}, imply the known result that the flip walk on triangulations of the convex polygon mixes rapidly. However, the bound given by Lemma~\ref{lem:triangnewcond} is $O(n^2)$ congestion, giving~$O(n^7)$ mixing time by Lemma~\ref{lem:expmixing}.
Through a more careful flow construction, one can further improve this bound to~$O(n^{3}\log^3 n)$.
For the more careful construction, we will define a different decomposition, via the \emph{central triangle}:
\begin{definition}
\label{def:centralpart}
Given a triangle~$T$ containing the center of the regular $n+2$-gon $P_{n+2}$ and sharing all of its vertices with $P_{n+2}$, identify~$T$ with the class~$\classt{}{T}$ of triangulations $t \in V(K_n)$ such that $T$ forms one of the triangles in~$t$. Let~$\mathcal{S}_{n}$ be the set of all such~$\classt{}{T}$ classes. 
\end{definition}
(If $P_{n+2}$ has an even number of edges, we perturb the center slightly so that every triangulation lies in some class.) 

\begin{remark}
The set $\mathcal{S}_{n}$ is a partition of $V(K_n)$, because no pair of triangles whose endpoints are polygon vertices can contain the origin without crossing.
\end{remark}
Molloy, Reed, and Steiger~\cite{molloylb} defined this same partition in their work.

See Figure~\ref{fig:centralclasses}.

We will combine this central-triangle decomposition with the oriented decomposition we defined earlier. What we gain from using the central-triangle decomposition is that the number of levels of induction will now be~$O(\log n)$, by the fact that using the central triangle to partition the classes divides the~$n+2$-gon into smaller polygons of size~$\leq n/2$. What we \emph{lose}, however, is that we no longer have matchings between every pair of classes, nor are all of the matchings between adjacent pairs sufficiently large to obtain a polynomial bound. Thus if we were to use just this decomposition on its own, we would be stuck with the quasipolynomial bound (which in fact is what we obtain for general $k$-angulations in Appendix~\ref{sec:quadmix}).

Fortunately, we will show how to combine the two decompositions. With suitable care, this will allow us to eliminate one of the factors of~$n$\textemdash which we incurred in the~$n$ levels of induction via Lemma~\ref{lem:fwstrong}. Some further optimizations will give us the claimed congestion bound~$O(\sqrt n)$:

\begin{restatable}{lemma}{lemflowind}
\label{lem:flowind}
Suppose that for all $1 \leq i \leq n/2,$ a uniform multicommodity flow exists with congestion $O(\sqrt i \log i)$ in $K_i.$ Then a uniform multicommodity flow exists in $K_n$ with congestion $O(\sqrt n \log n).$
\end{restatable}

(Here of course the constant hidden in the $O$ notation is independent of the number of induction levels.) Once we prove this lemma, then clearly Theorem~\ref{thm:triangmixub} follows via simple induction and an application of Lemma~\ref{lem:expmixing}. (The base case in the induction is trivial.)

What we will do is, before  routing the flow between triangulations in two different classes, to do the same shuffling step as before\textemdash this time within each class in the \emph{central-triangle} partition, instead of within each class in the \emph{oriented} partition. It is easy to see that this is simply a ``scaled-up'' uniform multicommodity flow in each class and produces no increase in congestion, using the same analysis as before. We then have an MSF problem for each pair of classes, in which the flow is routed through a set of intermediate classes, and:

\begin{remark}
\label{rmk:bdryoriented}
The boundary set~$\bdryt{n}{T}{T'}$ between every pair of central-triangle-induced classes is isomorphic to a Cartesian product~$\classtt{n}{U} \Box K_j \Box K_l$, where~$\classtt{n}{U}$ is an \emph{oriented} class in~$K_i$, and where~$\classt{n}{T} \cong K_i \Box K_j \Box K_l$, $i + j + l = n - 2$, $i,j,l\leq n/2$.
\end{remark}

In other words, even though we are now using the \emph{central-triangle} decomposition, our boundary classes are, as before, Cartesian products of \emph{oriented} classes with associahedron graphs. Therefore:

Remark~\ref{rmk:bdryoriented}, combined with our earlier congestion analysis for concentration and distribution flows, implies the following:

\begin{lemma}
\label{lem:combinemsfcong}
Suppose it is possible to construct a multicommodity flow~$f$ in~$K_n$ in which the total congestion across edges between a pair of classes is at most~$\bar\rho$. Then the total congestion produced by~$f$ is at most~$2\bar\rho n$.
\end{lemma}
\begin{proof}
Routing flow through an intermediate class~$\classt{n}{T''},$ say, that originates at class~$\classt{n}{T}$ and is bound for~$\classt{n}{T'}$, can be accomplished with the combination of a uniform multicommodity flow in~$\classt{n}{T}$ (a ``shuffling flow''), scaled as in the construction in Appendix~\ref{sec:flowdetails}, and an MSF with source set~$\classt{n}{T}$ and sink set~$\classt{n}{T''}$. This MSF then induces in~$\classt{n}{T''}$ both a concentration flow and a distribution flow. Notice that~$\classt{n}{T''}$ has~$O(n)$ distinct neighboring classes. Therefore there are~$O(n)$ such concentration flows and~$O(n)$ such distribution flows. Since each concentration flow and each distribution flow produces no increase in congestion relative to the amount across each edge between classes, the claim follows.
\end{proof}

\label{sec:optimizing}
We will use this idea of combining decompositions to obtain our~$O(\sqrt n\log n)$ congestion bound. We will exhibit a flow with~$\bar\rho = O(\sqrt n)\log n$, and will show how to avoid the~$O(n)$ gain in Lemma~\ref{lem:combinemsfcong}.
\begin{figure}[h]
\includegraphics[width=20em]{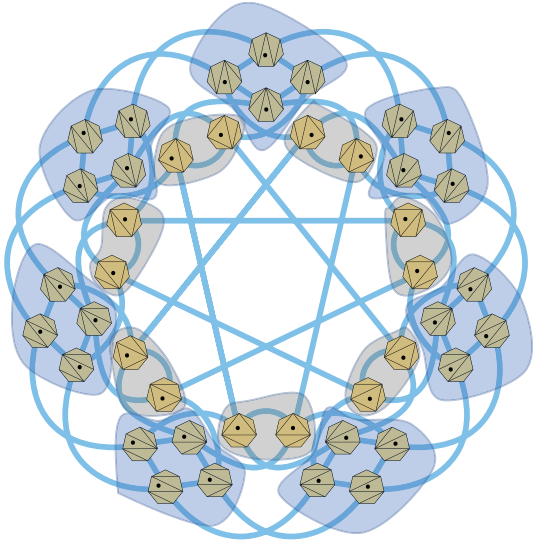}
\caption{An alternative partitioning of the associahedron graph $K_5$, with each vertex representing a triangulation of the regular heptagon. Flips are shown with edges (in blue). The vertex set~$V(K_n)$ is partitioned into a set~$\mathcal{S}_n$ of equivalence classes. Within each class, all triangulations share the same \emph{central triangle}\textemdash contrast with the oriented partition depicted in Figure~\ref{fig:heptclassesor}.}
\label{fig:centralclasses}
\end{figure}

We do so by choosing carefully a good set of ``paths'' between each pair of classes, where each path consists of a sequence of intermediate classes through which to route flow. One first attempt might be, given classes $\classt{n}{S}$ and $\classt{n}{U}$, to consider the $P_{n+2}$ vertices by which $S$ and $U$ differ. Route flow from $\classt{n}{S}$ to $\classt{n}{S'}$, where $S'$ is a triangle formed by replacing some vertex of $S \setminus U$ with a vertex of $U \setminus S$. It is easy to see that this results in routing flow through at most two intermediate classes. Unfortunately, if we do this for all $S,U$ pairs, then some of these intermediate classes will end up routing flow for too many $S,U$ pairs, and the congestion improvement will be insufficient for our purposes.

Roughly speaking, and perhaps counterintuitively, it turns out that the large congestion under the scheme described above results from using paths that are too short: there exist many pairs of large classes $\classt{n}{S}, \classt{n}{U}$ such that the intermediate classes found under this scheme are much smaller than $\classt{n}{S}$ and $\classt{n}{U}$, and thus cannot effectively ``spread out'' the congestion that results from sending the $\classt{n}{S}, \classt{n}{U}$ flow. Instead we will find slightly longer paths.

To define these longer paths, we first need to organize the classes into a (non-disjoint) union of larger classes, which we call \emph{regions}:
\begin{definition}
\label{def:regions}
Mark 24 equally spaced points on the convex polygon $P_{n+2}$, in counterclockwise order. Define the following 24 \emph{regions} as collections of the classes: let $\mathcal{U}_i$, $0 \leq i \leq 23$, be the set of all classes $\classt{n}{T}$ such that the vertex \emph{opposite the shortest edge} of~$T$ lies in the (inclusive) interval $[i\cdot n/24, (i+2)\cdot n/24)$.
\end{definition}

This is not a partition of the classes $\{\classt{n}{T}\}$, since it is possible for a triangle $T$ to have a vertex in two of the intervals described. However:

\begin{remark}
\label{rmk:equalpart}
It is easy to see that all of the classes in Definition~\ref{def:regions} are of equal size, and that the regions each have cardinality $C_n/12.$ Also, the regions form a cycle, in which each consecutive pair of regions shares an overlap of size~$C_n/24$.
\end{remark}

The idea now is that we will establish the existence of a flow, within each of the 24 regions, that has congestion $O(\sqrt{n/2})$. Once this is accomplished, we will use the constant-factor intersections of the classes to route flow between classes with additional (additive) congestion $O(\sqrt n)$. 

We will further partition the central-triangle-induced classes into two additional levels of classes. 

\begin{definition}Given a central triangle~$T$, let the \emph{apex} of~$T$ be the vertex of~$T$ opposite the shortest side. (If~$T$ has no unique shortest side, break ties in some arbitrary fashion; this will not change the asymptotics.) Let the \emph{second} vertex of~$T$ be the first of the two non-apex vertices that succeeds the apex in counterclockwise order; let the remaining vertex be the \emph{third} vertex. Given a vertex~$p \in [0, n+1]$ of the~$n+2$-gon, with the vertices labeled in counterclockwise order, let~$\mathcal{A}_p$ be the set of classes~$\classt{n}{T}$ with~$p$ as the apex of~$T$. Let~$\mathcal{L}_{pq}$ be the set of classes~$\classt{n}{T}$ in~$\mathcal{A}_p$ with~$q$ as the second vertex of~$T$. 
\end{definition}

\begin{remark}
\label{rmk:apqoriented}
Every~$\mathcal{L}_{pq}$ is a collection of central-triangle classes sharing an edge, namely the diagonal~$pq$. Every class~$\classt{n}{T} \in \mathcal{L}_{pq}$ is, therefore, a Cartesian product~$K_{q-p-1} \Box \classtt{n}{T}$ of a flip graph~$K_{q-p-1}$ over the $q-p+1$-gon on one side of the diagonal~$pq$, and a class~$\classtt{n}{T} \subseteq V(K_{n-q+p-1})$ in the \emph{oriented} partition induced by the edge~$pq$ in the $n-q+p+1$-gon on the other side of~$pq$.
\end{remark}

\begin{lemma}
\label{lem:sqrtwithinsecond}
Within~$\mathcal{L}_{pq}$, for all~$p\in[0, n+1]$, it is possible to route a unit of flow between every ordered pair of triangulations~$t,t' \in \mathcal{L}_{pq}$ while producing congestion one across the edges between any pair of central-triangle-induced classes.
\end{lemma}
\begin{proof}
By Remark~\ref{rmk:apqoriented}, we can apply Lemma~\ref{lem:matchingscard} to conclude that every pair of classes~$\classt{n}{T}, \classt{n}{T'} \in \mathcal{L}_{pq}$ can exchange one unit of a commodity with congestion at most
$$\frac{|\classt{n}{T}||\classt{n}{T'}|}{|\edget{n}{T}{T'}||V(K_n)|} \leq \frac{|V(K_{q-p-1}) \Box \classtt{q-p-1}{T}||V(K_{q-p-1}) \Box \classtt{q-p-1}{T'}|}{|V(K_{q-p-1}) \Box \edgett{q-p-1}{T}{T'}||V(K_n)|} \leq \frac{|\classtt{q-p-1}{T}||\classtt{q-p-1}{T'}|}{|\edgett{q-p-1}{T}{T'}|}\cdot \frac{C_{q-p-1}}{C_n},$$
with the first expression describing the product of the cardinalities of the classes divided by the number of edges between them, and a normalization factor~$|V(K_n)| = C_n$ according to the definition of congestion. The first inequality comes from Remark~\ref{rmk:apqoriented}, and the second from rearranging terms. Applying Lemma~\ref{lem:matchingscard} gives an upper bound of
$$\frac{C_{n-q+p-q}C_{q-p-1}}{C_n} \leq 1.$$ 
(Actually, this quantity is not only at most one but at most~$\frac{C_{q-p-1}C_{n-q+p}}{C_n} = O(1/(q-p)^{3/2})$.)
\end{proof}

Lemma~\ref{lem:sqrtwithinsecond} describes only the flow across edges \emph{between pairs of central-triangle classes}. The construction in Appendix~\ref{sec:flowdetails} shows how to obtain polynomial congestion \emph{within} classes from this bound. However, as we observed in the proof of Lemma~\ref{lem:fwstrongindstep}, one suffers a loss accounting for flow from~$\kappa = O(n)$ classes. We now show how to improve this~$O(n)$ loss to~$O(\sqrt n)$.

It is easy to see that, for the purpose of analyzing congestion, a uniform multicommodity flow in~$K_n$ is equivalent to the sum of~$|V(K_n)| = C_n$ \emph{single}-commodity flows, one ``originating'' at (having sink set as) a single triangulation~$t \in V(K_n)$. Furthermore, given the \emph{oriented} partition~$\mathcal{S}_n$, and considering the classes~$\classtt{n}{T_1}, \classtt{n}{T_2}, \dots, \classtt{n}{T_n} \in \mathcal{S}_n$, a uniform multicommodity flow in~$K_n$ is equivalent to the sum of~$n$ uniform multicommodity flows, one \emph{within} each class~$\classtt{n}{T_i}$, \emph{added} to~$n$ \emph{multi-way single-commodity flows} (MSFs), each of which distributes flow from one class~$\classtt{n}{T_i}$ to the rest of the graph~$K_n$. 

The congestion bound one obtains for the MSFs (ignoring the flows within the classes) from the analysis in the proof of Lemma~\ref{lem:fwstrongindstep} is then~$\kappa = O(n)$. The following lemma states that we can do better: we can solve these~$n$ MSF problems with congestion~$O(\sqrt n)$ by improving the flow construction. Intuitively, given two triangles~$T_i, T_j$ with third vertices~$i$ and~$j$ on the~$n+2$-gon, the size of the matching~$\edgett{n}{T_i}{T_j}$ between~$\classtt{n}{T_i}$ and~$\classtt{n}{T_j}$ is large when~$|j - i|$ is small, and small when~$|j - i|$ is large. When~$T_i$ and~$T_j$ are far apart ($|j -i|$ is large), we will route some of the $\classtt{n}{T_i}\rightarrow\classtt{n}{T_j}$ flow through a sequence of intermediate classes~$\{\classtt{n}{T_k}\}$, $i < k < j$, taking advantage of the larger matchings between~$\classtt{n}{T_i}$ and~$\classtt{n}{T_k}$, and between~$\classtt{n}{T_k}$ and~$\classtt{n}{T_j}$.

In particular, we will first group the classes into pairs of consecutive classes~$\classtt{n}{T_i}, \classtt{n}{T_{i+1}}$ (with, say, $i = 0 \pmod 2$), and let the two classes within a given pair exchange flow, so that all of the flow originating at either class in the pair is uniformly distributed throughout the pair~$\classtt{n}{T_i} \cup \classtt{n}{T_{i+1}}$. That way, subsequent flow sent by the two classes can now be considered as a single commodity. We will then group these pairs of classes into sets of four classes, then sets of eight, and so on\textemdash reaching~$O(\log n)$ hierarchical levels of sets, until all~$n(n-1)$ ordered pairs of classes have exchanged flow.

\begin{lemma}
\label{lem:logtrick}
Given the flip graph~$K_n$ over the $n+2$-gon and the special edge~$e^*$, consider the triangles~$T_1, T_2, \dots, T_n$ that include~$e^*$ as an edge, such that $T_1, T_2, \dots, T_n$ occur in consecutive order according to their third vertex. Consider the~$n$ MSF problems~$\pi_1,\pi_2, \dots, \pi_n$, one for each \emph{oriented} class~$\classtt{n}{T_i}$, $i = 1, 2, \dots, n$. Suppose each~$\pi_i$ has source set~$\classtt{n}{T_i}$ and sink set~$V(K_n)$, with uniform surplus and demand functions~$\sigma_i = |V(K_n)| = C_n, \delta_i = |\classtt{n}{T_i}|$. Then ~$\pi_1, \pi_2, \dots, \pi_n$ can be reduced to an alternative collection of MSF problems that can be solved with congestion~$O(\sqrt n)$.
\end{lemma}
\begin{proof}
Assume for simplicity that~$n$ is a power of two; it is easy to modify the solution if not. Group the classes hierarchically as described in the disussion preceding this lemma. Let~$\pi_{[i, j]}$ be the problem, defined over the subgraph of~$K_n$ induced by the classes~$\classtt{n}{T_i} \cup \cdots \cup \classtt{n}{T_j}$, of distributing flow from the ``left half'' of the classes~$\classtt{n}{T_i} \cup \cdots \cup \classtt{n}{T_{i+\frac{j-i+1}{2}-1}}$ to the ``right half''~$\classtt{n}{T_{i+\frac{j-i+1}{2}}} \cup \cdots \cup \classtt{n}{T_j}$. Define~$\bar\pi_{[i,j]}$ symmetrically.

As discussed, the original collection of MSF problems~$\{\pi_i\}$ reduces to a collection~$\{\pi_{[i,j]}, \bar\pi_{[i,j]}\}$, where the pairs~$[i,j]$ are those induced by hierarchically partitioning the classes\textemdash first into problems~$\pi_{[1, n]}$ and~$\bar\pi_{[1, n]}$, then into (on the ``left-hand side'')~$\pi_{[1, n/2]}, \bar\pi_{[1, n/2]}$ and (on the ``right-hand side'')~$\pi_{[n/2+1, n]}, \bar\pi_{[n/2+1, n]}$, then into four pairs of problems, and so on.

Now, for a given pair of problems~$\pi_{[i, j]}, \bar\pi{[i,j]}$, each class~$\classtt{n}{T_l}$ on the ``left-hand side''~$\classtt{n}{T_i} \cup \cdots \cup \classtt{n}{T_{i+\frac{j-i+1}{2}-1}}$, $i \leq l \leq i+\frac{j-i+1}{2}-1$, must distribute~$C_n|\classtt{n}{T_l}|$ units of flow\textemdash that is, the demand~$\sigma_{[i,j]}$ times the size of the source set~$\classtt{n}{T_l}$ of~$\pi_{[i,j]}$\textemdash to the right-hand side, and vice versa. Each  class~$\classtt{n}{T_r}$ on the right-hand side receives a~$\frac{|\classtt{n}{T_r}|}{|\sum_{r'\in[i, j]}|\classtt{n}{T_{r'}}|}$ factor of this flow, and the flow must be distributed across the matching~$|\edgett{n}{T_l}{T_r}|$.

This produces congestion at most
$$O\left(\frac{C_n|\classtt{n}{T_l}||\classtt{n}{T_r}|}{|\edgett{n}{T_l}{T_r}|\sum_{r'\in[i,j]}|\classtt{n}{T_{r'}}|C_n}\right) = O\left(\frac{|\classtt{n}{T_l}||\classtt{n}{T_r}|}{|\edgett{n}{T_l}{T_r}|\sum_{r'\in[i,j]}|\classtt{n}{T_{r'}}|}\right).$$

We will bound this quantity as~$O(\sqrt{j-i})$, by showing that $\frac{|\classtt{n}{T_l}|}{|\edgett{n}{T_l}{T_r}|} = O((j-i)^{3/2})$ and that $\frac{|\classtt{n}{T_{r}}|}{|\sum_{r'\in[i,j]}|\classtt{n}{T_{r'}}|} = O(1/(j-i)).$

The first inequality is true because, by Lemma~\ref{lem:projggmatching}, $\edgett{n}{T_l}{T_r}$ is in bijection with the vertex set of a Cartesian product~$K_{l-1} \Box K_{r-l-1} \Box K_{n-r}$ graph, whereas~$\classtt{n}{T_l} \cong K_{l-1} \Box K_{n-l}$. Thus~$\frac{|\classtt{n}{T_l}|}{|\edgett{n}{T_l}{T_r}|} \leq \frac{C_{n-l}}{C_{r-l-1}{n-r}}$. We can assume without loss of generality that~$1 \leq l \leq r \leq n/2$ (since~$T_l$ and~$T_r$ send one another the same amount of flow), and therefore this quantity is at most~$O((r-l)^{3/2}) = O((j-i)^{3/2})$.

The second inequality can be seen by noticing that for all $r' \in [i, j]$, $\classtt{n}{T_{r'}} \cong K_{r'-1} \Box K_{n-r'}$, so~$|\classtt{n}{T_{r'}}| = C_{r'-1}{C_n-r'} = \frac{C_n}{\Theta({r'}^{3/2})}.$ Since this is a decreasing function of~$r'$, we have
$$\sum_{i\leq r' \leq j}|\classtt{n}{T_{r'}}| \geq \sum_{i \leq r' \leq r}|\classtt{n}{T_{r'}}| \geq (r-i+1)|\classtt{n}{T_{r}}| \geq \frac{(j-i)}{2}|\classtt{n}{T_r}|,$$
which implies certainly that~$|\classtt{n}{T_r}||\sum_{r'\in[i,j]}| = O(j-i)$.

With every MSF pair~$\pi_{[i,j]},\hat\pi_{[i,j]}$ solvable with congestion~$O(\sqrt{j-i})$, where~$j-i=1, 2, 4, 8, \dots, n$, it is easy to see that the overall congestion is~$\sum_{k=0}^{\log n}\sqrt{2^k}$ = $O(\sqrt n)$, as claimed.

Finally, one may worry that there may be a factor~$(j-i+1)/2$ gain in congestion for each~$\pi_{[i,j]},\bar\pi_{[i,j]}$ pair, since~$\classtt{n}{T_r}$ must receive flow from $(j-i+1)/2$ classes\textemdash just as we had a~$\kappa$-factor gain in the proof of Lemma~\ref{lem:fwstrongindstep}. However, that gain occurred because we had~$\kappa$ separate MSF problems. Here, however, we only have two MSF problems, inducing two flows. The same construction we used in that proof then gives $O(1)$ congestion per~$\pi_{[i,j]},\bar\pi_{[i,j]}$ pair.

\end{proof}

\begin{corollary}
\label{cor:sqrtwithinsecondtotal}
Within~$\mathcal{L}_{pq}$, for a given~$p\in[0, n+1]$, consider a collection of~$q-p-1$ MSF problems, each of which corresponds to one class~$\classt{n}{T} \in \mathcal{L}_{pq}$ and describes distributing a single commodity with surplus value~$C_n$, initially concentrated in~$\classt{n}{T}$, throughout the rest of~$\mathcal{L}_{pq}$. All of these problems can be solved while producing total congestion~$O(\sqrt{q-p})$.
\end{corollary}
\begin{proof}
It suffices to combine the constructions in Lemma~\ref{lem:sqrtwithinsecond} and Lemma~\ref{lem:logtrick}. The exchange in Lemma~\ref{lem:sqrtwithinsecond}, that is, induces MSF subproblems that can be viewed, by Remark~\ref{rmk:apqoriented}, as an exchange between pairs of~\emph{oriented} subclasses of~$K_{n-q+p}$ (in copies of~$K_{q-p-1})$\textemdash in which the surplus values are all~$C_{q-p-1}C_{n-q+p}$. We can then apply Lemma~\ref{lem:logtrick} to obtain congestion~$O(\sqrt n)\frac{C_{q-p-1}C_{n-q+p}}{C_n}$.

Here we need to be careful. First, it may be that this bound exceeds~$O(\sqrt n)$, in particular if~$\frac{C_{q-p-1}C_{n-q+p}}{C_n}\geq \omega(1/\sqrt{n})$. Fortunately, it is easy to see from the proof of Lemma~\ref{lem:logtrick} that the~$O(\sqrt n)$ bound in that lemma can be sharpened to~$O(\sqrt{q - p})$ (by noticing that at no level of the hierarchical partitioning do we have~$j - i > q - p$). Thus we have the congestion bound~$O(\sqrt{q-p})\frac{C_{q-p-1}C_{n-q+p}}{C_n}.$

Finally, we have assumed surplus values of~$C_{q-p-1}C_{n-q+p}$. Actually, however, the present claim concerns surplus values~$C_n$. Scaling by~$\frac{C_n}{C_{q-p-1}C_{n-q+p}}$ gives congestion~$O(\sqrt{q-p})$.
\end{proof}

\begin{lemma}
\label{lem:sqrtwithinapex}
Within a given~$\mathcal{A}_p$, $p\in[0, n+1]$, consider a collection of MSF problems, one for each~$\mathcal{L}_{pq} \in \mathcal{A}_p$ (with source set~$\mathcal{L}_{pq}$, with surplus values~$C_n$), with flow that must be distributed uniformly throughout~$\mathcal{A}_p$. It is possible to solve these problems while producing total congestion~$O(\sqrt{n})$.
\end{lemma}
\begin{proof}
To route flow between pairs of central-triangle classes lying in distinct second-vertex classes, i.e. between~$\classt{n}{T_{pqr}} \in \mathcal{L}_{pq},\classt{n}{T_{pq'r'}} \in \mathcal{L}_{pq'},$ $q \neq q'$, we will use the same trick as in the hierarchical grouping in Lemma~\ref{lem:logtrick}.  Assume that~$p=0$ without loss of generality and for simplicity. For all~$\mathcal{L}_{pq}, \in \mathcal{A}_p$, it holds that~$n/4\leq q < n/2$: $q \geq n/4$ since the triangle edge~$pq$ must be at least as long as the edge~$qr$ in any~$T_{pqr}, \classt{n}{T_{pqr}} \in \mathcal{L}_{pq}$ by definition of~$\mathcal{L}_{pq}$ and of~$\mathcal{A}_p$, and $q \leq n/2$ since for every~$\classt{n}{T_{pqr}}$, $T_{pqr}$ is a central triangle.
It will turn out to be convenient to include in the grouping only the classes~$\{\mathcal{L}_{pq} | 7n/24 \leq q < n/2\}$. Order the second-vertex-induced classes~$\{\mathcal{L}_{pq}\}$ with~$q \in [7n/24, n/2-1]$ in increasing order. Group pairs of adjacent second-vertex classes, then group these pairs into adjacent pairs, and so on.

Suppose $q < q'$ without loss of generality. At the~$j-i$ level of the grouping, i.e. the level at which the number of second-vertex classes on the left- and right-hand sides combined is~$j-i$, the amount of flow to be exchanged between~$\mathcal{L}_{pq}$ and~$\mathcal{L}_{pq'}$ lying on respectively the left and right-hand sides of the group, in each direction, is
$$\frac{C_n|\mathcal{L}_{pq}||\mathcal{L}_{pq'}|}{\sum_{i\leq q'' \leq j}|\mathcal{L}_{pq''}|},$$
where~$[i, j]$ is the interval of classes~$\mathcal{L}_{pi},\dots, \mathcal{L}_{pj}$ defining the group.

Let $\edget{n}{L_{pq}}{L_{pq'}} = \bigcup_{\classt{n}{T_{pqr}}\in\mathcal{L}_{pq}} \edget{n}{T_{pqr}}{T_{pq'r}}$ denote the matching connecting~$\mathcal{L}_{pq}$ and~$\mathcal{L}_{pq'}$. The resulting congestion is at most
$$\frac{|\mathcal{L}_{pq}||\mathcal{L}_{pq'}|}{\sum_{i\leq q'' \leq j}|\mathcal{L}_{pq''}||\edget{n}{L_{pq}}{L_{pq'}}|} \leq \frac{O((q' - q)^{3/2})}{(j-i)},$$
where the inequality holds because ~$|\mathcal{L}_{pq''}| \geq |\mathcal{L}_{pq}|$ for~$i + (j-i)/2 \leq q'' \leq j$, so that~$\frac{|\mathcal{L}_{pq}|}{\sum_{i\leq q'' \leq j}|\mathcal{L}_{pq''}|} \leq \frac{1}{(j-i)/2}$,
and because~$\frac{|\mathcal{L}_{pq'}|}{|\edget{n}{L_{pq}}{L_{pq'}}|} = O((q'-q)^{3/2})$.
The latter fact can be seen as follows: first,
$$\edget{n}{L_{pq}}{L_{pq'}} = \bigcup_{T_{pqr} \in \mathcal{L}_{pq}} \edget{n}{T_{pqr}}{T_{pq'r}}.$$
Every~$\classt{n}{T_{pqr}}$ has a nonempty matching to its neighboring class~$\classt{n}{T_{pq'r}}$, and indeed~$\classt{n}{T_{pq'r}}$ lies in~$\mathcal{A}_p$ and in~$\mathcal{L}_{pq'}$. On the other hand, due to the constraint for membership in~$\mathcal{A}_p$ that~$pq$ and~$pq'$ be the shortest edges of their respective central triangles, there may exist some values of~$r$ for which~$\mathcal{T}_{pq'r} \in \mathcal{L}_{pq'}$ but for which there is no neighbor of~$\mathcal{T}_{pq'r} \in \mathcal{L}_{pq}$. Fortunately:
\begin{enumerate}[(i)]
\item Since by assumption~$q \geq 7n/24$, it is easy to show that~$\classt{n}{T_{pqr}} \in \mathcal{L}_{pq}$ (i.e. the edge~$qr$ is indeed shorter than the edges~$pq$ and~$rp$) for~$r=n/2+1, \dots, 2\cdot 7n/24 = 7n/12$, and thus there are at least~$n/12$ central-triangle classes in~$\mathcal{L}_{pq}$ (and thus at least as many in~$\mathcal{L}_{pq'}$.
\item In~$\mathcal{L}_{pq}$ (and similarly~$\mathcal{L}_{pq'}$), the central-triangle classes occur in decreasing order of size (up to asymptotic order) as~$r$ increases.
\end{enumerate}

Facts~(i) and~(ii) imply that an~$\Omega(1)$ factor of the triangulations in~$\mathcal{L}_{pq'}$ lie in central-triangle classes having a neighboring class in~$\mathcal{L}_{pq}$, and thus
$$|\edget{n}{L_{pq}}{L_{pq'}}| = \Omega(1)\Omega(\frac{1}{(q'-q)^{3/2}}|\classt{n}{L_{pq'}}|.$$

Now, the~$\frac{O((q' - q)^{3/2})}{(j-i)} \leq \sqrt{j-i} \leq \sqrt{n}$ congestion that occurs across the boundary matching~$\edget{n}{L_{pq}}{L_{pq'}}$ for a given~$q,q'$ pair occurs  for a \emph{single commodity}, at a single level in the hierarchical grouping. We need to distribute this flow evenly first throughout each class~$\classt{n}{T_{pqr}}$ that receives it, and then throughout~$\mathcal{L}_{pq}$. By the same reasoning as in the proof of Lemma~\ref{lem:logtrick}, this flow can be distributed throughout a given class~$\classt{n}{T_{pqr}} \in \mathcal{L}_{pq}$ with no asymptotic congestion gain. Summing over all levels of the grouping produces
$$\sum_{s=0}^{\log(n/2-7n/24)} O(\sqrt{2^s}) = O(\sqrt n)$$
congestion within each~$\classt{n}{T_{pqr}}$.

To distribute the flow received by~$\classt{n}{T_{pqr}}$ throughout~$\mathcal{L}_{pq}$, first notice that the total amount of (normalized by a factor of~$C_n$) flow received by~$\classt{n}{T_{pqr}}$ is at most~$O(\log(r-q))|\classt{n}{T_{pqr}}|$ from classes~$\classt{n}{T_{pq'r}}, q < q' < (r-q)/2$, because each vertex in each boundary set~$\bdryt{n}{T_{pqr}}{T_{pq'r}}$ receives~$O(\sqrt{q'-q})$ flow and ~$|\bdryt{n}{T_{pqr}}| = \Theta(\frac{1}{(q'-q)^{3/2}})|\classt{n}{T_{pqr}}|,$ so the total is
$$|\classt{n}{T_{pqr}}|\sum_{k=1}^{(r-q)/2} \frac{\sqrt k}{k^{3/2}} = O(\log(r-q))|\classt{n}{T_{pqr}}|.$$

The analysis is similar for classes with~$q' < q$.

Now notice that the total amount of normalized flow received by~$\classt{n}{T_{pqr}}$ from classes~$\classt{n}{T_{pq'r}}$ with~$q' - q \geq (r-q)/2$ is at most
$$\sum_{k=r-n/2}^{(r-q)/2}O\left(\frac{1}{k^{3/2}}\right)\cdot O(\sqrt{(r-q)})|\classt{n}{T_{pqr}}| = O(\frac{\sqrt{r-q}}{\sqrt{r-n/2}})|\classt{n}{T_{pqr}}|.$$

Recall that we are dealing with a single commodity. Thus we do not have multiple MSFs to be concerned about. Unfortunately, however, the bound given by the construction in the proof of Lemma~\ref{lem:fwstrongindstep} gives a bound of~$O(\frac{\sqrt{r-q}}{\sqrt{r-n/2}}\cdot \sqrt{n})$, insufficient for our purposes. 

Fortunately, we can apply the hierarchical grouping trick again within~$\mathcal{L}_{pq}$, but we need to take care: first, it is insufficient merely to apply Corollary~\ref{cor:sqrtwithinsecondtotal}, as we simply recover the~$O(\sqrt{n})$ factor gain mentioned above. Second, unlike in Corollary~\ref{cor:sqrtwithinsecondtotal}, we are dealing here with only a single commodity (this will help us). What we do is observe that since the average flow received by a class~$\classt{n}{T_{pqr}}$ is, as stated,~$O(\frac{\sqrt{r-q}}{\sqrt{r-n/2}}) = O(\frac{\sqrt{n/2-q}}{\sqrt{r-n/2}})$, this bound decreases as~$r$ increases, and the average over all classes~$\classt{n}{T_{pqr}}$ within the range~$r \in[i,j]$ (assuming for the worst case that~$i = n/2+(n/2-1)$, since we are only considering flow from classes~$q'$ with) is at most
$$\sum_{r=n/2+1}^{n/2+j}\frac{(j+n/2-q)^{3/2}\sqrt{n/2-q}}{j(r-q)^{3/2}\sqrt{r-n/2}},$$
where we have used the fact that
$$\sum_{s\in[n/2+1,n/2+j]}|\classt{n}{T_{pqs}}| \geq j\cdot \frac{\sqrt{n/2-q}}{(j+n/2-q)^{3/2}}|\mathcal{L}_{pq}|,$$
since the left-hand side is a sum of~$j$ terms each of which is at least
$$|\classt{n}{T_{pq(n/2+j)}}| = \frac{\sqrt{n/2-q}}{(j+n/2-q)^{3/2}}|\mathcal{L}_{pq}|,$$
and also the fact that
$$|\classt{n}{T_{pqr}}| = \frac{O(\sqrt{n/2-q})}{(r-q)^{3/2}}|\mathcal{L}_{pq}|.$$

When~$j \leq n/2 - q$, we can bound the term
$$\sum_{r=n/2+1}^{n/2+j}\frac{(j+n/2-q)^{3/2}\sqrt{n/2-q}}{j(r-q)^{3/2}\sqrt{r-n/2}} \leq \frac{(j+n/2-q)^{3/2}}{n/2-q}\frac{1}{j}\sum_{r=n/2+1}^{n/2+j}\frac{1}{\sqrt{r-n/2}}$$
$$\leq \sqrt{n/2-q}\frac{\sqrt{j}}{j} = \frac{\sqrt{n/2-q}}{\sqrt{j}}$$
since~$r-q \geq n/2 - q$ always, and since we are assuming~$j \leq n/2 - q$.

When~$j > n/2 - q$, remember that we are considering only the flow to each~$\classt{n}{T_{pqr}}$ from classes~$\classt{n}{T_{pq'r}}$ with~$q' - q \geq (r-q)/2$, and therefore with~$r \leq n/2 - q$. Thus the average never exceeds~$\frac{\sqrt{n/2-q}}{\sqrt{j}}$.

Let~$\mu_{j-i}$ denote this average. Now we can bound the congestion across a given matching~$\edget{n}{T_{pqr}}{T_{pqr'}}$, for~$\classt{n}{T_{pqr}}, \classt{n}{T_{pqr'}} \in \mathcal{L}_{pq}$ as
$$\frac{C_n\mu_{j-i}|\classt{n}{T_{pqr}}||\classt{n}{T_{pqr'}}|}{C_n\sum_{s=i}^j |\classt{n}{T_{pqs}}||\edget{n}{T_{pqr}}{T_{pqr'}}|} \leq \frac{(r'-r)^{3/2}\mu_{j-i}}{j-i} \leq \frac{(r'-r)^{3/2}\sqrt{n/2-q}}{(j-i)^{3/2}} \leq \sqrt{n/2-q} \leq \sqrt{n}$$
for all $j-i$.
Since we are dealing with a single commodity and every class~$\classt{n}{T_{pqr}}$ has at most~$O(\sqrt{n})$ surplus or demand in each of its boundary vertices, we can use the same construction as in the proof of Lemma~\ref{lem:fwstrongindstep} to conclude that the overall resulting congestion in~$\mathcal{L}_{pq}$ is at most~$O(\sqrt{n})$.

It remains to consider the flow received by~$\classt{n}{T_{pqr}}$ from classes~$\classt{n}{T_{pq'r}}$ with~$q' - q \leq (r-q)/2$. As we have already observed, the average for each~$\classt{n}{T_{pqr}}$ is at most~$O(\log(r-q)) = O(\log n)$, and thus we can simply apply Corollary~\ref{cor:sqrtwithinsecondtotal} to obtain~$O(\sqrt n \log n)$ congestion.

Lastly, we have only distributed flow so far among the classes~$\mathcal{L}_{pq}$ with~$q \geq 7n/24$. We need to send flow from classes with~$q \geq 7n/24$ to those with~$n/4 \leq q < 7n/24$ and vice versa. We will first use the same construction as above to concentrate all of the flow from the~$[7n/24, n/2]$ classes within the~$[7n/24, n/3]$ classes. Because (as it is easy to show) the~$[7n/24, n/3]$ classes constitute a~$\Theta(1)$ factor of the triangulations in~$\mathcal{A}_p$, this concentration causes at most an~$O(1)$ increase in congestion.

Now let the~$[7n/24, n/3]$ and the~$[n/4, 7n/24]$ classes exchange flow. Once more we apply the hierarchical grouping trick. The challenge is now that for the number of central-triangle classes in~$\mathcal{L}_{pq}$ is small. Let~$\chi(\mathcal{L}_{pq}) = |\{\classt{n}{T}_{pqr} | \classt{n}{T_{pqr}} \in \mathcal{L}_{pq}\}|$ denote the number of central-triangle classes in~$\mathcal{L}_{pq}$. It is easy to show that~$\chi(\mathcal{L}_{pq}) = 2(q-n/4)$ whenever~$n/4 \leq q \leq n/3$.

Thus for~$n/4 \leq q < q' \leq n/3$ we can bound
$$\frac{C_n|\mathcal{L}_{pq}||\mathcal{L}_{pq'}|}{C_n\sum_{q''=i}^j |\mathcal{L}_{pq''}||\edget{n}{L_{pq}}{L_{pq'}}|} \leq \frac{(q'-q)^{3/2}\chi(\mathcal{L}_{pq'})(1/n^{3})}{(j-i)/2\chi(\mathcal{L}_{p(q'-(j-i)/2)})(1/n^{3})} = O(\sqrt{j-i})$$
for each group, and we are done.
\end{proof}

\begin{lemma}
\label{lem:sqrtwithinregion}
Within every region~$\mathcal{U}_i$, it is possible to route a unit of flow between every ordered pair of triangulations~$t,t' \in \mathcal{U}_i$ while producing total congestion~$O(\sqrt{n}\log n)$.
\end{lemma}
\begin{proof}
We need a shuffling step first: let each central-triangle class shuffle via a uniform multicommodity flow, scaled so that each triangulation~$t\in\classt{n}{T}$ in~$\mathcal{U}_i$ sends~$\frac{|\mathcal{U}_i|}{|\classt{n}{T}|}$ units to each~$t'\in\classt{n}{T}$. By the natural induction we have been using, this can be done with~$O(\sqrt{n/2}\log(n/2))$ congestion. We then have a collection of MSFs, each with source set~$\classt{n}{T_{pqr}}$, for each~$\classt{n}{T_{pqr}}$. Apply Lemma~\ref{lem:sqrtwithinapex} to solve these MSFs with~$O(\sqrt{n}\log n)$ additional congestion.

Finally, we need to solve~$n/12$ MSFs, one for each apex class in the region~$\mathcal{U}_i$. Each MSF has as its source set an apex class. All apex classes are isomorphic to one another and have cardinality~$C_n/n$; the surplus values are all~$|\mathcal{U}_i| = \Theta(1)C_n$, and the sink set for each MSF is all of~$\mathcal{U}_i$.

We will use the hierarchical grouping trick once more: just as we grouped together central-triangle classes within a second-vertex class in Lemma~\ref{lem:logtrick}, and just as we grouped together second-vertex classes in the proof  of Lemma~\ref{lem:sqrtwithinapex}, here we group apex classes first into pairs, then into contiguous sequences (in, say, counterclockwise order according to the apex~$p$) of four, then eight, and so on up to~$n/3$.

Crucially, whenever~$\mathcal{A}_p, \mathcal{A}_{p'}$ lie in a given~$\mathcal{U}_i$ (i.e.~$|p' - p| \leq n/12$), it is easy to show that a~$\Theta(1)$ factor of the triangulations in~$\mathcal{A}_p$ lie in classes~$\classt{n}{T_{pqr}}$ having a neighboring class~$\classt{n}{T_{p'qr}}$ in~$\mathcal{A}_{p'}$ such that~$|\edget{n}{T_{pqr}}{T_{p'qr}}| \geq (p'-p)^{3/2}|\classt{n}{T_{pqr}}|$.

Thus the hierarchical grouping produces
$$\frac{\Theta(1)C_n|\mathcal{A}_p||\mathcal{A}_{p'}|}{\Theta(1)C_n(j-i)|\mathcal{A}_p||\edget{n}{A_p}{A_{p'}}|} = O\left(\frac{(p'-p)^{3/2}}{j-i}\right) = O(\sqrt{j-i})$$
congestion at the~$j-i$ level,
and~$O(\sqrt n)$ congestion overall.

The MSF subproblems induced in each~$\mathcal{A}_p$ each involve distributing a single commodity with surplus~$O(\sqrt{j-i})$ throughout a given class~$\classt{n}{T_{pqr}}$, such that the resulting average flow concentrated in~$\classt{n}{T_{pqr}}$ is~$O(1)C_n$, then distributing this flow throughout~$\mathcal{A}_p$, which in turn produces~$O(\sqrt n \log n)$ congestion, proving the lemma.
\end{proof}

We have now demonstrated that a flow exists, in the subgraph of $K_n$ induced by the region $\mathcal{U}_i$\textemdash in which the additional congestion added in the inductive step is $O(\sqrt{n} \log n)$. We are now ready to prove Theorem~\ref{thm:triangmixub} by way of Lemma~\ref{lem:flowind}, by routing flow among the 24 regions:

\lemflowind*
\begin{proof}
The first step is to apply Lemma~\ref{lem:sqrtwithinregion}, obtaining a flow $f_i$ within each $\mathcal{U}_i$ in which each pair of triangulations exchanges a unit of flow, and in which each edge carries at most $O(\sqrt n \log n)$ congestion. 

We do the same for all regions. There is a wrinkle: since some edge classes (and pairs thereof) belong to more than one region, these 24 scaled-up flows result in multiple units of flow being sent between some pairs, as well as a constant-factor increase in congestion. For the former, we simply let pairs in the same class abstain from exchanging flow after the (lexicographically, say) first of the six flows. Clearly, the flows between pairs can never increase the congestion in the network.

For the latter, one may worry that we have lost our ``additive advantage'' and will now incur a multiplicative penalty in the induction. Fortunately, however, it is easy to see that the multiplicative factor is only applied \emph{after} we have applied the inductive hypothesis within each triangular class.

Next, we need to route the $\mathcal{U}_i \rightarrow \mathcal{U}_{i+1}$ flow through the triangular classes in the intersection $\mathcal{U}_i \cap \mathcal{U}_{i+1}$. This we accomplish by noting that, by Remark~\ref{rmk:equalpart}, we can simply concentrate the flow within the intersection between the regions, then send it with~$O(1)$ congestion gain. To get from $\mathcal{U}_{i+1}$ to the other 22 classes, we send flow in turn within $\mathcal{U}_{i+1}$, concentrating it on the boundary wtih~$\mathcal{U}_{i+2}$, and so on. Upon reaching the destination region, we then distribute the flow in a fashion symmetric to the concentration.

The increases in congestion in this process are all by a constant factor, and crucially, again, these increases are not applied more than once in the induction: our application of the inductive hypothesis occurs only within each \emph{central-triangle} class, and all subsequent routing and redistribution of flow through and within these classes avoids multiplying these factors by the congestion assumed in the inductive hypothesis. 
 
Finally, the overall $O(\sqrt n \log n)$ congestion bound claimed now follows from combining the $\log n$ levels of induction with the master theorem.
\end{proof}
Theorem~\ref{thm:assocexplb} is now immediate. A mixing upper bound of $O(n^{4}\log^2 n)$ follows from Lemma~\ref{lem:expmixing}; in Section~\ref{sec:avgcond} we will improve this to the~$O(n^3 \log^3 n)$ bound claimed in Theorem~\ref{thm:triangmixub}.

\subsection{Eliminating~$\log|V(K_n)|$: mixing time~$O(n^3 \log^3 n)$ for triangulations}
\label{sec:avgcond}
We have obtained our~$O(n^4\log^2 n)$ bound by showing that the expansion of~$K_n$ is~$\Omega(1/(\sqrt n \log n))$, then applying Lemma~\ref{lem:expmixing}. The loss comes from: (i) normalizing by the degree~$\Theta(n)$ of~$K_n$, (ii) squaring the resulting bound per Lemma~\ref{lem:expmixing}, and~(iii) multiplying by an additional factor of~$\log|V(K_n)| = \Theta(n)$. We show in this section that we can eliminate the~$\Theta(n)$ factor in step~(iii), obtaining an overall bound of~$O(n^3 \log^3 n)$ via a result of Lov\'{a}sz and Kannan:
\begin{lemma}~\cite{avgcond}
\label{lem:avgcond}
Given a family of finite, reversible, connected Markov chains~$\{\mathcal{M}_n = (\Omega_n, P_n)\}$ parameterized by~$n$, with stationary distribution~$\pi$, let~$\pi_{\min} = \min_{\{t \in \Omega_n\}}\pi(t)$. For all~$x \in [1/\pi_{\min}, 1/2],$ define the quantity
$$\phi(x) = \min_{S: \pi(S) \leq x}\frac{|\partial S|}{\operatorname{vol}(S)},$$ where~$\operatorname{vol}(S) = \sum_{t \in S}\frac{\pi(t)}{\Delta}$ is the probability mass of~$S$ normalized by the maximum degree~$\Delta$ of the chain~$\mathcal{M}_n$ (viewed as a graph).
Then the mixing time of~$\mathcal{M}_n$ is at most
$$\tau(n) \leq O(1)\int_{\pi_{\min}}^{1/2}{\frac{dx}{(\phi(x))^2 x}}.$$
\end{lemma}

Lemma~\ref{lem:avgcond} implies that in a given flip graph, if small sets have sufficiently larger expansion than large sets, then one can eliminate the~$\log|\Omega|$ factor incurred in passing to mixing from squared expansion. This in fact is true for~$K_n$: suppose a set~$S \subseteq V(K_n)$ is at most~$(C_{n/k})^{k}/2$, for a given integer~$k \in [1, \dots, n + 1]$. It is easy to show that~$S$ can be partitioned into a collection of subsets of disjoint Cartesian products of the form~$K_{i_1} \Box K_{i_2} \Box \cdots \Box K_{i_{k}}$, where each~$K_{i_j}$ is a smaller flip graph with all~$i_j \leq \frac{n}{k}$, because of the following fact:
\begin{lemma}
\label{lem:partitioncart}
For every integer~$1 \leq k \leq n$, every triangulation~$t\in V(K_n)$ lies in some Cartesian product of flip graphs~$K_{i_1} \Box K_{i_2} \Box \cdots \Box K_{i_{k}}$, with~$i_j \leq \frac{n}{2^{\lfloor\log_3 k \rfloor}}$ for all~$j$.
\end{lemma}
\begin{proof}
To identify the Cartesian product to which~$t$ belongs, partition~$K_n$ using the central-triangle partitioning. Each class is a Cartesian product of three smaller flip graphs induced by three smaller polygons; partition each of these classes according to the three central triangles in the three smaller polygons. Repeat this process recursively, in a ``breadth-first'' fashion, with the triangles placed at a given level in some consistent lexicographic order. . Stop the partitioning after~$k$ polygons have been obtained. Now the original $n$-gon has been partitioned into a collection of smaller polygons, the size of each of which is at most $\max\{1, n/2^{\lfloor\log_3 k\rfloor}\}$. This is because, first, if the recursion depth is~$d$, then the number of ``leaf nodes''\textemdash polygons at the bottom level of partitioning\textemdash is at most~$3^d$.
Second, the breadth-first nature of the partitioning guarantees that each level of partitioning decreases the maximum size of a polygon by at least half, so the largest polygon has size at most~$n/2^{\lfloor\log_3 k\rfloor}$.

Now, once the partitioning has stopped, the number of triangulations lying in the resulting partition is at least~$(C_{n/k})^k$, because the partition consists of~$k$ polygons whose sizes add up to at least~$n$, and because the size of the resulting Cartesian product~$C_{l_1}C_{l_2}\cdots C_{l_k}$, $\sum_i l_i \geq n$, is minimized when~$l_i = n/k$ for all~$i$.
\end{proof}

The following is now easy:
\begin{corollary}
\label{cor:partitioncond}
For every~$S \subseteq V(K_n)$, if~$|S| \leq (C_{n/k})^k/2$, for integer~$k \in [1, n]$, then~$|\partial S|/|S| \geq \Omega(1/((n/2^{\lfloor\log_3 k\rfloor})^{3/2}\log(n/2^{\lfloor\log_3 k\rfloor})))$.
\end{corollary}
\begin{proof}
The claim follows from noticing that any such set can be partitioned into its intersections with Cartesian products (sets of triangulations) of the form described in Lemma~\ref{lem:partitioncart}, each of which is at most half full, then noticing that in each such Cartesian product, by Lemma~\ref{lem:partitioncart} each graph~$K_{i_j}$ in the product has~$i_j \leq \frac{n}{2^{\lfloor\log_3 k\rfloor}}$. Appyling Theorem~\ref{thm:assocexplb} then proves the claim.
\end{proof}

We now combine Lemma~\ref{lem:avgcond} with Lemma~\ref{lem:partitioncart}, then combine Lemma~\ref{lem:expmixing} with Theorem~\ref{thm:assocexplb} to obtain mixing time~$O(n^3 \log^3 n)$ for triangulations, proving Theorem~\ref{thm:triangmixub}:
\begin{proof}{Proof of Theorem~\ref{thm:triangmixub}}
We can write, by Lemma~\ref{lem:avgcond},
$$\tau(n)\leq O(1)\int_{\pi_{\min}}^{1/2}{\frac{dx}{(\phi(x))^2 x}}$$
$$= O(1)\sum_{k=1}^n \int_{(C_{n/(k+1)})^{k+1}/C_n}^{(C_{n/k})^k/C_n} O((n/2^{\lfloor\log_3 k\rfloor})^3 \log^2 (n/2^{\lfloor\log_3 k\rfloor})) \frac{dx}{x}$$
$$\leq O(n^3\log^2 n )\sum_{k=1}^n O((1/2^{\lfloor\log_3 k\rfloor})^3  \int_{(C_{n/(k+1)})^{k+1}/C_n}^{(C_{n/k})^k/C_n} \frac{dx}{x}$$
$$= O(n^3\log^2 n )\sum_{k=1}^n O((1/2^{\lfloor\log_3 k\rfloor})^3  \ln\left(\frac{(C_{n/k})^k}{(C_{n/(k+1)})^{k+1}}\right)$$
$$= O(n^3\log^2 n )\sum_{k=1}^n O((1/2^{\lfloor\log_3 k\rfloor})^3  \ln(O(n^{3/2}))$$
$$= O(n^3\log^3 n )\sum_{k=1}^n O(1/k^{\log_3 8})$$
$$= O(n^3\log^3 n ) \cdot O(1) = O(n^3\log^3 n ).$$
\end{proof}

\section{Associahedron expansion upper bound}
\label{sec:expub}
To prove Theorem~\ref{thm:assocexpub}, we simply find a sparse cut and apply the definition of expansion. We use the central-triangle partition we used in Appendix~\ref{sec:quadmix} and Appendix~\ref{sec:combinedec}. These are the same classes used to show the $\Omega(n^{3/2})$ mixing lower bound by Molloy, Reed, and Steiger~\cite{molloylb}. As we discussed in the introduction, their mixing lower bound does \emph{not} imply the expansion upper bound we give here, but our expansion upper bound does imply their mixing lower bound.
\label{sec:assocexpub}
\subsection{Finding a sparse cut}
\label{sec:thecut}
We will find a cut $(S, \bar{S})$ with $|\partial S|/|S| = O(n^{-1/2}).$ We start by partitioning the vertices of the associahedron into central-classes as in Appendix~\ref{sec:combinedec}; within any given class, all vertices will be on the same side of the cut. Consider the associahedron~$K_{n-2}$ over the~$n$-gon. Draw the regular~$n$-gon in the plane, and label the vertices of the regular~$n$-gon~$[0, n-1]$, with~$0$ as the topmost vertex. 

Let~$\mathcal{C}_l = \{t \in \classt{n}{T} | T\textnormal{ has shortest side length} l\}$ be the set of all triangulations whose central triangle's shortest side has length~$l$, ~$l \in [1, n/3]$.

Let~$\mathcal{S} = \bigcup_{\mathcal{C}_l | l \in [1, n/6]}\mathcal{C}_l$. Let~$\mathcal{\bar S} = V(K_{n-2}) \setminus \mathcal{S} = \bigcup_{\mathcal{C}_l | l \in (n/6, n/3]}\mathcal{C}_l$.
\begin{lemma}
\label{lem:lencuteven}
The cut~$\mathcal{S}$ is indeed a partition of~$V(K_{n-2})$, has~$|\mathcal{S}| = \Theta(1)|V(K_{n-2})|$ and~$|\mathcal{\bar S}| = \Theta(1)|V(K_{n-2}|$.
\end{lemma}
\begin{proof}
It is clear that every triangulation lies in exactly one~$\mathcal{C}_l$, and that~$\mathcal{S}$ and~$\mathcal{\bar S}$ together partition all of the triangulations. To see that~$|\mathcal{S}| = \Theta(1)|V(K_{n-2})| = \Theta(1)C_{n-2}$, we first count the cardinality of each~$\mathcal{C}_l, l \in [1,n/6]$. Consider the number of ways to choose a central triangle~$T$ so that~$\classt{n}{T} \subseteq \mathcal{C}_l$, i.e. so that~$T$ has shortest side length~$l$. Notice that there are~$n$ ways to choose the apex of a central triangle (the vertex opposite the shortest side). Conditioned on this choice, and conditioned on a choice of~$l$ for the side length opposite the apex, there are~$l$ ways to choose the second vertex (the first vertex after the apex in counterclockwise order) so that the center of the~$n$-gon still lies inside the triangle. For all of these choices, the side opposite the apex is indeed shortest. The number of triangulations lying in a class with shortest side length~$l$ is~$\Theta\left(\frac{1}{n^{3/2}l^{3/2}}\right)C_{n-2}$, and thus the number of triangulations is
$$\sum_{l=1}^{n/6}nl\frac{1}{\Theta(n^{3/2}l^{3/2})} = \Theta\left(\frac{1}{\sqrt n}\right)\sum_{l=1}^{n/6}\Theta\left(\frac{1}{\sqrt l}\right) = \Theta\left(\frac{\sqrt n}{\sqrt n}\right) = \Theta(1).$$

Thus~$|\mathcal{S}| = \Theta(1)C_{n-2}|$. For~$|\mathcal{S}|$, notice that for every~$T$ with shortest side length~$l \in (n/6, n/3]$, there are~$\Theta(n^2)$ ways to choose~$T$, by the same argument we used for~$l \in [1, n/6]$, and each central-triangle-induced class~$\classt{n}{T}$ with shortest side~$l \in (n/6, n/3]$ has~$|\classt{n}{T}| = \Theta\left(\frac{1}{n^3}\right)$. Thus we have the sum
$$\sum_{l=n/6+1}^{n/3}\frac{\Theta(n^2)}{\Theta(n^{3})} = \Theta(1).$$
\end{proof}

\begin{lemma}
\label{lem:lencutsparse}
The cut~$(\mathcal{S}, \mathcal{\bar S})$ has~$|\partial S|/|S| = O(1/\sqrt n)$
\end{lemma}
\begin{proof}
Notice that in order for a triangulation~$t \in \classt{n}{T}$, given a central-triangle class~$T$ in~$\mathcal{C}_l$, $l \in [1, n/6]$, to have a neighbor in~$\mathcal{\bar S}$, i.e. for~$t$ to have a neighboring triangulation~$t' \in \classt{n}{T'}$ with~$T'$ having shortest side length~$k\geq n/6 + 1$, the central triangles~$T$ and~$T'$ must form a quadrilateral in~$t$ and in~$t'$. This quadrilateral, in~$t$, consists of~$T$ along with a triangle~$U$, where~$U$ has shortest side length~$k$. The fraction of triangulations lying in the boundary set~$\bdryt{n}{T}{T'}$ is therefore at most~$O\left(\frac{1}{(k-l)^{3/2}}\right).$

For~$l \in [1, n/6]$, let
$$\partial_l S = \{(t, t') \in E(K_{n-2}) | t \in \mathcal{C}_l, t' \in \mathcal{\bar S}\} = \bigcup_{\classt{n}{T} \subseteq \mathcal{C}_l, \classt{n}{T'} \subseteq \mathcal{\bar S} }\edget{n}{T}{T'}$$
be the set of all cut edges incident to triangulations in~$\mathcal{C}_l$. 

We will split the sets~$\{\mathcal{C}_l\}$ in~$\mathcal{S}$ into the cases~$l \in [1, n/8]$ and~$l \in (n/8, n/6]$. First, for all~$l \in [1, n/8]$, by the above reasoning, we have
$$|\partial_l S|/|\mathcal{C}_l| = \sum_{k=n/6+1}^{n/3}O\left(\frac{1}{(k-l)^{3/2}}\right) \leq (n/3-n/6)O\left(\frac{1}{(n/6+1-l)^{3/2}}\right)$$
$$\leq (n/3-n/6)O\left(\frac{1}{(n/6+1-n/8)^{3/2}}\right) = O\left(\frac{n}{n^{3/2}}\right) = O(1/\sqrt n).$$

On the other hand, for~$l \in (n/8, n/6]$, we compute the sum
$$\frac{\sum_{l=n/8+1}^{n/6}|\partial_l S|}{|\mathcal{S}|} = \sum_{l=n/8+1}^{n/6}\frac{|\partial_l S|}{\Theta(1)C_{n-2}} \leq \sum_{l=n/8+1}^{n/6}|\mathcal{C}_l|\sum_{k=n/6+1}^{n/3}O\left(\frac{1}{(k-l)^{3/2}}\right)$$
$$= \int_{l=n/8+1}^{n/6}O\left(\frac{1}{\sqrt n \sqrt l}\right)\int_{k=n/6+1}^{n/3}O\left(\frac{1}{(k-l)^{3/2}}\right)dkdl,$$
where for the last inequality we have applied the observation from the proof of Lemma~\ref{lem:lencuteven} that~$|\mathcal{C}_l| = O\left(\frac{1}{\sqrt n\sqrt l}\right)C_{n-2},$ and have used the fact that asymptotically the summation is equal to a double integral. Evaluating the inner integral we obtain
$$O\left(\frac{1}{\sqrt n}\right)\int_{l=n/8+1}^{n/6}O\left(\frac{1}{\sqrt l\sqrt{n/6+1-l}}\right)dl = O\left(\frac{1}{n}\right)\int_{u=1}^{n/6-n/8}O\left(\frac{1}{\sqrt{u}}\right)du,$$
with the substitution~$u = n/6 + 1 - l$ and the observation that when~$l\geq n/8$ we have~$1/\sqrt{l} = O(1/\sqrt{n})$. Finally, evaluating the integral gives
$$O\left(\frac{1}{n}\cdot \sqrt{n/6-n/8}\right) = O(1/\sqrt n).$$

We now have
$$|\partial S|/|S| = \sum_{l \in [1, n/6]}|\partial_l S|/|S| = \sum_{l \in [1, n/8]} |\partial_l S|/|S| + \sum_{l \in [n/8+1, n/6]} |\partial_l S|/|S| = O(1/\sqrt n) + O(1/\sqrt n) = O(1/\sqrt n),$$
as claimed.
\end{proof}

\section{$k$-angulations of convex point sets: quasipolynomial mixing}
\label{sec:quadmix}
\subsection{Generalizing triangulations}
As we stated in the introduction, one can generalize triangulations to $k$-angulations. We do so in more detail here. A \emph{quadrangulation} of a point set is a maximal subdivision of the point set into quadrilaterals, where each quadrilateral has all of its vertices in the point set. Consider $P_{2n+2}$, the regular polygon with $2n + 2$ vertices. We denote by $K_{4,2n+2}$ the graph whose vertex set is the set of all quadrangulations of $P_{2n+2}$, and whose edges are the flips between quadrilaterals. Here, a flip is defined as follows: each diagonal belongs to two quadrilaterals, which together form a hexagon. Replace the diagonal with one of the other two diagonals in the hexagon. (Thus each diagonal in a qudrangulation can be flipped in two possible ways~\cite{caraceni2020}.)

There is a polytope, analogous to the associahedron, known as the \emph{accordiohedron}~\cite{towardspgaccord,banerjee}, whose vertices and edges are those of a \emph{subgraph} of $K_{4,2n+2}$. However, we ignore this polytope and just consider the graph $K_{4,2n+2}$.

We refer to a \emph{$k$-angulation} of a point set as a maximal subdivision of the point set into $k$-gons, each of whose vertices all belong to the point set. A bijection exists~\cite{hiltonpedersen} between the $k$-angulations of $P_{(k-2)n+2}$ and the set of all $k-1$-ary plane trees with $n$ internal nodes. 

It is easy to generalize the definition of a flip between triangulations or quadrangulations to a flip between $k$-angulations: each diagonal in a $k$-angulation belongs to two $k$-gons, which together form a $2k-2$-gon. A flip then consists of replacing this diagonal\textemdash which connects two opposite vertices in the $2k-2$-gon\textemdash with one of the $k - 2$ other such diagonals. 

We generalize the associahedron graph $K_n$ as follows:
\begin{definition}
\label{def:kangflipgraph}
Define the \emph{$k$-angulation flip graph} $K_{k,(k-2)n+2}$ as the graph whose vertices represent the $k$-angulations of $P_{(k-2)n+2}$, and whose edges represent the flips between $k$-angulations.
\end{definition}

\begin{definition}
\label{def:kangflipchain}
Define the \emph{$k$-angulation flip walk} as the natural Markov chain whose state space is $K_{k,(k-2)n+2}$.
\end{definition}

\subsection{(Generalized) Catalan numbers}
The usual notation for Catalan numbers is simply $C_n$; we will now consider a generalization:
\begin{definition}\cite{Raman2019,klarner,hiltonpedersen}
\label{def:fcat}
Let $C_{k,n} = \frac{1}{(k-2)n+1}\binom{(k-1)n}{n}.$
\end{definition}
These numbers, which generalize Catalan numbers, are similar but not identical to the \emph{Fuss-Catalan} numbers.

We will use the following fact in proving that the random walk on $k$-angulations mixes in quasipolynomial time:
\begin{lemma}~\cite{klarner,hiltonpedersen}
\label{lem:kangcount}
The number of $k$-angulations of the convex $(k-2)n+2$-gon is counted by~$C_{k,n}.$
\end{lemma}

One can show using Stirling's formula, and in particular a result by Robbins~\cite{robbinsapprox}, that:
\begin{lemma}
\label{lem:fusscatstirling}
For all $k \geq 3$ and $n \geq 1$, $e^{-1/6}\frac{k-2}{k-1}f(k, n) \leq C_{k,n} \leq e^{1/12}\cdot f(k,n)$, where
$$f(k,n) = \frac{\sqrt{k-1}}{\sqrt{2\pi}((k-2)n)^{3/2}}\cdot \frac{(k-1)^{(k-1)n}}{(k-2)^{(k-2)n}}.$$
\end{lemma}

We will prove the following:
\begin{restatable}{lemma}{lemkangfw}
\label{lem:kangfw}
The flip graph~$K_{k,(k-2)n+2}$, along with the partition~$\mathcal{S}_{k,(k-2)n+2}$, satisfies Lemma~\ref{lem:fwquasi}.
\end{restatable}

Theorem~\ref{thm:kangmix}, as we will show in Appendix~\ref{sec:quadmix}, will follow from tracing the particular quasipolynomial factors in the proof of Lemma~\ref{lem:kangfw}.

To prove Lemma~\ref{lem:kangfw}, we will partition $K_{k,(k-2)n+2}$ into a set of classes $\mathcal{S}_k$ in a suitable fashion. We will define a partition that generalizes one by Molloy, Reed, and Steiger~\cite{molloylb}. In order to define $\mathcal{S}_k$, we need some observations about the structure of the graph $K_{k,(k-2)n+2}$.

\subsection{Partition into classes}
\label{sec:mg2}

\begin{figure}
\includegraphics[width=15em]{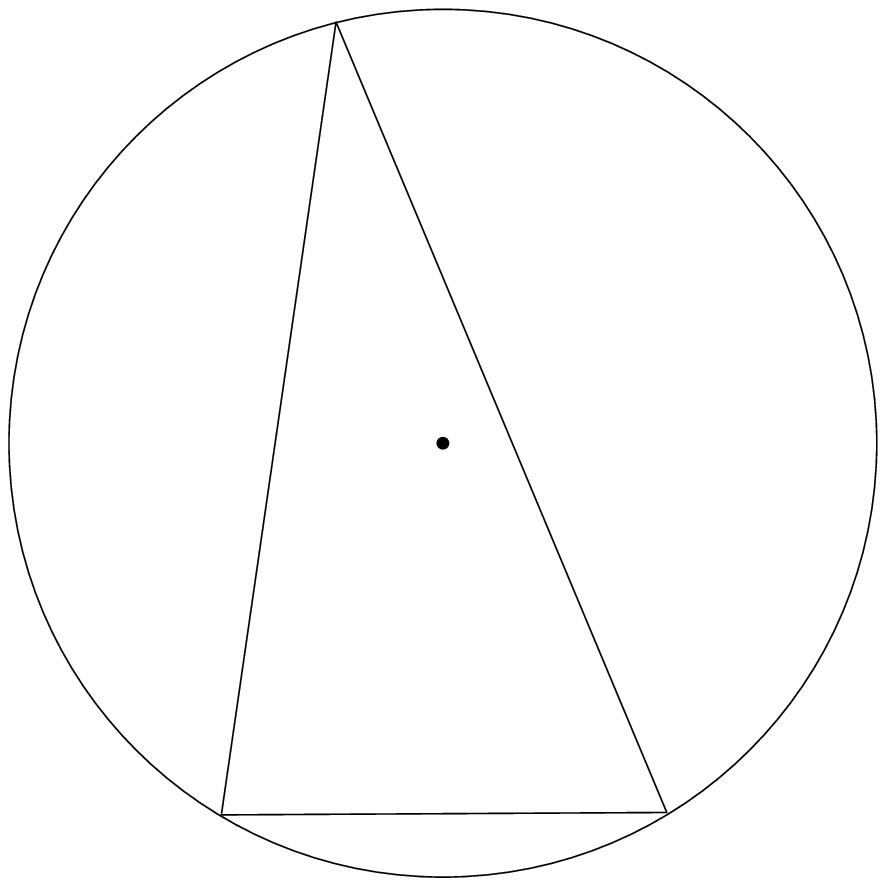}
\hspace*{1.5em}
\includegraphics[width=15em]{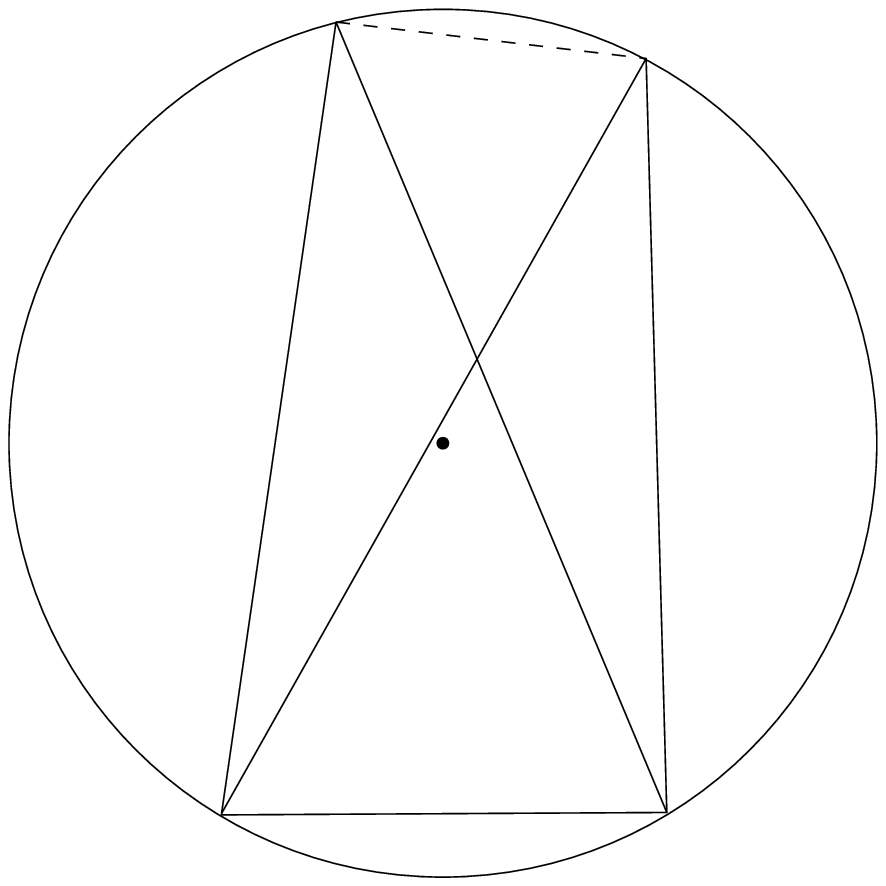}
\caption{Left: a class $\classt{}{T}$ in $K_{3,n+2}$. Each triangulation in $\classt{}{T}$ contains the central triangle depicted. We depict the polygon $P_{n+2}$ as a circle for simplicity. Right: the set of edges $\edget{}{T}{T'}$ (which form a matching) between two classes.}
\label{fig:tclass}
\end{figure}

\begin{definition}
\label{def:kpart}
Given a $k$-gon $T$ containing the center of the regular $(k-2)n+2$-gon $P_{(k-2)n+2}$ and sharing all of its vertices with $P_{(k-2)n+2}$, identify $T$ with the class $\classt{}{T}$ of $k$-angulations $v \in V(K_{k,(k-2)n+2})$ such that $T$ forms one of the $k$-gons in the $k$-angulation $v$. Let $\mathcal{S}_{k,(k-2)n+2}$ be the set of all such $\classt{}{T}$ classes. 
\end{definition}
(If $P_{n+2}$ has an even number of edges, we perturb the center slightly so that every triangulation lies in some class.) 

\begin{remark}
The set $\mathcal{S}_{k,(k-2)n+2}$ is a partition of $V(K_{k,(k-2)n+2})$, because no pair of $k$-gons whose endpoints are polygon vertices can contain the origin without crossing.
\end{remark}

(This generalizes the partition of Molloy, Reed, and Steiger~\cite{molloylb}.)

\begin{definition}
\label{def:tdse}
Given classes $\classt{}{T}, \classt{}{T'} \in \mathcal{S}_{k,(k-2)n+2}$, let $\edget{}{T}{T'}$ be the set of edges between with one endpoint in $\classt{}{T}$ and one endpoint in $\classt{}{T'}$. Let $\bdryt{}{T}{T'}$ denote the set of vertices in $\classt{}{T}$ that have at least one neighbor in $\classt{}{T'}$.
\end{definition}
See Figure~\ref{fig:tclass}. 

\begin{remark}
\label{rmk:tdsepartition}
The set of edge sets of the form $\edget{}{T}{T'}$ is a partition of all edges between pairs of vertices in different classes.
\end{remark}

\subsubsection{Cardinalities of classes and of edge sets}
We make some observations about the nature and cardinalities of the classes in $\mathcal{S}_{k,(k-2)n+2}$, and of the sets and numbers of edges between the classes.

\begin{lemma}
\label{lem:tdclasstype}
Each $k$-gonal class in $\mathcal{S}_{k,(k-2)n+2}$ induces a subgraph of $K_{k,(k-2)n+2}$ that is isomorphic to the Cartesian product $K_{k,(k-2)i_1+2}\Box K_{k,(k-2)i_2+2} \Box \cdots K_{k,(k-2)i_k+2},$ for some $1 \leq i_1 \leq \cdots \leq i_k \leq n/2$, $i_1 + \cdots + i_k = n-1.$ 
\end{lemma}
\begin{proof}
Each $k$-gon $T$ partitions the regular $(k-2)n+2$-gon into smaller convex polygons with side lengths $(k-2)i_1+2, (k-2)i_2+2, \dots, (k-2)i_k+2$. Thus each $k$-angulation in $\classt{k}{T}$ can be identified with a tuple of $k$-angulations of these smaller polygons. The Cartesian product structure then follows from the fact that every flip between two $k$-angulations in $\classt{k}{T}$ can be identified with a flip in one of the smaller polygons.
\end{proof}

\begin{lemma}
\label{lem:bdrytype}
For each pair of classes $\classt{}{T}$ and $\classt{}{T'}$, the boundary set $\bdryt{}{T}{T'}$ induces a subgraph of $\classt{}{T}$ isomoprhic to a union of Cartesian products of the form $K_{(k-2)i_1+2}\Box K_{(k-2)i_2+2} \Box \cdots \Box K_{(k-2)i_{2k-2}+2}$, for some $i_1 \leq \cdots \leq i_{2k-2} \leq n/2,$ $i_1+\cdots+i_{2k-2}=n-2$.
\end{lemma}
\begin{proof}
Each flip between $k$-angulations in adjacent classes $\classt{k}{T}$ involves flipping a diagonal of the $k$-gon $T$ to transform $k$-angulation $t \in \classt{k}{T}$ into $k$-angulation $t' \in \classt{k}{T'}$. Whenever this is possible, there must exist a $2k-2$-gon $Q$, sharing $k-1$ sides with $T$ (the $k-1$ sides that are not flipped), such that both $t$ and $t'$ contain $Q$. Furthermore, every $t \in \classt{k}{T}$ containing $Q$ has a flip to a distinct~$t' \in \classt{k}{T'}$. The set of all such boundary vertices $t \in \classt{k}{T}$ can be identified with the Cartesian product described because $Q$ partitions $P_{(k-2)n+2}$ into a collection of smaller polygons, so that each $k$-angulation in $\bdryt{}{T}{T'}$ consists of a tuple of $k$-angulations in each of these smaller polygons, and such that every flip between $k$-angulations in $\bdryt{}{T}{T'}$ consists of a flip in one of these smaller polygons. (There may be many such~$2k-2$-gons for a given pair of classes, but the claim holds as a lower bound.)
\end{proof}

\begin{lemma}
\label{lem:tdsetype}
Each set of edges between classes in $\mathcal{S}_{k,(k-2)n+2}$ is in bijection with the vertices of a union of Cartesian products of the form $K_{(k-2)i_1+2}\Box K_{(k-2)i_2+2} \Box \cdots \Box K_{(k-2)i_{2k-2}+2}$, for $i_1 \leq \cdots \leq i_{2k-2} \leq n/2,$ $i_1+\cdots+i_{2k-2}=n-2$. Furthermore, no two edges in any such edge set share a vertex, i.e. the edge set is a matching. 
\end{lemma}
\begin{proof}
The claim follows from the reasoning in Lemma~\ref{lem:bdrytype}.
\end{proof}

\begin{corollary}
\label{cor:tdclasssize}
Each $k$-gonal class in $\mathcal{S}_{k,(k-2)n+2}$ has cardinality $C_{k,i_1}C_{k,i_2}\cdots C_{k,i_k}$, and each edge set between classes has cardinality at least $C_{k,i_1}C_{k,i_2}\cdots C_{k,i_{2k-2}}$. Here, $i_1, \dots, i_{2k-2}$ are as in Lemmas~\ref{lem:tdclasstype} and \ref{lem:tdsetype}.
\end{corollary}

\subsection{Applying the framework}
We are almost ready to prove that $K_{k,(k-2)n+2}$ satisfies the conditions of Lemma~\ref{lem:fwquasi}, but first we note the following known result:
\begin{lemma}
\label{lem:quadconn}
$K_{k,(k-2)n+2}$ is connected.
\end{lemma}

One way to prove Lemma~\ref{lem:quadconn} is via the isomorphism~\cite{hiltonpedersen} between flips on $k$-angulations and rotations on $k-1$-ary plane trees. One can prove that the rotation graph on $k-1$-ary plane trees is connected as follows: find a path from any given tree to a ``spine,'' where all internal nodes belong to a simple path via left children from the root to the leftmost leaf~\cite{culik1982note}. (This path consists of repeated left rotations.) Every non-spine tree has some internal node at which a left rotation can be performed. Furthermore, when no such operation is still possible, one has a spine.

Nakamoto, Kawatani, Matsumoto, and Urrutia~\cite{nakamoto} also gave a proof of connectedness for the special case~$k=4$. Sleator, Tarjan, and Thurston proved~\cite{sleator1988rotation} that the diameter of $K_{3,n+2}$ is at most $2n-6$ for $n\geq 11$. 

We now prove Lemma~\ref{lem:kangfw}:
\lemkangfw*
\begin{proof}
By Lemma~\ref{lem:tdclasstype} and the observation that there are at most~$\binom{(k-2)n+2}{k}$ classes, the partition~$\mathcal{S}_{k,(k-2)n+2}$ meets Conditions 1 and 5 of the framework, with the modification to Condition~1 that the~$O(1)$ term is replaced with~$O(n^{O(1)})$, and Condition~6 easily follows from the identification of each class with a~$k$-gon containing the center of the~$(k-2)n+2$-gon.

Corollary~\ref{cor:tdclasssize} gives a formula for the size of each class and each edge set between classes. Lemma~\ref{lem:fusscatstirling} then easily gives a polynomial bound on the ratio of $N = |V(K_{k,(k-2)n+2})|$ to the size of the smallest class (similarly the smallest edge set). Conditions~2, 3, and~4 follow, with the modification that the~$O(1)$ terms are replaced with~$O(n^{O(1)})$ terms. 
\end{proof}

To derive the specific quasipolynomial bound in Theorem~\ref{thm:kangmix}, we first observe the following:
\begin{remark}
\label{rmk:edgesets}
The smallest edge set between classes in $\mathcal{S}_{k,(k-2)n+2}$ has size at least
$$C_{k,i_1}\cdots C_{k,i_{2k-2}} \geq N\cdot \frac{f(k,i_1)\cdots f(k, i_{2k-2})}{e^{(2k-2)/6+1/12}((k-1)/(k-2))^{2k-2}f(k,n)}$$
$$\geq Ne^{(3-4k)/12}\cdot\frac{(k-2)^{k-3/2}}{(k-1)^{3k-5/2}}\cdot \frac{1}{(2\pi)^{k-3/2}}\cdot\frac{1}{n^{3k}}.$$
\end{remark}

The next fact we need comes from Lemma~\ref{lem:nonhierexact}:
\lemnonhierexact*
Applying Lemma~\ref{lem:nonhierexact}, and using the fact that $K_{k,(k-2)n+2}$ is a $\leq(k-2)n$-regular graph with $\log N \leq (k-1)n\log(k-1)$, gives
$$O((2N/\mathcal{E}_{\min})^{2\log n}(k-1)^3 (\log(k-1))n^3)$$
$$= O((k-1)^3(\log(k-1)) n^3(2e^{(4k-3)/12}\cdot \frac{(k-1)^{3k-5/2}}{(k-2)^{k-3/2}}\cdot (2\pi)^{k-3/2}\cdot n^{3k})^{2\log n})$$
$$= O((k-1)^3(\log(k-1)) n^3(2e^{(4k)/12}\cdot (k-1)^{3k}\cdot (2\pi)^{k} \cdot n^{3k})^{2\log n})$$
$$= O((k-1)^3(\log(k-1))\cdot n^{2(3k\log(k-1)+k(1+\log\pi)+3k\log n + k)+5}).$$

Here we have implicitly used Lemma~\ref{lem:expmixing} to pass from the expansion bound given by Lemma~\ref{lem:nonhierexact} to a mixing bound. Actually, we can do better using the following standard lemma, which allows for passing from congestion to mixing without a quadratic loss:
\begin{lemma}~\cite{diacstroock, sinclair_1992}
\label{lem:flowmixing}
Suppose a uniform multicommodity flow~$f$ exists in a graph~$G = (V, E)$ with congestion~$\rho$, in which for every~$s,t\in V$, 
$$\max_{P\in \Gamma_{st}}|P| \leq l,$$
for some~$l > 0$, where $\Gamma_{st}$ s the set of (simple) paths in~$G$ from~$s$ to~$t$, and where we use the shorthand~$f_{st}(P)$ to denote the fraction of the~$s,t$ commodity that~$f$ sends along the path~$P$. Then the mixing time of the uniform random walk on~$G$ is
$$O\left(\frac{\rho l}{d} \log(|V(G)|)\right),$$
where~$d$ is the maximum degree of~$G$.
\end{lemma}

Then we obtain mixing time
$$O((2N/\mathcal{E}_{\min})^{\log n}(k-1)^2 (\log(k-1))n^2\cdot l),$$
where~$l$ is the maximum length of a path in the flow construction. It is not difficult to see that since the diameter of the projection graph is at most~$k$, we obtain a recurrence 
$$l = T(n) = k + 2k T(n/2) = O(n^{\log_2 k + 1}),$$
giving total mixing
$$O((2N/\mathcal{E}_{\min})^{\log n}(k-1)^2 (\log(k-1))n^2 \cdot n^{\log_2 k + 1}),$$
Notice that, as discussed in Remark~\ref{rmk:deletegamma}, we did not incur a term~$\gamma/\Delta$ in this calculation. Furthermore, it is not clear how one could avoid this~$\gamma/\Delta = k$ factor using the spectral decomposition technique (Theorem~\ref{thm:specprojres}). That technique would give a mixing bound of
$$O(((3N/\mathcal{E}_{\min})\cdot k\cdot k)^{\log n} \cdot (k-2)n).$$
(Here we have ignored an additional~$\log|\Omega|$ term, as one might be able to reduce this term via, for instance, the log-Sobolev version of the Jerrum/Son/Tetali/Vigoda decomposition.)
Comparing the two expressions above shows that our technique gives an improvement of
$$\Omega\left(\frac{n^{\log_2 (k^2)}}{n^{\log_2 k + 1}\cdot kn\log k}\right) = \Omega(n^{\log_2 k - 2} /(k\log k)).$$

\section{Proof that the conditions of Lemma~\ref{lem:fw} imply rapid mixing}
\label{sec:fwpf}
In this section we prove Theorem~\ref{thm:flowprojres}. Lemma~\ref{lem:fw} will then follow by way of Lemma~\ref{lem:nonhierexact}. 
\begin{figure}[h]
\includegraphics[height=10em]{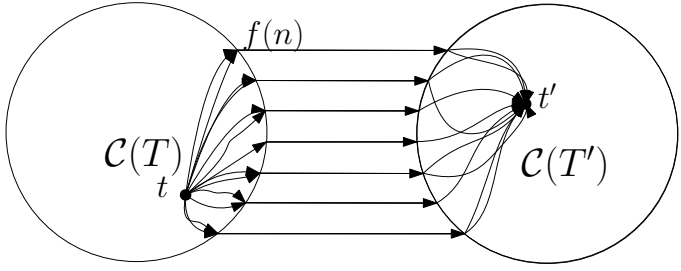}
\caption{In the flow construction we use for quasipolynomial mixing (Theorem~\ref{thm:kangmix}), we first find a flow in $\classt{n}{T}$ (similarly $\classt{n}{T'}$) and bound its congestion. (Actually, we assume such a flow exists, for the inductive hypothesis.) We then reuse the paths from this flow in routing the flow between $t$ and $t'$. Reusing these paths results in compounding the amount of flow across each path by $f(n)$, where $f(n)$ is the amount of flow across an edge between the two classes.}
\label{fig:assocquasiedge}
\end{figure}

\thmflowprojres*
Here, instead of a uniform flow in which each pair of states exchanges a single unit, it will be convenient to use a specification of demands and definition of congestion that are closer to standard in the analysis of Markov chains: the demands are~$D(z, w) = \pi(z)\pi(w)$, and the congestion across an edge~$(x, y)$ produced by a multicommodity flow~$f$ that satisfies the demands~$\{D(z, w) | z,w\in\Omega\}$ is~$\rho(x, y) = f(x, y)/(\Delta \cdot Q(x, y))$, where~$Q(x, y) = \pi(x)P(x, y) = \pi(y)P(y, x)$ (by reversibility).
One can check that in the uniform case, this definition is equivalent to our definition in Section~\ref{sec:prelim}. Furthermore, Lemma~\ref{lem:flowexp} and Lemma~\ref{lem:expmixing} are known~\cite{sinclair_1992} to work for the weighted case with this adjusted definition of congestion. 

The proof of Theorem~\ref{thm:flowprojres} is in fact not difficult to describe intuitively: if one finds a flow (collection of fractional paths) through the projection graph between every pair of classes (restriction chains), this flow induces a subproblem in each class~$\Omega_i$, in which each ``boundary vertex''\textemdash each vertex (state)~$z\in\Omega_i$ that brings in flow from a neighbor~$w\in\Omega_j$\textemdash must route the flow it receives throughout~$\Omega_i$. The state~$z$ may bring in an amount of flow up to~$\bar\rho\gamma\Delta$ from such neighbors, and~$z$ must route this flow (which we will show in the proof is at most~$\bar\rho\gamma\Delta\pi(z)$ throughout~$\Omega_i$. By assumption, it is possible for~$z$ to route~$\pi(z)$ flow throughout~$\Omega_i$ with congestion at most~$\rho_{\max}$, and therefore~$z$ can route the~$\bar\rho\gamma\Delta\pi(z)$ flow throughout~$\Omega_i$ with congestion at most~$\bar\rho\gamma\Delta\rho_{\max}$. The factor of~2 in the term~$1+2\bar\rho\gamma\Delta$ comes from applying the above reasoning twice: once for ``inbound flow'' that~$z$ brings into~$\Omega_i$, and once for ``outbound flow'' that~$z$ must route from~$\Omega_i$ to other classes. Finally, the factor of~1 comes from routing flow between pairs of states within~$\Omega_i$.

We now make this reasoning precise:
\begin{proof}{(Proof of Theorem~\ref{thm:flowprojres})}
Let~$\{f_i, i = 1, \dots, k\}$ be a collection of flow functions over the restriction chains with congestion~$\rho_i \leq \rho_{max}$, as supposed in the theorem statement. Suppose we have a flow~$\bar f$ with congestion~$\bar\rho$ in the projection chain. 

We construct a multicommodity flow~$f$ in the overall chain~$\mathcal{M}$ as follows: for every edge~$e = (x, y), x \in \Omega_i, y \in \Omega_j, i \neq j$ between restriction state spaces, let~$f_{xy} = \bar f(i, j)Q(x, y)/\bar Q(i, j).$ For pairs of states~$x, y \in \Omega_i$, simply use the same (fractional) paths to send flow as in~$f_i$. Now, for \emph{non-adjacent} pairs of states~$x \in \Omega_i, y \in \Omega_j$, $i \neq j$, we will use the flow~$\bar f$ to route the $x-y$ flow, perhaps through one or more intermediate restriction spaces. We need to consider how to route the flow \emph{through} each intermediate restriction space. This induces a collection of subproblems over each restriction space~$\Omega_i$ in which each state~$z\in \Omega_i$ ``brings in'' and similarly ``sends out'' at most~$\sum_{j\neq i}\sum_{w\in \Omega_j}\bar\rho\bar{Q}(i,j)Q(z,w)/\bar{Q}(i,j) \leq \bar\rho\pi(z)\gamma\Delta$ units of flow. We reuse the (fractional) paths that produce the flow with congestion~$\rho_{max}$, scaling the resulting congestion by~$\bar\rho\gamma\Delta$. More precisely, if~$e = (x, y)$ is an edge internal to a restriction space~$\Omega_i$, let
$$\hat f_{zu}(x, y) = (D_{in}(z, u)+D_{out}(u, z))(f_{i,zu}(x, y)\cdot \frac{\bar\pi(i)}{\pi(z)\pi(u)})$$ and denote
$$\hat f(x, y) = \sum_{zu}\hat f_{zu}(x, y) = \sum_{z,u\in\Omega_i}(D_{in}(z, u)+D_{out}(u, z))(f_{i,{zu}}(x, y)\cdot \frac{\bar\pi(i)}{\pi(z)\pi(u)})$$
and~$\hat \rho(x, y) = \frac{\hat f(x, y)}{Q(x, y)}$,
where
$$D_{in}(z, u) = (\sum_{j\neq i}\sum_{w\in\Omega_j}\bar f(i, j)Q(z, w)/\bar Q(i, j))\cdot \frac{\pi(u)}{\bar \pi(i)} \leq \bar \rho \frac{\pi(z)\pi(u)}{\bar \pi(i)} \gamma \Delta$$
is the share of the demand brought in by~$z$ to~$\Omega_i$ that must be sent to~$u$, and~$D_{out}(u, z)$ is similar.
This definition~$\hat f_{zu}(x, y)$ indeed satisfies the demands~$D_{in}(z, u)$ and~$D_{out}(u, z)$: $f_{i,zu}$ is defined as sending~$\frac{\pi(z)\pi(u)}{\bar \pi(i)}$ units of flow along a set of fractional paths from~$z$ to~$u$, so the function
$$\hat f_{zu}(x, y) = (D_{in}(z, u)+D_{out}(u, z))(f_{i,{zu}}(x, y)\cdot \frac{\bar\pi(i)}{\pi(z)\pi(u)})$$ sends~$D_{in}(z, u) + D_{out}(u, z)$ units of flow along the same set of fractional paths.

Now, we know by the definition of the congestion~$\rho_i$ produced by~$f_i$ that
$$\sum_{z,u\in\Omega_i} f_{i,zu}(x,y) \leq \rho_i Q_i(x, y),$$
where~$Q_i(x, y) = \frac{\pi(x)P(x, y)}{\bar\pi(i)}.$
Therefore
$$\hat\rho(x, y) = \frac{\hat f(x, y)}{Q(x, y)} = \frac{1}{Q(x, y)}\sum_{z,u\in\Omega_i}(D_{in}(z, u)+D_{out})(f_{i,{zu}}(x, y)\cdot \frac{\bar\pi(i)}{\pi(z)\pi(u)})$$
$$\leq \frac{1}{Q(x, y)}\sum_{z, u\in\Omega_i}2\bar\rho\gamma\Delta \frac{\pi(z)\pi(u)}{\bar\pi(i)}(f_{i,zu}(x,y)\cdot \frac{\bar\pi(i)}{\pi(z)\pi(u)}) = 2\bar\rho\gamma\Delta\sum_{z\in\Omega_i}\frac{f_{i,zu}(x, y)}{Q(x, y)} \leq 2\bar\rho\gamma\Delta\rho_{max}.$$

Now, for~$x,y\in \Omega_i$ and for~$u \in \Omega_i, v \in \Omega_j \neq \Omega_i$, we let
$$f_{vu}(x, y) = \sum_{z \in \Omega_i}\hat f_{v,zu}(x, y),$$
where
$$\hat f_{v,zu}(x, y) = (D_{v,in}(z, u) + D_{v, out}(u, z))(f_{i,zu}(x,y)\cdot \frac{\bar\pi(i)}{\pi(z)\pi(u)},$$
where $$D_{v,in}(z, u) = \frac{\pi(u)}{\bar\pi(i)} \sum_{k: \exists w\in\Omega_k, w \sim z} \bar f_{j,i}(k, i)\cdot \frac{Q(\Omega_k, z)}{\bar Q(k, i)}\cdot \frac{\pi(v)}{\bar\pi(j)}$$
and~$D_{v,out}(u, z)$ is symmetric.
It is easy to see that~$f_{vu}$ indeed is a valid flow sending~$\pi(u)\pi(v)$ units from~$v$ to $u$, and also that
$$\sum_{j\neq i} \sum_{v\in\Omega_j}\hat f_{v,zu} = \hat f_{zu}.$$
Therefore
$$\rho(x,y) = \sum_{u\in\Omega_i}\sum_{v \notin \Omega_i} \frac{f_{vu}(x, y)}{Q(x, y)} = \sum_{u,z\in \Omega_i}\sum_{v}\frac{\hat f_{v,zu}(x,y)}{Q(x,y)} = \sum_{u,z\in \Omega_i}\frac{\hat f_{zu}(x,y)}{Q(x,y)} = \hat\rho(x, y) \leq 2\bar \rho\gamma\Delta\rho_{max}.$$

Finally, in the term~$\rho(x,y)$ we have only considered~$u,v$ flow where~$u,v$ lie in different classes. Adding the congestion~$\rho_i \leq \rho_{max}$ produced by reusing the flow~$f_i$ for pairs~$u,v\in\Omega_i$ justifies the expression
$$(1+2\bar\rho\gamma\Delta)\rho_{max}.$$
\end{proof}

Lemma~\ref{lem:fw} and Lema~\ref{lem:nonhierexact} are now immediate, as is Lemma~\ref{lem:fwquasi}.

\section{Integer lattice triangulation flip graphs}
\label{sec:latticetri}
\subsection{Definition}
The integer lattice triangulation flip graph, studied extensively in prior work (\cite{anclin, sinclairlambda, phaseshift, kzlimit, Stauffer2015ALF}), is analogous to the associahedron and is defined as follows:

\begin{definition}
\label{def:latticefg}
Let the \emph{integer lattice triangulation flip graph} be the graph $F_{n}$ whose vertices are the triangulations of the $n \times n$ integer lattice point set (integer grid), and whose edges are the pairs of triangulations that differ by exactly one diagonal.
\end{definition}
\begin{figure}
\includegraphics[height=10em]{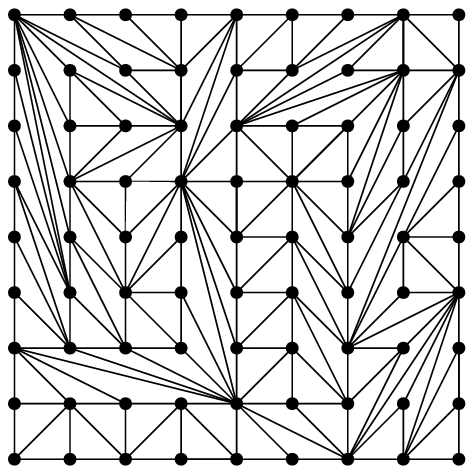}
\hspace*{1.5em}
\includegraphics[height=10em]{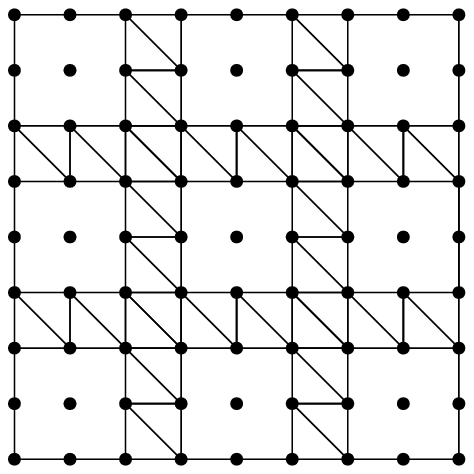}
\hspace*{1.5em}
\includegraphics[height=10em]{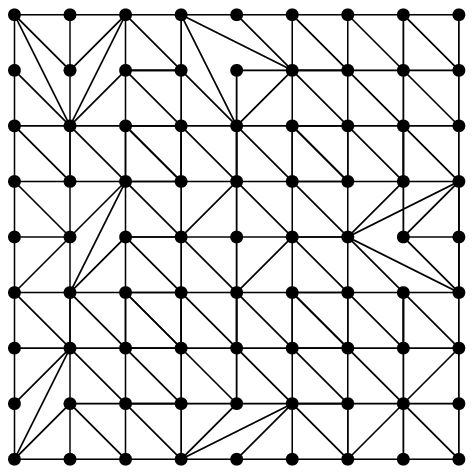}
\caption{Left: a triangulation of the 9x9 integer lattice. Center: a division of the lattice into 9 3x3 sections, as described in the proof of Theorem~\ref{thm:latticetw}. Right: a full triangulation compatible with the division of the lattice.}
\label{fig:lattice}
\end{figure}

It will be useful to define notation for the number of triangulations in this graph:
\begin{definition}
Let $g(n)$ be the number of triangulations of the $n \times n$ integer lattice point set.
\end{definition}

In fact, $g(n)$ is unknown in general, though much progress has been made on upper and lower bounds, including the following result of Kaibel and Ziegler \cite{kzlimit}:
\begin{lemma}
\label{lem:kzlim}
For $n \geq 1$,
$$h(n) = \Theta(2^{cn^2}),$$
for some constant $c$.
\end{lemma}

As discussed in the introduction, in recent years progress has been made (\cite{sinclairlambda, phaseshift, Stauffer2015ALF}) in studying the mixing properties of the natural flip walk on the integer lattice triangulation flip graph. However, this work has focused on biased versions of the flip walk, in which a real parameter $\lambda > 0$ induces a weight function on the triangulations of the lattice, and in which the random walk is biased in favor of triangulations with larger weights. The case of $\lambda = 1$ is the unbiased version of the walk. It is now known \cite{phaseshift} that when $\lambda > 1$, the walk does not mix rapidly, but that rapid mixing does occur for certain values of $\lambda$ smaller than one. However, the question is open for the biased version.

We do not settle the question\textemdash which would equate to showing that the integer lattice triangulation flip graph has expansion at least $1/p(n)$ for some polynomial function $p$, but we do show a weaker result in a similar spirit: that the flip graph has large subgraphs with large expansion. Note that expansion measures the extent to which bottlenecks exist in a graph: large expansion corresponds to a graph that does not have bottlenecks, roughly speaking. Thus, even if bottlenecks exist in $F_n$\textemdash that is, if rapid mixing does \emph{not} occur, i.e. if the expansion is too small\textemdash then there still exist regions of the graph that are not prone to bottlenecks, and thus internally induce rapidly mixing walks. Although far from clear evidence of large expansion in $F_n$ itself, one might hope that if bottlenecks exist, this result may suggest places to look for them.

\subsection{Additional preliminaries: treewidth, separators, and vertex expansion}
\label{sec:tw}
The \emph{treewidth} of a graph $G$ is a different density parameter from expansion. There are many equivalent definitions of treewidth; one of the standard definitions is in terms of a so-called \emph{tree decomposition}. 

Closely related to treewidth are \emph{vertex separators}:
\begin{definition}
\label{def:vsep}
A \emph{vertex separator} for a graph $G$ is a subset $X \subseteq V(G)$ of the vertices of $G$ such that $G\setminus X$ is disconnected. $X$ is a \emph{balanced separator} if $G \setminus X$ consists of two subgraphs, $A$ and $B$, such that no edge exists between $A$ and $B$, and such that $|V(G)|/3 \leq |V(A)| \leq |V(B)| \leq 2|V(G)|/3$. We also say, if $|X| \leq s$ for a given $s \geq 1$, that $X$ is an \emph{$s$-separator}.
\end{definition}

\begin{definition}
\label{def:rvsep}
With respect to an integer $s \geq 1$, a graph $G$ is \emph{recursively $s$-separable} if either $|V(G)| \leq 1$, or $G$ has a balanced $s$-separator $X$ such that the two mutually disconnected subgraphs induced by removing $X$ from $G$ are both recursively $s$-separable.
\end{definition}
The following is known \cite{ericksontw}:
\begin{lemma}
\label{lem:vseptw}
For every $t \geq 1$, every graph with treewidth at most $t$ is recursively $t + 1$-separable.
\end{lemma}

Treewidth in general is of interest in large part because many NP-hard problems become tractable on graphs of bounded treewidth. For a survey of this phenomenon, known as \emph{fixed-parameter tractability}, see \cite{FPTsurvey}. Our interest in treewidth, however, is mainly in its role as a density parameter, in particular for Theorem~\ref{thm:latticetw}.

Treewidth, as a density parameter, is weaker than vertex expansion, in the sense that a high vertex expansion implies a high treewidth, but not vice versa. This is easy to see in the following corollary to Lemma~\ref{lem:vseptw}:
\begin{corollary}
\label{cor:twexp}
If the vertex expansion of a family of graphs $G(N)$ on $N$ vertices is at least $h_v(N),$ then the treewidth $t(N)$ of the family is $\Omega(N\cdot h_v(N))$.
\end{corollary}
\begin{proof}
Suppose $G(N)$ has vertex expansion at least $h_v(N)$. Then every balanced separator $X$ is of size at least
$$|X| \geq h_v(N) \cdot N/3,$$
by the definition of a balanced separator and the definition of vertex expansion.
\end{proof}
\label{sec:latticetw}
In this section we prove Theorem~\ref{thm:latticetw}. 
\thmlatticetw*
\begin{proof}
We will show that $F_n$ has a large induced subgraph with large expansion, which will imply large treewidth. Partition the points of the $n \times n$ grid into $n$ grids of size $\sqrt{n} \times \sqrt{n}$. (If $n$ is not a perfect square, we can take $\sqrt{\lfloor{n}\rfloor}$.) That is, fill in a partial triangulation as follows: let each point in the grid have coordinates $(i, j)$, where $1 \leq i \leq n$ and $1 \leq j \leq n$. Fill in all vertical edges connecting two consecutive points with the same $j$ coordinate whenever $j \equiv 0 \pmod{\sqrt{n}}$ or $j \equiv 1 \pmod{\sqrt{n}}$, and fill in all horizontal edges connecting two consecutive points with the same $i$ coordinate whenever $i \equiv 0 \pmod{\sqrt{n}}$ or $i \equiv 1 \pmod{\sqrt{n}}$.

Fill in also all unit horizontal and unit vertical edges connecting vertices in adjacent $\sqrt{n} \times \sqrt{n}$ sub-grids, and fill in all unit diagonals with negative slope inside the resulting squares. See Figure~\ref{fig:lattice}, center. (The choice of these edges and diagonals to fill in between subgrids is arbitrary, but must be consistent.)

Now consider the subgraph $H_n$ of $F_n$ induced by restricting $V(F_n)$ to the triangulations that extend this partial triangulation. That is, the vertices of $H_n$ are the triangulations that consist of separately triangulating each of the $\sqrt{n} \times \sqrt{n}$ grids. $H_n$ is the Cartesian product of $n$ graphs that are each isomorphic to $F_{\sqrt{n}}$. See Figure~\ref{fig:lattice}, right, for an example of such a triangulation.

Assuming the smallest possible expansion for a graph on $g(\sqrt{n})$ vertices, $F_{\sqrt{n}}$ graph has expansion $\Omega(1/g(\sqrt{n}))$. The degree of~$F_n$ is~$O(n^2)$. Now, by Lemma~\ref{lem:cartexp}, $H_n$ has expansion $\Omega(1/g(\sqrt{n}))$.

Therefore, by Corollary~\ref{cor:twexp}, $H_n$ has treewidth $\Omega(g(n)/(n^2 g(\sqrt{n})))$.

Now, by Lemma~\ref{lem:kzlim}, $H_n$ has treewidth
$$\Omega(g(n)/(n^2 g(\sqrt{n}))) = \Omega(2^{c\cdot n^2 - cn-2\log n}) = \Omega(2^{cn^2(1 - o(1))}) = \Omega(N^{1-o(1)}),$$ proving the theorem.
\end{proof}

\section{Missing proofs from previous sections}
\label{sec:missingproofs}
\subsection{Missing details from $k$-angulation flip walk mixing proofs}
\label{sec:factors}

\lemmatchingscard*
\begin{proof}
$\classtt{}{T}$ and $\classtt{}{T'}$ are Cartesian products of the form $\classtt{}{T} = K_i \Box K_{j+k}$ and $T' = K_{i+j} \Box K_k$, where $|\edgett{}{T}{T'}| = K_i \Box K_j \Box K_k$. Therefore, $|\classtt{}{T}| = C_{i-1}C_{j+k-1}$, $|\classtt{}{T'}| = C_{i+j-1}C_{k-1}$, and $|\edgett{}{T}{T'}| = C_{i-1}C_{j-1}C_{k-1}$. Thus we have
$$\frac{|\classtt{}{T}||\classtt{}{T'}|}{|\edgett{}{T}{T'}|C_{n-1}} \leq \frac{C_{j+k-1}C_{i+j-1}}{C_{j-1}C_{n-1}}.$$
This ratio increases as $j$ increases, for any fixed $i$ (similarly, for any fixed $k$). This is because, if $i$ is fixed, maximizing the ratio is equivalent to maximizing
$$\frac{C_{i+j-1}}{C_{j-1}}.$$
It suffices to show that $C_{i+j-1}/C_{j-1}$ increases whenever $j$ increases by one, i.e.
$$\frac{C_{i+j}/C_{j}}{C_{i+j-1}/C_{j-1}} > 1.$$
I.e., it suffices to show that 
$$\frac{C_{i+j}}{C_{i+j-1}} > \frac{C_j}{C_{j-1}},$$
i.e.
$$\frac{i+j}{i+j+1}\frac{(2(i+j))!(i+j-1)!^2}{(2(i+j-1))!(i+j)!^2} > \frac{j}{j+1}\frac{(2j)!(j-1)!^2}{(2(j-1))!j!^2},$$
i.e.
$$\frac{2(i+j)-1}{i+j+1} > \frac{2j-1}{j+1}.$$
The latter inequality clearly holds for all $i \geq 1$.

Therefore, the ratio in the lemma statement is maximized when $j$ is maximized, i.e. $j = n - 2$ and $i = k = 1$. Thus we have
$$\frac{|\classtt{}{T}||\classtt{}{T'}|}{|\edgett{}{T}{T'}|C_{n-1}} \leq \frac{C_{n-2}C_{n-2}}{C_{n-3}C_{n-1}}.$$ It is immediate from Definition~\ref{def:catalan} that $C_{n-1}/C_{n-2} \geq C_{n-2}/C_{n-3}$, so this ratio is at most one, and the claim follows.
\end{proof}

\lemcartflow*
\begin{proof}
Let $g$ and $h$ be as stated; we construct $f$ as follows:
\begin{enumerate}
\item Within each copy of $H$ in $J$, construct the flow internally according to $h$. Similarly, use $g$ internal to each $G$ copy for each pair of vertices within the $G$ copy.
\item Order the copies of $H$ arbitrarily $H_1, \dots, H_{|V(G)|}$. For each pair of $H$ copies $H_r$ and $H_s$, $s < r$, and for each vertex $h_r \in H_r, h_s \in H_s$, let the flow from $h_r$ to $h_s$ go through (i) the $h$ flow in $H_r$ from $h_r$ to the counterpart vertex $u \in H_r$ of $h_s$, then through (ii) the $g$ flow that goes from $u$ to $h_s$ (in the $G$ copy that $h_s$ and $u$ both belong to).
\end{enumerate}

Part 1 generates no additional flow.
Part 2 generates at most $|V(H)|$ extra flow through each existing $g$ flow, and at most $|V(G)|$ extra flow through each existing $h$ flow. This results in scaling the amount of $g$ flow using any given edge in a $G$ copy by a factor of $|V(H)|$\textemdash while replacing the~$\frac{1}{|V(G)|}$ term in the congestion definition by $\frac{1}{|V(J)|} = \frac{1}{|V(G)||V(H)|}$\textemdash and similarly scaling the amount of $h$ flow using an edge in an $H$ copy by $|V(G)|$. The result follows.
\end{proof}
\bibliography{references}

\end{document}